\documentclass[nonblindrev]{informs3_noinforms}
\usepackage{mathrsfs}
\usepackage{pdfsync}
\usepackage{enumerate}
\usepackage{dsfont}
\usepackage{enumitem}
\usepackage[dvipsnames]{xcolor}
\usepackage{amsmath}
\usepackage{subcaption}

\usepackage{comment}

\usepackage{algorithm, bm}
\usepackage{algpseudocode}
\usepackage{booktabs}

\OneAndAHalfSpacedXI % current default line spacing
\usepackage{endnotes}
\definecolor{chicago-maroon}{RGB}{128,0,0}
\definecolor{northwestern-purple}{RGB}{82,0,99}
\usepackage[nameinlink]{cleveref}
\crefname{assumption}{Assumption}{Assumptions}
\crefname{lemma}{Lemma}{Lemmas}
\crefname{theorem}{Theorem}{Theorems}
\crefname{corollary}{Corollary}{Corollaries}
\crefname{proposition}{Proposition}{Propositions}
\crefname{condition}{Condition}{Conditions}
\crefname{claim}{Claim}{Claims}
\crefname{procedure}{Procedure}{Procedures}
\crefname{algorithm}{Algorithm}{Algorithms}
\crefname{figure}{Figure}{Figures}
\crefname{remark}{Remark}{Remarks}
\crefname{section}{Section}{Sections}
\crefname{procedure}{Procedure}{Procedures}
\crefname{example}{Example}{Examples}
\crefname{definition}{Definition}{Definitions}
\crefname{table}{Table}{Tables}
\crefname{equation}{}{}
\crefname{enumi}{}{}
\crefname{conjecture}{Conjecture}{Conjectures}
\crefname{step}{Step}{Steps}
\crefname{appendix}{Appendix}{Appendices}
\crefname{footnote}{Footnote}{Footnotes}

\numberwithin{subcase}{case}

\numberwithin{subsubcase}{subcase}

\renewcommand{\theequation}{\thesection.\arabic{equation}}

\usepackage{natbib}
 \bibpunct[, ]{(}{)}{,}{a}{}{,}%
 \def\bibfont{\small}%
\TheoremsNumberedThrough     % Preferred (Theorem 1, Lemma 1, Theorem 2)
\ECRepeatTheorems

\EquationsNumberedThrough    % Default: (1), (2), ...
\begin{document}%%%%%%%%%%%%%%%%

% Outcomment only when entries are known. Otherwise leave as is and
%   default values will be used.
%\setcounter{page}{1}
%\VOLUME{00}%
%\NO{0}%
%\MONTH{Xxxxx}% (month or a similar seasonal id)
%\YEAR{0000}% e.g., 2005
%\FIRSTPAGE{000}%
%\LASTPAGE{000}%
%\SHORTYEAR{00}% shortened year (two-digit)
%\ISSUE{0000} %
%\LONGFIRSTPAGE{0001} %
%\DOI{10.1287/xxxx.0000.0000}%

% Author's names for the running heads
% Sample depending on the number of authors;
% \RUNAUTHOR{Jones}
% \RUNAUTHOR{Jones and Wilson}
% \RUNAUTHOR{Jones, Miller, and Wilson}
% \RUNAUTHOR{Jones et al.} % for four or more authors
% Enter authors following the given pattern:
%\RUNAUTHOR{}

% Title or shortened title suitable for running heads. Sample:
% \RUNTITLE{Bundling Information Goods of Decreasing Value}
% Enter the (shortened) title:
\RUNTITLE{Online Order Fulfillment with Replenishment}
% Full title. Sample:
% \TITLE{Bundling Information Goods of Decreasing Value}
% Enter the full title:
\TITLE{Online Order Fulfillment with Replenishment}

% Block of authors and their affiliations starts here:
% NOTE: Authors with same affiliation, if the order of authors allows,
%   should be entered in ONE field, separated by a comma.
%   \EMAIL field can be repeated if more than one author

\ARTICLEAUTHORS{%
\AUTHOR{\bf Zi Ling}
\AFF{Booth School of Business, University of Chicago, Chicago, IL}
\AFF{zling@chicagobooth.edu}
\AUTHOR{\bf Jiashuo Jiang}
\AFF{Department of Industrial Engineering \& Decision Analytics, Hong Kong University of Science and Technology}
\AFF{jsjiang@ust.hk}
\AUTHOR{\bf Linwei Xin}
\AFF{School of Operations Research and Information Engineering, Cornell University, Ithaca, NY}
\AFF{lx267@cornell.edu}
\RUNAUTHOR{Ling, Jiang and Xin}
}

\ABSTRACT{In modern e-commerce and service operations, firms jointly manage
inventory replenishment and real-time order fulfillment under demand
uncertainty. We study their interaction in a single-resource lost-sales system
with heterogeneous rewards, positive lead times, and replenishment
cycles of length \(T\). A central difficulty is that low regret in one
fulfillment cycle need not imply a small profit gap under repeated
replenishment, since fulfillment decisions change inventory and, under
base-stock replenishment, subsequent orders.

We consider base-stock and constant-order replenishment and compare online
fulfillment with a myopic offline benchmark that knows current-cycle demand but
not future demand. We give conditions under which
\(O(T^\alpha)\) one-cycle regret yields an \(O(T^\alpha)\) upper bound on the
expected average profit gap, uniformly over all finite horizons \(N\). For
base stock, the bound remains valid when the online and offline targets are
optimized separately. For constant orders, a uniform quadratic count-loss
bound gives an unrestricted two-sided comparison between separately optimized
systems. For the cycle-restarted Bayes Selector, the finite-horizon bounds are
\(O(1)\). The
optimized long-run constant-order profit gap is also \(O(1)\).

We use these results to compare optimized base stock with greedy fulfillment
against optimized constant order with the Bayes Selector and decompose the
profit difference into replenishment and fulfillment components. A
representative finite-horizon simulation shows that the ranking of the two
policy combinations can change with cycle length. We also extend the transfer
analysis to multiple resources. Under single-resource base-stock replenishment,
a mean-LP quota policy has uniform \(O(\sqrt T)\) regret against a
full-information benchmark.}

\HISTORY{\today{}}
\maketitle

\section{Introduction}
Online retailers allocate limited inventory to sequentially arriving orders while replenishing that inventory with delay. Each order must be accepted or rejected immediately and, in a network, assigned to a fulfillment center with available stock. The resulting allocation affects both current shipping cost and the inventory available before outstanding orders arrive. Fulfillment and replenishment decisions must therefore be coordinated.

These decisions have largely been studied separately. Online fulfillment research, including \cite{Xu09,Acimovic15,Andrews19,DeValve23}, usually treats inventory as a fixed resource, while classical inventory models, beginning with \cite{arrow1951optimal}, optimize replenishment when demand is filled automatically whenever stock is available.

To our knowledge, sequential fulfillment with recurring replenishment and positive lead times has not been studied over multiple cycles. Under base-stock replenishment, sales in one cycle determine the replacement order placed at the beginning of the next cycle, so fulfillment changes both current inventory and the future pipeline. Under constant orders, replenishment does not react to sales, and inventory left by an online rule can persist across cycles. We analyze these two effects separately to determine when one-cycle regret bounds continue to hold under repeated replenishment.

\subsection{Our Contributions}
We study a single-resource system over repeated cycles under base-stock or constant-order replenishment and later extend the analysis to multiple resources. Online fulfillment observes arrivals sequentially. Its main benchmark knows demand within the current cycle but not future cycles. A full-information benchmark is used separately in \cref{sec:look-ahead} under a fixed replenishment policy. Performance is measured by expected profit per cycle, including fulfillment rewards and holding costs.

First, we give conditions under which the multi-cycle profit gap inherits the \(O(T^\alpha)\) upper-bound rate of the one-cycle regret. For single-resource base stock, a direct coupling yields a finite-horizon one-sided bound for the separately optimized myopic-offline and original uncapped online policies. For constant orders, a first-moment condition gives the same bound on a fixed negative-drift domain. A single-cycle quadratic count-loss condition yields an unrestricted two-sided finite-horizon bound. The cycle-restarted Bayes Selector satisfies this condition with \(\alpha=0\) and also admits an \(O(1)\) long-run comparison. The multi-resource results state the reward and allocation structures needed for analogous transfers.

Second, we decompose the comparison between two policy combinations into a replenishment gap and two fulfillment gaps. For the cycle-restarted Bayes Selector, the constant-order fulfillment term is bounded independently of the cycle length, which gives a sharper parameter-dependent condition. In the representative simulation, the policy ranking changes with cycle length.
Third, we formulate exact finite-horizon full-information benchmarks under fixed replenishment policies. A mean-LP quota policy under base-stock replenishment has a uniform \(O(\sqrt T)\) guarantee. We also evaluate a receding-horizon heuristic numerically.

\subsection{Literature Review}
\label{sec:literature_review}

Our paper relates to online fulfillment, lost-sales inventory with lead times, and models that combine fulfillment and replenishment.

\noindent\textbf{Online fulfillment algorithms:} 
Forward-looking fulfillment rules can improve upon myopic allocation \citep{Xu09,Acimovic15}. LP-based allocation and rounding methods give performance guarantees for multi-item orders \citep{Jasin15,ma2023order,amil2025multi,Zhao25}.

Low-regret algorithms have been developed for network revenue management \citep{reiman2008asymptotically,bumpensanti2020re}, the multi-secretary problem \citep{arlotto2019uniformly}, and general online allocation \citep{vera2021bayesian,banerjee2024good}. Related work treats fulfillment with limited delay \citep{Xie24} and unknown demand distributions \citep{li2025infrequentresolvingalgorithmonline}.
With continuous demand and rewards, regret can grow logarithmically \citep{bray2024logarithmic,jiang2025degeneracy}. Constant-regret results also appear in two-way and multi-way matching \citep{wei2023constant,Kerimov2023}; see \cite{balseiro2024survey} for a survey. We study how one-cycle guarantees propagate through replenishment.

Other settings include dynamic pricing \citep{Lei2018}, omnichannel retail \citep{Gao22}, online job selection \citep{Aminian23}, middle-mile transportation \citep{Feng24}, and fulfillment with time windows \citep{ZhouGumusMiao2025}.

\noindent\textbf{Lost-sales inventory models with lead times:}
Lost-sales inventory with lead times dates to \citet{scarf1958inventory}. Because exact optimization is difficult, recent work analyzes simple policies. Base-stock policies are asymptotically optimal as the lost-sales penalty increases \citep{huh2009asymptotic,wei2021deterministic}, including a perishable extension \citep{Bu2022}. Constant-order policies are asymptotically optimal as the lead time increases \citep{goldberg2016asymptotic,xin2016optimality,xin2021understanding}; see \cite{goldberg2021survey} for a survey.
We retain these replenishment structures but allow fulfillment to depend on
customer rewards and sequential arrivals. The fulfillment rule therefore
affects subsequent inventory and, under base stock, the order pipeline.

\noindent\textbf{Joint fulfillment and replenishment:}
\cite{Lim2020} combine inventory allocation and reactive fulfillment, but allocation precedes demand and fulfillment occurs after demand is realized. \cite{Govindarajan21,Arlotto2023,Bai2025} combine one initial placement decision with subsequent fulfillment. We instead study sequential fulfillment with recurring delayed replenishment, so current allocations affect future inventory and, under base stock, pipeline states over multiple cycles.

\subsection{Outline of Paper}
\Cref{sec:model} introduces the model. \Cref{sec:Replenishment} establishes the regret-transfer results, and \cref{sec:reple-vs-fulfill} compares replenishment and fulfillment. \Cref{sec:look-ahead} gives a receding-horizon heuristic, full-information benchmarks, and a guaranteed quota policy. \Cref{sec:extension} treats multiple resources. Proofs are in the electronic companion.

\section{Model} \label{sec:model}

We consider a single-location lost-sales system over \(N\) replenishment cycles of \(T\) periods each. There are \(M\ge2\) customer types, indexed by \([M]=\{1,\ldots,M\}\), with rewards \(0<r_1<\cdots<r_M\). The single-type case reduces to a classical lost-sales model and has no fulfillment comparison. \Cref{sec:extension} treats multiple resources.
Arrivals are i.i.d. across periods and cycles. Type \(j\) arrives with probability \(\lambda_j>0\), and no customer arrives with probability \(\lambda_0=1-\sum_j\lambda_j>0\). Set \(r_0=0\) for the no-arrival event, and write \(\mu=1-\lambda_0\). The boundary \(\lambda_0=0\) has deterministic total demand and does not exhibit the \(\sqrt T\)-fluctuation regime studied below.

Let \(j_t^n\in[M]\cup\{0\}\) be the period-\(t\) arrival in cycle \(n\), and let \(\mathcal H=\{j_t^n:n\in[N],t\in[T]\}\). Define \(D_j^n[t_1,t_2]=\sum_{t=t_1}^{t_2}\mathbb I\{j_t^n=j\}\) and \(D^n=\sum_jD_j^n[1,T]\). Thus the \(D^n\) are i.i.d. \(\mathrm{Binomial}(T,\mu)\).

\textbf{State and inventory.}
Let $L\in\mathbb N$ denote the deterministic lead time for replenishment. 
For algorithm \(\mathrm{alg}\) and replenishment parameter \(z\), let \(I_t^{n,\mathrm{alg}}(z,\mathcal H)\) be on-hand inventory at the beginning of period \(t\), and let \(q^{n+l,\mathrm{alg}}(z,\mathcal H)\) arrive at the beginning of cycle \(n+l\). Thus the order placed in cycle \(n\) is \(q^{n+L,\mathrm{alg}}\). The pre-order pipeline is \(\bm q^{n,\mathrm{alg}}=(q^{n+1,\mathrm{alg}},\ldots,q^{n+L-1,\mathrm{alg}})\), which is empty when \(L=1\). Inventory \(I_1^{n,\mathrm{alg}}\) is measured after the order due in cycle \(n\) arrives. A holding cost \(h>0\) is charged on \(I_{T+1}^{n,\mathrm{alg}}\).

\textbf{Replenishment policies.}
At the beginning of each cycle, the state determines an order that arrives \(L\) cycles later. We study two standard policies.

\begin{itemize}
    \item \textbf{Base-stock policy} (\(z=S\)). For target inventory position \(S\),
    \begin{equation}\label{eq:base-replenish}
    q^{n+L,\text{alg}}(S,\mathcal{H})
    = S - I_{1}^{n,\text{alg}}(S,\mathcal{H}) - \sum_{l=1}^{L-1} q^{n+l,\text{alg}}(S,\mathcal{H}),
    \end{equation}
    and
    \[
    I_{1}^{n+1,\text{alg}}(S,\mathcal{H})=I_{T+1}^{n,\text{alg}}(S,\mathcal{H})+q^{n+1,\text{alg}}(S,\mathcal{H}),
    \]
    followed by the usual pipeline shift. For finite-horizon optimization, write \(P=L+1\) and \(S=Ps+t\), where \(s\in\mathbb Z_+\) and \(0\le t<P\), and initialize
    \[
    \bigl(I_1^{1,\mathrm{alg}},q^{2,\mathrm{alg}},\ldots,
       q^{L+1,\mathrm{alg}}\bigr)
    =(a_0,a_1,\ldots,a_L),\qquad
    a_k=
    \begin{cases}
    s, & 0\leq k<L,\\
    s+t, & k=L.
    \end{cases}
    \]
    These entries sum to \(S\), and the same state is used for every fulfillment rule. The fixed-target transfer holds from any common feasible post-order state with this sum; the displayed initialization is used to locate the finite-horizon optimizer uniformly in \(N\).

    \item \textbf{Constant-order policy} (\(z=c\)). The delivery in every cycle is \(c\).
    Thus \(I_1^{n+1,\mathrm{alg}}=I_{T+1}^{n,\mathrm{alg}}+c\). We use zero carryover before cycle 1, giving \(I_1^{1,\mathrm{alg}}=c\) for every fulfillment algorithm.
\end{itemize}
Algorithms compared at a common replenishment parameter use the same deterministic initialization. The same rule is applied to each candidate parameter in finite-horizon optimization.

\textbf{Fulfillment algorithms.}
Let \(x_t^{n,\mathrm{alg}}\in\{0,1\}\) denote acceptance, with \(x_t^{n,\mathrm{alg}}\le\mathbb I\{j_t^n\ne0\}\), \(x_t^{n,\mathrm{alg}}\le I_t^{n,\mathrm{alg}}\), and \(I_{t+1}^{n,\mathrm{alg}}=I_t^{n,\mathrm{alg}}-x_t^{n,\mathrm{alg}}\). The cycle profit is
\begin{equation}\label{eq:cycle-profit}
V^{n,\text{alg}}(T,I_1^{n,\text{alg}}(z,\mathcal{H})) \triangleq\ \sum_{t=1}^T r_{j_t^n}   x_t^{n,\text{alg}} - h   I_{T+1}^{n,\text{alg}}(z,\mathcal{H}), \text{ for } z \in \{S,c\}.
\end{equation}

We now define our performance metrics. 
Expected average profit under base-stock replenishment is
\begin{equation}\label{eq:pi-b}
\pi_{\mathbf b}^{\text{alg}}(T,S,N)\triangleq\ \frac{1}{N}   \mathbb E_{\mathcal{H}, \text{alg}}\left[\sum_{n=1}^N V^{n,\text{alg}}(T,I_1^{n,\text{alg}}(S,\mathcal{H}))\right],
\end{equation}
where the expectation includes arrivals and algorithmic randomization.
Under a constant order \(c\), define
\begin{equation}\label{eq:pi-c}
\pi_{\mathbf c}^{\text{alg}}(T,c,N)\triangleq \frac{1}{N}   \mathbb E_{\mathcal{H}, \text{alg}} \left[\sum_{n=1}^N V^{n,\text{alg}}(T,I_1^{n,\text{alg}}(c,\mathcal{H}))\right].
\end{equation}
Whenever the displayed limits exist, we define
$\pi_{\mathbf b}^{\text{alg}}(T,S,\infty) \triangleq \lim_{N\to\infty}\pi_{\mathbf b}^{\text{alg}}(T,S,N)$ and $\pi_{\mathbf c}^{\text{alg}}(T,c,\infty) \triangleq \lim_{N\to\infty}\pi_{\mathbf c}^{\text{alg}}(T,c,N)$ as the long-run average profits.

We compare fulfillment and replenishment improvements through these objectives.

\subsection{Offline Fulfillment Algorithm}
Partial foresight reveals \(\{D_j^n[1,T]\}_{j\in[M]}\) before fulfillment in cycle \(n\), but not demand in later cycles.
Unless stated otherwise, offline refers to the myopic allocation that serves types in decreasing reward order until inventory is exhausted, without using the order pipeline. We use the superscript \(\mathrm{off}\). This benchmark need not upper bound multi-cycle profit because it may serve low-reward demand instead of preserving inventory for later cycles. \Cref{sec:look-ahead} therefore introduces a separate full-information benchmark.

\subsection{Online Fulfillment Algorithm}  
An online algorithm observes past arrivals and the distribution of future
arrivals. We use the single-resource Fluid Bayes Selector analyzed in
Section~4.1 and the proof of Theorem~2 of \citet{vera2021bayesian}.
Throughout the paper, we refer to this policy simply as the Bayes Selector.
With our reward indexing, it accepts type \(M\) whenever inventory is available
and accepts a type \(j<M\) when its fluid allocation is at least
\(\Lambda_j[t,T]/2\), where
\(\Lambda_j[t,T]=(T-t+1)\lambda_j\) is the expected remaining demand. We use
the superscript \(\mathrm{on}\).

For one cycle without replenishment, let \(V^{\mathrm{off}}(T,I,\mathcal H)\) and \(V^{\mathrm{on}}(T,I,\mathcal H)\) be total rewards, excluding holding cost, from common inventory \(I\). The regret is
\[
\mathrm{REG}\triangleq \mathbb{E}_{\mathcal{H}, \text{on}} \left[V^{\text{off}}(T,I,\mathcal{H})-V^{\text{on}}(T,I,\mathcal{H})\right].
\]
We use the following conditional version in the multi-cycle model.

\begin{assumption}[Uniform single-cycle regret]
\label{ass:online-algorithm}
An online fulfillment algorithm makes sequential accept/reject decisions without access to future realizations. For every initial inventory level \(I\), and conditional on any information available before the cycle starts, its expected reward regret against the myopic offline allocation satisfies
\[
\mathbb E\!\left[V^{\mathrm{off}}(T,I,\mathcal H)-V^{\mathrm{on}}(T,I,\mathcal H)\,\middle|\,\mathcal F_{\mathrm{pre}}\right]
\le C_1T^\alpha,
\]
for constants \(C_1>0\) and \(\alpha\in[0,1]\) independent of \(I,T\), and of the pre-cycle history \(\mathcal F_{\mathrm{pre}}\). Demand in the new cycle is independent of \(\mathcal F_{\mathrm{pre}}\). When the cycle is indexed by \(n\), we write \(\mathcal F_{n-1}\) for this pre-cycle sigma-field.
\end{assumption}

The Bayes Selector defined above has \(\alpha=0\) with
\(C_1=r_M\sum_{j\ne M}2/\lambda_j\leq
2(M-1)r_M/\lambda_{\min}\), after reversing their reward indexing, where \(\lambda_{\min}=\min_j\lambda_j\). We may restrict \(\alpha\le1\) because one-cycle regret is at most linear in \(T\).

These online and myopic-offline rules use current inventory and within-cycle demand information, but not the pipeline. \Cref{sec:look-ahead} incorporates the pipeline into fulfillment.

For \(\mathrm{alg}\in\{\mathrm{on},\mathrm{off}\}\), let
\[
f^{n,\mathrm{alg}}(I,\mathcal H)
\triangleq \sum_{t=1}^T x_t^{n,\mathrm{alg}}
\]
be the number served in cycle \(n\) from inventory \(I\). The replenishment parameter is suppressed because it does not affect within-cycle decisions once \(I\) is fixed.

\begin{lemma}\label{lem:inv-diff}
Under \cref{ass:online-algorithm}, for any cycle \(n\), if the start-of-cycle
inventories coincide at \(I\), then, for every demand path \(\mathcal H\) and
every realization of the online algorithm's internal randomization,
\[
\epsilon^n(I,\mathcal H)
\triangleq I_{T+1}^{n,\mathrm{on}}(I,\mathcal H)-I_{T+1}^{n,\mathrm{off}}(I,\mathcal H)
=f^{n,\mathrm{off}}(I,\mathcal H)-f^{n,\mathrm{on}}(I,\mathcal H)\ge0.
\]
Moreover, the expected inventory gap is bounded:
\[
\mathbb E_{\mathcal H,\mathrm{on}}\!\left[
\epsilon^n(I,\mathcal H)\,\middle|\,\mathcal F_{\mathrm{pre}}
\right]
\le C_2T^\alpha,
\]
where one may take \(C_2=C_1/r_1\), independently of \(I,T\), and the pre-cycle history.
\end{lemma}

\section{Replenishment Policies} \label{sec:Replenishment}
To support the comparison in \cref{sec:reple-vs-fulfill}, we first quantify
how a single-cycle fulfillment loss affects average profit under repeated
replenishment. We consider the base-stock and constant-order policies. The
scope of the transfer depends on the replenishment policy and the one-cycle
conditions stated below.
We are interested in comparing the performance of the online and offline fulfillment algorithms under the same replenishment policy; specifically, we aim to bound
\[
\pi_{\mathbf b}^{\text{off}}(T,S^{\text{off}},N)-\pi_{\mathbf b}^{\text{on}}(T,S^{\text{on}},N)
\quad\text{and}\quad
\pi_{\mathbf c}^{\text{off}}(T,c^{\text{off}},N)-\pi_{\mathbf c}^{\text{on}}(T,c^{\text{on}},N),
\]
where \(S^{\text{on}}\) and \(S^{\text{off}}\) denote the optimal base-stock levels for the online and offline fulfillment algorithms, and \(c^{\text{on}}\) and \(c^{\text{off}}\) denote the corresponding optimal constant-order quantities.
In every finite-horizon statement, these symbols denote maximizers of the
displayed \((T,N)\)-objective under the common initialization convention.
Long-run maximizing parameters carry an \(\infty\) subscript, while optimized
long-run values carry the superscript \(*\).

\subsection{Base-Stock Policy}
\label{subsec:base-replenishment}

Under a base-stock policy, the units fulfilled in cycle \(n\) determine an
equal replacement order at the beginning of cycle \(n+1\). This order arrives
at the beginning of cycle \(n+L+1\). This mechanism
limits the persistence of inventory differences across cycles. However,
\cref{lem:inv-diff} only compares the two fulfillment algorithms within a
single cycle when they start from the same inventory. In later cycles, their
on-hand inventories and outstanding orders generally differ. We first
establish a fixed-target comparison for the original online algorithm.

Write \(\pi_{\mathbf b}^{\mathrm{uncap}}\) for the profit of the original
online algorithm when it is run on all available inventory. This is the online
algorithm denoted by \(\pi_{\mathbf b}^{\mathrm{on}}\) elsewhere in the paper.

\begin{proposition}[Fixed-target comparison under base-stock replenishment]
\label{prop:uncapped-direct-bound}
Suppose \(L\geq1\) and \cref{ass:online-algorithm} holds. For every
\(K_S<\infty\), there are constants \(C_{K_S},T_{K_S}<\infty\), independent
of \(S,T,N\), such that, for every \(T\geq T_{K_S}\), \(N\geq1\), and
\(S\geq(L+1)\mu T-K_S\sqrt T\), the original online and myopic-offline
algorithms, started from a common feasible augmented post-order state with
inventory position \(S\), satisfy
\begin{equation}\label{eq:uncapped-fixed-S-transfer}
\pi_{\mathbf b}^{\mathrm{off}}(T,S,N)
-\pi_{\mathbf b}^{\mathrm{uncap}}(T,S,N)
\leq C_{K_S}T^\alpha.
\end{equation}
\end{proposition}

For a post-order state \(z=(I,q_1,\ldots,q_L)\), define
\[
C_k(z)=I+\sum_{i=1}^kq_i,
\qquad k=0,\ldots,L.
\]
The proof uses the span of the differences in these readiness totals.
Work-conserving transitions are nonexpansive in this metric. With probability
bounded away from zero, a \(2L\)-cycle block clears the initial pipelines and
coalesces the two work-conserving states. The count loss in
\cref{lem:inv-diff} is an \(O(T^\alpha)\) perturbation, so geometric
contraction gives a bound uniform in \(N\). Details are in the electronic
companion.

We now apply this fixed-target comparison after locating the offline optimizer.
The original online algorithm is run on all available inventory.

\begin{theorem}[Finite-horizon uncapped base-stock transfer]
\label{thm:reg-total-base}
Suppose \(L\geq1\), \(0<\mu=1-\lambda_0<1\), and
\cref{ass:online-algorithm} holds. Under the post-order initialization specified
in \cref{sec:model}, the offline and online finite-horizon objectives attain
their maxima for all sufficiently large \(T\). Moreover, there are constants
\(C_{\mathbf b},T_0<\infty\), independent of \(T,N\), such that, for
\(T\geq T_0\) and \(N\geq1\),
\begin{equation}\label{eq:uncapped-finite-optimized}
\pi_{\mathbf b}^{\mathrm{off}}(T,S^{\mathrm{off}},N)
-\pi_{\mathbf b}^{\mathrm{on}}(T,S^{\mathrm{on}},N)
\le C_{\mathbf b}T^\alpha.
\end{equation}
\end{theorem}

The comparison is one-sided because myopic-offline fulfillment need not
maximize profit over multiple cycles.

The fixed-target result applies because the offline optimizer satisfies
\[
S^{\mathrm{off}}(T,N)\geq(L+1)\mu T-K\sqrt T
\]
uniformly in \(N\). This follows by comparing the optimizer with
\(S_0=(L+1)\lceil\mu T\rceil\), whose lost demand and ending inventory are
\(O(\sqrt T)\). Apply \cref{prop:uncapped-direct-bound} at the offline
optimizer and then use online optimality. The initialization is used only to
locate the optimizer.

Whenever the fixed-target long-run averages exist, write
\[
\pi_{\mathbf b}^{a,*}(T,\infty)
\triangleq\sup_{S\in\mathbb Z_+}\pi_{\mathbf b}^{a}(T,S,\infty),
\qquad a\in\{\mathrm{off},\mathrm{uncap}\}.
\]

\begin{corollary}[Cycle-restarted Bayes Selector under base-stock replenishment]
\label{cor:fbs-base-ordinary}
For the original cycle-restarted Bayes Selector defined in \cref{sec:model}, the fixed-target
long-run averages and the optimized long-run values exist. There is a constant
\(C_{\mathrm{Bayes}}<\infty\), independent of \(T\), such that, for all
sufficiently large \(T\),
\[
\pi_{\mathbf b}^{\mathrm{off},*}(T,\infty)
-\pi_{\mathbf b}^{\mathrm{uncap},*}(T,\infty)
\leq C_{\mathrm{Bayes}}.
\]
\end{corollary}

The Bayes Selector has \(\alpha=0\). For fixed \(S\), consecutive zero-demand cycles
give a regenerative state in the finite post-order state space, while holding
cost makes the optimized target finite. Details are in the electronic
companion.

The e-companion also studies a shadow-capped implementation and gives an
explicit transfer coefficient. This implementation is different from the
original online algorithm and is not used in the proof of
\cref{thm:reg-total-base}.

\subsection{Constant-Order Policy}
\label{subsec:constant-replenishment}

Under constant orders, missed sales do not reduce subsequent replenishment, so
inventory differences may persist across cycles. We control this carryover
through a perturbed Lindley recursion.

\textbf{Setup for this subsection.}
Recall that \( f^{n,\text{on}}(I) \) and \( f^{n,\text{off}}(I) \) denote
the number of fulfilled demands in cycle \(n\), starting from inventory
\(I\), with the demand-path argument suppressed.
Write \(\mu=1-\lambda_0\).

Fix \(\delta\in(0,\mu)\) and define the stability-restricted optimization
domain
\begin{equation}\label{eq:constant-stable-domain}
\mathcal C_\delta(T)
\triangleq
\{c\in\mathbb Z_+:c\leq\lfloor(\mu-\delta)T\rfloor\}.
\end{equation}
For \(a\in\{\mathrm{on},\mathrm{off}\}\), let
\(c^{a,\delta}\) maximize
\(\pi_{\mathbf c}^{a}(T,c,N)\) over \(\mathcal C_\delta(T)\), and write
\(\pi_{\mathbf c,\delta}^{a}(T,N)
=\pi_{\mathbf c}^{a}(T,c^{a,\delta},N)\).
We use \(\mathcal C_\delta(T)\) only for the result based on the expected
single-cycle loss. As shown in \cref{lem:constant-off-optimizer-slack}, the
myopic-offline optimizer is \(\Theta(\sqrt T)\) below mean demand and therefore
need not belong to this restricted domain. We consider the optimized regime
separately after introducing a stronger single-cycle condition.

All finite-horizon constant-order comparisons below use the common
zero-carryover initialization. For \(a\in\{\mathrm{on},\mathrm{off}\}\), let
\(c^a=c^a(T,N)\) denote an optimizer over \(\mathbb Z_+\); the notation
\(c^{a,\delta}\) is reserved for the restricted optimizer above. For the
long-run comparison, write
\[
\pi_{\mathbf c}^{\mathrm{off},*}(T,\infty)
\triangleq\max_{c\in\mathbb Z_+:\,c<\mu T}
\pi_{\mathbf c}^{\mathrm{off}}(T,c,\infty),
\qquad
\underline\pi_{\mathbf c}^{\mathrm{on},*}(T)
\triangleq\sup_{c\in\mathbb Z_+:\,c<\mu T}
\liminf_{N\to\infty}\pi_{\mathbf c}^{\mathrm{on}}(T,c,N).
\]
Let \(c_\infty^{\mathrm{off}}\) be the smallest maximizer in the first
definition.

\paragraph{Coupled inventory dynamics and auxiliary bounds.}
The electronic companion first proves that, from a common initial state and
under a common order quantity, the online inventory is pathwise no smaller
than the myopic-offline inventory; see
\cref{lem:constant-off-more}. Write
\[
X_n=I_{T+1}^{n,\mathrm{on}},\qquad
Y_n=I_{T+1}^{n,\mathrm{off}},\qquad X_0=Y_0=0,
\]
and define the same-start count loss at the online start inventory by
\[
\epsilon^n
\triangleq
f^{n,\mathrm{off}}(I_1^{n,\mathrm{on}})
-f^{n,\mathrm{on}}(I_1^{n,\mathrm{on}})\in[0,T].
\]
The two inventories then satisfy the exact coupled recursions
\begin{equation}\label{eq:constant-coupled-recursions}
X_n=(X_{n-1}+c-D^n)^++\epsilon^n,
\qquad
Y_n=(Y_{n-1}+c-D^n)^+.
\end{equation}
In particular, \cref{lem:constant-off-more} gives \(X_n\geq Y_n\)
pathwise. The offline reward is nondecreasing in its starting inventory, so
this ordering reduces the reward comparison in each cycle to the same-start
bound in \cref{ass:online-algorithm}. The remaining task is to control
\(X_n-Y_n\), which determines the holding-cost difference.

The perturbation \(\epsilon^n\) depends on the pre-cycle history, so the
i.i.d. random-walk identity for the offline recursion cannot be applied to
the online process. By \cref{lem:inv-diff},
\[
\mathbb E[\epsilon^n\mid\mathcal F_{n-1}]\leq C_2T^\alpha.
\]
The first moment is sufficient under fixed-fraction demand slack. Near an
optimized order, where the slack is of order \(\sqrt T\), it permits rare
losses of order \(T\) whose inventory effect may persist. The formal example
in \cref{prop:first-moment-insufficient} motivates the following one-cycle
condition for the unrestricted comparison.

\begin{assumption}[Uniform single-cycle quadratic count loss]
\label{ass:single-cycle-sm2}
For the same exponent \(\alpha\) as in \cref{ass:online-algorithm}, suppose
\(0\leq\alpha<1\) and there is a constant \(K_2<\infty\), independent of
\(I,T,n\), and the pre-cycle history, such that
\[
\mathbb E\!\left[
 \epsilon^n(I,\mathcal H)^2
 \middle|\mathcal F_{\mathrm{pre}}
\right]\leq K_2T^{2\alpha}.
\]
Demand in the new cycle is independent of \(\mathcal F_{\mathrm{pre}}\).
\end{assumption}

A quadratic reward-regret bound of the form
\[
\mathbb E\!\left[
 \operatorname{Reg}_T(I,\mathcal H)^2
 \middle|\mathcal F_{\mathrm{pre}}
\right]\leq K_RT^{2\alpha}
\]
is sufficient because the augmentation argument in \cref{lem:inv-diff} gives
\(r_1\epsilon^n\leq\operatorname{Reg}_T\). The conditioning in
\cref{ass:single-cycle-sm2} allows the entering inventory to depend on the
past. The exponential one-cycle
regret bound of \citet[Section~4.3, Proposition~3]{vera2021bayesian} implies
the stronger conditional bound in \cref{lem:fbs-single-cycle-et}; hence the
cycle-restarted Bayes Selector satisfies both one-cycle assumptions with
\(\alpha=0\).

The remaining ingredients are stated and proved in the electronic companion.
Under fixed-fraction slack, \cref{lem:Inv-on-bound} gives a uniform
\(O(T^\alpha)\) bound on online carryover. Under
\cref{ass:single-cycle-sm2}, \cref{lem:constant-sqrt-transfer-average}
controls the horizon-average inventory gap by \(O(T^\alpha)\) whenever
\(c\leq\mu T-\kappa\sqrt T\). Finally,
\cref{lem:constant-off-optimizer-slack} places the stable myopic-offline
optimizer at \(\Theta(\sqrt T)\) demand slack. The first result uses a
conditional exponential Lyapunov argument, the second combines reset cycles
with a quadratic bound on large losses, and the third follows from adjacent
workload differences and binomial tail estimates.

\begin{theorem}[Constant-order transfer]\label{thm:reg-total-constant}
Parts~(a), (b), and (d) assume \cref{ass:online-algorithm}, whereas
part~(c) does not.
\begin{enumerate}[label=(\alph*)]
\item If \(\alpha<1\), then, for every fixed \(\delta\in(0,\mu)\), all
sufficiently large \(T\), and every \(N\geq1\),
\begin{equation}\label{ineq:reg-total-constant-case-a}
\pi_{\mathbf c,\delta}^{\mathrm{off}}(T,N)
-\pi_{\mathbf c,\delta}^{\mathrm{on}}(T,N)
\leq (C_1+hK_{\mathrm c})T^\alpha
=\mathcal O(T^\alpha).
\end{equation}
\item Suppose in addition that \cref{ass:single-cycle-sm2} holds and
\(\alpha<1/2\). Then, for all sufficiently large \(T\),
\begin{equation}\label{eq:constant-unrestricted-et}
\pi_{\mathbf c}^{\mathrm{off},*}(T,\infty)
-\underline\pi_{\mathbf c}^{\mathrm{on},*}(T)
\leq K_*T^\alpha
=\mathcal O(T^\alpha),
\end{equation}
where \(K_*<\infty\) is independent of \(T\).

\item For every finite \(N\geq1\), the unrestricted finite-horizon envelope
\begin{equation}\label{eq:constant-unrestricted-envelope}
\left|
\pi_\mathbf c^{\mathrm{off}}(T,c^{\mathrm{off}},N)
-\pi_\mathbf c^{\mathrm{on}}(T,c^{\mathrm{on}},N)
\right|
\leq r_M\mu T
\end{equation}
always holds. Moreover, whenever the corresponding optimized long-run values
exist and are finite, each lies in \([0,r_M\mu T]\), and hence their absolute
difference is also at most \(r_M\mu T\).

\item If, in addition, \cref{ass:single-cycle-sm2} holds with
\(0\leq\alpha<1\), there are constants \(K,T_0<\infty\), independent of
\(N\), such that, for every \(T\geq T_0\) and finite \(N\geq1\),
\begin{equation}\label{eq:constant-finite-unrestricted-et}
\left|
\pi_{\mathbf c}^{\mathrm{off}}(T,c^{\mathrm{off}},N)
-\pi_{\mathbf c}^{\mathrm{on}}(T,c^{\mathrm{on}},N)
\right|\leq K T^\alpha.
\end{equation}
\end{enumerate}
\end{theorem}

\paragraph{Immediate consequences.}
For the Bayes Selector, part~(b) gives an \(O(1)\) one-sided
optimized-liminf gap over all stable orders, without a fixed-fraction slack
restriction. Part~(d) gives an unrestricted \(O(1)\) two-sided comparison
between the separately optimized finite-horizon objectives, uniformly over
all finite \(N\). The stronger two-sided result for the ordinary long-run
limits is given in \cref{cor:fbs-constant-ordinary}.

If an online algorithm satisfying the assumptions in part (b) has a finite
long-run average at \(c_\infty^{\mathrm{off}}\), the same bound holds with
\(\underline\pi_{\mathbf c}^{\mathrm{on},*}(T)\) replaced by
\[
\sup_{\substack{c\in\mathbb Z_+,\ c<\mu T:\\
\pi_{\mathbf c}^{\mathrm{on}}(T,c,\infty)
\text{ exists and is finite}}}
\pi_{\mathbf c}^{\mathrm{on}}(T,c,\infty).
\]

\paragraph{Proof idea.}
For a common order, \eqref{eq:constant-coupled-recursions} and
\(X_n\geq Y_n\) reduce the reward comparison to the same-start one-cycle
regret. Part~(a) then follows from the linear-slack carryover bound. For
part~(b), the offline optimizer has \(\Theta(\sqrt T)\) slack, where the
average near-critical transfer controls the holding-cost difference.
Part~(c) uses the interval \([0,r_M\mu T]\) containing both optimized
profits.

For part~(d), a common near-critical order provides a benchmark for both
finite-horizon objectives. If \(\alpha<1/2\), sufficiently long horizons
exclude orders too close to mean demand, while the finitely many shorter
horizons are controlled directly. If \(\alpha\geq1/2\), a quadratic drift
bound at an order with \(T^\alpha\) slack places both optimized profits within
\(O(T^\alpha)\) of the reward from serving all demand. The electronic
companion gives the complete proof.

The additional condition cannot be omitted from the unrestricted comparison.
The construction in \cref{prop:first-moment-insufficient} has one-cycle
expected regret \(O(T^\alpha)\), but its unrestricted optimized gap is
\(\Omega(T^{(1+\alpha)/2})\), and \(\Omega(\sqrt T)\) when \(\alpha=0\).
Thus \cref{ass:online-algorithm} alone is insufficient.

For the long-run specialization below, restart the Bayes Selector defined in
\cref{sec:model} in every cycle from its actual on-hand inventory, with fresh
independent arrivals and a fixed tie-breaking rule.
Whenever they are attained, let \(c_T^a\),
\(a\in\{\mathrm{Bayes},\mathrm{off}\}\), denote the corresponding
ordinary long-run optimizer and write
\[
\pi_{\mathbf c}^{a,*}(T,\infty)
\triangleq \pi_{\mathbf c}^a(T,c_T^a,\infty).
\]

\begin{corollary}[Long-run constant-order transfer for the Bayes Selector]
\label{cor:fbs-constant-ordinary}
Under the fixed single-resource primitives of \cref{sec:model}, suppose
\(0<\lambda_0<1\) and \(h>0\). For every integer \(c<\mu T\), both
long-run average profits exist and are finite; for
\(c\geq\mu T\), the average holding cost diverges. Hence the unrestricted
long-run optimizers \(c_T^{\mathrm{Bayes}}\) and \(c_T^{\mathrm{off}}\) exist.
There are constants \(0<\kappa<K<\infty\) and \(K_{\mathrm{gap}}<\infty\),
independent of \(T\), such that for sufficiently large \(T\),
\[
\kappa\sqrt T\le\mu T-c_T^a\le K\sqrt T
\quad(a\in\{\mathrm{Bayes},\mathrm{off}\}),\qquad
\left|\pi_{\mathbf c}^{\mathrm{off},*}(T,\infty)
-\pi_{\mathbf c}^{\mathrm{Bayes},*}(T,\infty)\right|
\le K_{\mathrm{gap}}.
\]
\end{corollary}

For stable orders, a Foster drift argument gives positive recurrence of the
cycle-boundary Bayes Selector chain. The single-cycle property in
\cref{lem:fbs-single-cycle-et} and the pointwise near-critical transfer
control the inventory gap uniformly over cycles. Orders at or above mean
demand have diverging holding cost, and both optimizers have
\(\Theta(\sqrt T)\) slack, which yields the two-sided \(O(1)\) bound.

For moderate or small \(T\), the asymptotic transfer bounds alone may not
determine which operational decision has the greater effect. We next compare
the two policy combinations at fixed finite cycle lengths \(T\), using
optimized long-run average profits.

\section{Replenishment vs. Online Fulfillment} 
\label{sec:reple-vs-fulfill}

We compare optimized base-stock replenishment with greedy fulfillment against optimized constant-order replenishment with a specified online algorithm. The resulting comparison bounds depend on the model parameters. The finite-horizon simulation in \cref{fig:sim-comparison} compares the cycle-restarted Bayes Selector defined in \cref{sec:model} and greedy fulfillment under both policies.

\begin{figure}[htbp]
  \centering
  \includegraphics[width=0.7\textwidth]{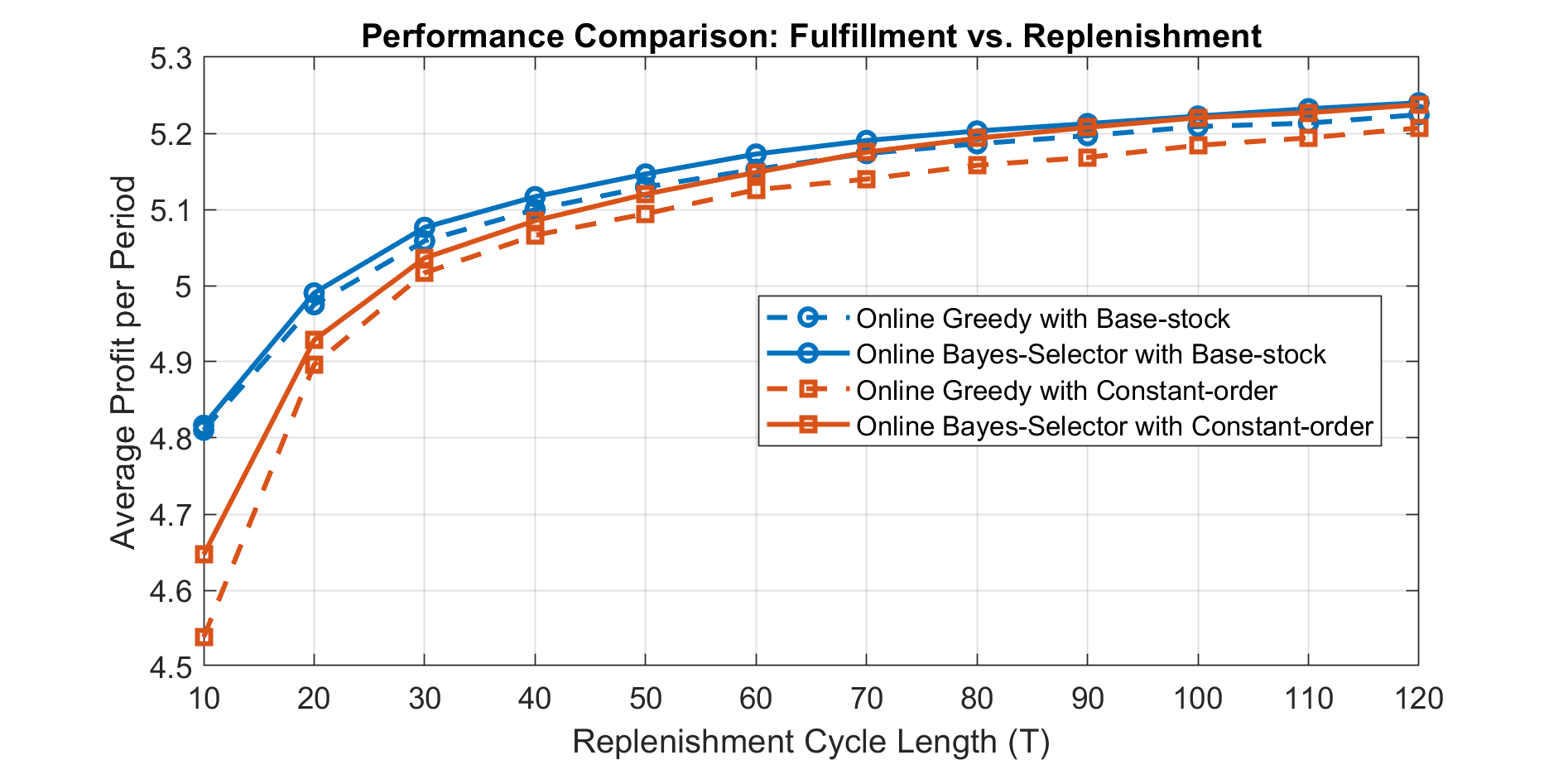}
  \caption{Average profit per period versus cycle length \(T\) under four policy combinations (Bayes Selector vs.\ greedy fulfillment) \(\times\)
  (base-stock vs.\ constant-order). Parameters: \(L=2\), \(N=1{,}000\), \(h=2\),
  \([r_1,r_2,r_3]=[5,8,10]\), \([\lambda_0,\lambda_1,\lambda_2,\lambda_3]=[0.3,0.2,0.3,0.2]\), and \(K=100\) sample paths.}
  \label{fig:sim-comparison}
\end{figure}

In this instance, the Bayes Selector performs better for large \(T\), whereas base stock with greedy fulfillment exceeds constant order with the Bayes Selector for small \(T\). This is a simulation over \(N=1{,}000\) cycles, not a direct computation of the long-run quantities defined below.

Assume that the long-run averages below exist. For \(a\in\{\mathrm{off},\mathrm{greedy},\mathrm{on}\}\), as applicable, write
\[
\pi_\mathbf b^a\triangleq
\sup_{S\in\mathbb Z_+}\pi_\mathbf b^a(T,S,\infty),
\qquad
\pi_\mathbf c^a\triangleq
\sup_{c\in\mathbb Z_+:\,\pi_\mathbf c^a(T,c,\infty)
\text{ exists and is finite}}
\pi_\mathbf c^a(T,c,\infty).
\]
Define the signed differences
\[
\pi_\mathbf{b}^{\text{off}} - \pi_\mathbf{c}^{\text{off}} \triangleq \eta_0, \qquad
\pi_\mathbf{b}^{\text{off}} - \pi_\mathbf{b}^{\text{greedy}} \triangleq \eta_1, \qquad
\pi_\mathbf{c}^{\text{off}} - \pi_\mathbf{c}^{\text{on}} \triangleq \eta_2.
\]
Here greedy accepts every arrival while inventory is available, whereas the
general online rule may ration inventory using the remaining-demand
distribution. It follows that
\begin{equation}\label{eq:ineq}
\pi_\mathbf{b}^{\text{greedy}} - \pi_\mathbf{c}^{\text{on}} = \eta_0 - \eta_1 + \eta_2.
\end{equation}
A nonnegative value in \eqref{eq:ineq} favors base stock with greedy fulfillment. We bound its three components below.

\subsection{Base-Stock vs. Constant-Order under Offline Fulfillment}
\label{subsec:base-constant-off}

We analyze \(\eta_0=\pi_\mathbf b^{\mathrm{off}}-\pi_\mathbf c^{\mathrm{off}}\) through long-run costs with a uniform lost-sales or backorder penalty \(p\). Backorder comparisons bound the base-stock cost, and Spitzer's identity gives the constant-order cost.

Let \(C_\mathbf b^{\mathscr L,S}(h,p)\) and \(C_\mathbf b^{\mathscr B,S}(h,p)\) be the base-stock costs in the lost-sales and backorder systems, with minimizers \(S^{\mathscr L*}(h,p)\) and \(S^{\mathscr B*}(h,p)\). Let \(C_\mathbf c^c(h,p)\) be the constant-order cost and let \(c^*(h,p)\) minimize it over integers \(0\le c<T(1-\lambda_0)\). Here \(h,p>0\), and dependence on fixed \(T,L\) is suppressed.
The minimizers are attained. For fixed \(S\), the lost-sales state space is finite, and \(L\) zero-demand cycles reach the aperiodic state \((S,0,\ldots,0)\). Moreover,
\[
C_\mathbf b^{\mathscr{L},S}(h,p)
\geq h[S-(L+1)T]^+\longrightarrow\infty
\qquad(S\to\infty),
\]
so the lost-sales minimizer is finite. The backorder minimizers are the newsvendor quantiles below.

We use the lost-sales and backorder bridge in \cref{lem:lost-back-bridges}.

Let \(P_k\) be the probability that demand over the \((L+1)T\)-period pipeline window equals \(k\):
\[P_k \triangleq \Pr  \left(\sum_{n=1}^{L+1} D^n = k\right) = \binom{(L+1)T}{k}(1-\lambda_0)^k \lambda_0^{(L+1)T-k}.\] 
The bridge gives the following computable bounds.

\noindent\textbf{Upper bound:}
\[
\bar{a}_1\left(T,h,p,L\right) \triangleq 
h \sum_{k=0}^{\bar{S}^{\mathscr{B}}(h,p)} \left(\bar{S}^{\mathscr{B}}(h,p)-k\right) P_k 
+ \left(p+Lh\right) \sum_{k=\bar{S}^{\mathscr{B}}(h,p)+1}^{\left(L+1\right)T} \left(k-\bar{S}^{\mathscr{B}}(h,p)\right) P_k.
\]
\textbf{Lower bound:}
\[
\underline{a}_1\left(T,h,p,L\right) \triangleq
h \sum_{k=0}^{\underline{S}^{\mathscr{B}}(h,p)} \left(\underline{S}^{\mathscr{B}}(h,p)-k\right) P_k 
+ \frac{p}{L+1} \sum_{k=\underline{S}^{\mathscr{B}}(h,p)+1}^{\left(L+1\right)T} \left(k-\underline{S}^{\mathscr{B}}(h,p)\right) P_k,
\]
where \(\bar{S}^{\mathscr{B}}(h,p)\triangleq S^{\mathscr{B}*}\left(h,p+Lh\right)\) and \(\underline{S}^{\mathscr{B}}(h,p)\triangleq S^{\mathscr{B}*}\left(h,p/(L+1)\right)\) are the base-stock levels that minimize the corresponding backorder costs. With
\(F_K(s)=\sum_{k=0}^{s}P_k\), they are the smallest integers satisfying
\[
F_K\!\left(\bar S^{\mathscr B}(h,p)\right)
\geq \frac{p+Lh}{p+(L+1)h},
\qquad
F_K\!\left(\underline S^{\mathscr B}(h,p)\right)
\geq \frac{p}{p+(L+1)h}.
\]
Thus the cdf one unit below either base-stock level is strictly smaller than
the corresponding displayed critical fractile.

\begin{lemma}
\label{lem:bounds-base-off}
For all \(T\), the long-run average cost of the lost-sales system with the optimal base-stock level \( S^{\mathscr{L}*}(h,p) \) is bounded as:
\[
\underline{a}_1\left(T,h,p,L\right) \le  C_\mathbf{b}^{\mathscr{L},    S^{\mathscr{L}*}(h,p)}  \left(h,p\right)   \le   \bar{a}_1\left(T,h,p,L\right),
\]
where \(\bar{a}_1\left(T,h,p,L\right)\) and \(\underline{a}_1\left(T,h,p,L\right)\) are as defined above.
\end{lemma}
This is \cref{lem:lost-back-bridges} evaluated at the two backorder optimizers.

For \(n\ge1\), define
\[
K_{n,T}\triangleq\sum_{i=1}^{n}D^i
\sim\operatorname{Binomial}\!\left(nT,1-\lambda_0\right),
\]
and let \(P_{n,k}\triangleq\Pr(K_{n,T}=k)\) denote its exact probability
mass function, so
\(P_{n,k}=\binom{nT}{k}(1-\lambda_0)^k\lambda_0^{nT-k}\).

\begin{lemma} \label{lem:bound-constant-off}
For all \(T\), let \(c^*(h,p)\) be an integer minimizer over the stable
domain \(\{c\in\mathbb Z_+:c<T(1-\lambda_0)\}\), and extend the stationary
cost by \(C_\mathbf c^c(h,p)=+\infty\) outside this domain. Then
\[
C_\mathbf{c}^{c^*(h,p)}(h, p) = a_2(T,h,p) \triangleq  h \sum_{n=1}^{\infty} \sum_{k=0}^{n c^*(h,p)} \left(c^*(h,p) - \frac{k}{n} \right) P_{n,k} + p T(1-\lambda_0) - p c^*(h,p).
\]
\end{lemma}
The proof substitutes the binomial aggregate demand into Spitzer's identity; \cref{lem:constant-cost-forward-difference} locates the integer optimizer.

\begin{theorem} \label{thm:replenishment-gap}
The difference in long-run average profit per cycle between the base-stock policy and the constant-order policy under the myopic offline fulfillment algorithm is bounded as:
\begin{align*}
     a_2(T,h,r_1) - \bar{a}_1(T,h,r_M,L) \leq  \pi_{\mathbf{b}}^{\text{off}} - \pi_{\mathbf{c}}^{\text{off}} \leq a_2(T,h,r_M) - \underline{a}_1(T,h,r_1,L),
\end{align*}
where \( \bar{a}_1(\cdot)\), \(\underline{a}_1(\cdot) \) and \( a_2(\cdot) \) are as defined in \cref{lem:bounds-base-off,lem:bound-constant-off}.
\end{theorem}
\proof{Proof sketch.} Write profit as total demand value minus holding cost
and lost reward. The demand value cancels between the two policies. For the
lower bound, value base-stock lost sales at \(r_M\) and constant-order lost
sales at \(r_1\), and apply
\cref{lem:bounds-base-off,lem:bound-constant-off}; reversing these reward
bounds gives the upper bound. \hfill \Halmos
\endproof

\begin{corollary}
\label{coro:replenishment-gap-sqrtT-scaling}
Assume \(0<\lambda_0<1\) and \(p>0\). As \(T\to\infty\) (with \(L,h,p\) fixed), the three functions satisfy:
\[
\bar{a}_1(T,h,p,L) = \Theta(\sqrt{T}),\quad \underline{a}_1(T,h,p,L) = \Theta(\sqrt{T}) \quad \text{and} \quad a_2(T,h,p) =\Theta(\sqrt{T}).
\]
\end{corollary}

The base-stock terms are binomial newsvendor costs over \((L+1)T\) periods. For constant orders, adjacent workload differences locate the optimizer at \(\Theta(\sqrt T)\) slack, and Spitzer's identity controls inventory. Details are in \cref{app:proof_base_constant_off}. At fixed \(T\), the formulas can be evaluated directly, but their differences may cancel.

\subsection{Offline vs. Greedy under Base-Stock Policy}
\label{subsec:off-greedy-base}

At a common base-stock level, myopic offline and greedy fulfillment serve the same number of requests in every cycle. Their cycle-boundary states and holding costs coincide, so their profit difference is entirely due to allocation across types. Let \(I_1^{\infty,S}\) denote their common stationary start inventory.

\textbf{Setup for this subsection.}
Let \(S^{\mathrm{greedy}}\) and \(S^{\mathrm{off}}\) maximize the respective long-run profits. Existence is verified in \cref{app:proof_off_greedy_base}. For fixed \(S\), write
\[
m_T(S)\triangleq\mathbb E[I_1^{\infty,S}],
\qquad
B_T(S)\triangleq
\mathbb E\!\left[(T(1-\lambda_0)-I_1^{\infty,S})^+\right],
\]
and
\[
G_T(S)\triangleq
\mathbb E\!\left[
 (T(1-\lambda_0)-I_1^{\infty,S})
 \mathbb I\{T\lambda_M<I_1^{\infty,S}\leq T(1-\lambda_0)\}
\right].
\]
Here \(B_T(S)\) is the effective shortfall from mean demand, and \(G_T(S)\) is its part in the deterministic rationing region. Write
\(B_T^a=B_T(S^a)\) and \(G_T^a=G_T(S^a)\).

The bounds use the stationary flow identity in \cref{lem:stationary-flow-balance} and require no additional concentration assumption.

\begin{lemma}[Square-root bound for optimized effective shortfall]
\label{lem:optimized-effective-shortfall-sqrt}
For \(a\in\{\mathrm{off},\mathrm{greedy}\}\), suppose \(S^a\) maximizes the
stationary long-run base-stock profit. There is a finite constant \(K_B\),
depending only on the fixed primitives, such that, for every \(T\geq1\),
\begin{equation}\label{eq:optimized-BG-sqrt}
0\leq G_T^a\leq B_T^a\leq K_B\sqrt T.
\end{equation}
\end{lemma}

The proof compares each optimizer with \(\lceil(L+1)T(1-\lambda_0)\rceil\), uses flow balance to bound lost demand and ending inventory there, and transfers the bound by optimality.

\begin{theorem}\label{thm:offline-greedy-UB}
The difference in long-run average profit per cycle between the myopic offline and the online greedy fulfillment algorithms under their respective optimal base-stock levels is upper bounded by:
\[
\pi_\mathbf{b}^{\text{off}} - \pi_\mathbf{b}^{\text{greedy}}
\le
(r_M-r_1)\frac{1-\lambda_0-\lambda_1}{1-\lambda_0}
\left(
B_T^{\mathrm{off}}
+\sqrt{T(1-\lambda_0)\lambda_0}
\right).
\]
\end{theorem}
\proof{Proof sketch.}
Compare both algorithms at \(S^{\mathrm{off}}\). Their cycle-boundary states
and holding costs agree, so only the allocation reward can differ. The number
of units that offline can reassign from lower- to higher-reward demand is
bounded by the demand overshoot. Centering this overshoot gives
\(B_T^{\mathrm{off}}+\sqrt{T(1-\lambda_0)\lambda_0}\), and each reassignment
gains at most \(r_M-r_1\).
\hfill \Halmos 
\endproof

For the lower bound, define
\[
A_k   \triangleq   \sum_{j=k}^M r_j\lambda_j   -   r_{k-1}  \sum_{j=k}^M \lambda_j,\quad k=2,\dots,M,
\qquad
A   \triangleq   \min_{2\le k\le M} A_k.
\]
Note that \(A>0\) because rewards are strictly increasing and
\(\lambda_j>0\) for every customer type.

\begin{theorem}\label{thm:offline-greedy-LB}
The difference in long-run average profit per cycle between the myopic offline
and the online greedy fulfillment algorithms under their respective optimal
base-stock levels is lower bounded, for every \(T\), by
\begin{equation}\label{eq:LB-effective-shortfall}
\pi_\mathbf{b}^{\text{off}}-\pi_\mathbf{b}^{\text{greedy}}
\ \ge\
\frac{A}{1-\lambda_0}G_T^{\mathrm{greedy}}
- M_1\sqrt{T},
\end{equation}
where \(M_1=\sum_{j=1}^M r_j\sqrt{\lambda_j(1-\lambda_j)}\).
\end{theorem}
\proof{Proof sketch.}
Compare both algorithms at \(S^{\mathrm{greedy}}\). The expected offline
reward is within \(O(\sqrt T)\) of the deterministic fractional relaxation,
while Jensen's inequality bounds the expected greedy reward by its
deterministic counterpart. In the rationing region, the deterministic reward
gain is at least \(A\) per unit of effective shortfall. Averaging over the
stationary inventory gives
\(A G_T^{\mathrm{greedy}}/(1-\lambda_0)\), with the stated relaxation error.
\hfill \Halmos 
\endproof

\begin{corollary} \label{coro:offline-greedy-base}
Under the standing fixed primitives \(0<\lambda_0<1\), \(h>0\),
\(M\ge2\), \(\lambda_j>0\), and
\(0<r_1<\cdots<r_M\), the separately optimized long-run
myopic-offline versus greedy gap satisfies, without an additional
effective-shortfall condition,
\begin{equation}\label{eq:offline-greedy-sqrt-upper}
\pi_\mathbf b^{\mathrm{off}}-\pi_\mathbf b^{\mathrm{greedy}}
=\Theta(\sqrt T).
\end{equation}
\end{corollary}

The upper bound follows from \cref{lem:optimized-effective-shortfall-sqrt,thm:offline-greedy-UB}. For the lower bound, the greedy optimizer enters a \(\sqrt T\)-sized rationing region with positive probability, where an independent demand fluctuation permits \(\Theta(\sqrt T)\) profitable exchanges. No separate condition on \(G_T^{\mathrm{greedy}}\) is needed.

\subsection{Online vs. Offline Fulfillment under Constant-Order Policy}
\label{subsec:online-offline-constant}

The term \(\eta_2=\pi_{\mathbf c}^{\mathrm{off}}-\pi_{\mathbf c}^{\mathrm{on}}\) need not be nonnegative because the myopic offline rule ignores cross-cycle inventory value. \Cref{subsec:constant-replenishment} bounds \(\eta_2\); the next result bounds the reverse difference.

\begin{theorem} \label{thm:on_vs_off_constant}
    Under \cref{ass:online-algorithm}, suppose \(\alpha<1\). For every fixed
    \(\delta\in(0,\mu)\) and all sufficiently large \(T\), the
    stability-restricted optimized profit difference satisfies
    \[
       \pi_{\mathbf c,\delta}^{\mathrm{on}}(T,N)
       -\pi_{\mathbf c,\delta}^{\mathrm{off}}(T,N)
       \leq r_MK_{\mathrm c}T^\alpha .
    \]
    For every finite \(N\), the corresponding difference between the
    unrestricted finite-horizon optimizers is at most \(r_M\mu T\).
    For the cycle-restarted Bayes Selector defined in
    \cref{sec:model} and optimized long-run
    profits, the unrestricted reverse difference is instead at most
    \(K_{\mathrm{gap}}\), independent of \(T\), by
    \cref{cor:fbs-constant-ordinary}.
\end{theorem}
\textit{Proof sketch.} Under a common constant order, the start-inventory gap
equals the preceding end-inventory gap and is bounded by the online carryover
inventory. The offline reward value is \(r_M\)-Lipschitz in inventory, while
the online policy has weakly higher holding cost by
\cref{lem:constant-off-more}. Applying \cref{lem:Inv-on-bound} at
\(c^{\mathrm{on},\delta}\) gives the stated bound; the complete argument is in the
e-companion.
\hfill \Halmos

On \(\mathcal C_\delta(T)\), the bound is \(O(T^\alpha)\). The unrestricted comparison uses the linear envelope without the quadratic condition and \cref{thm:reg-total-constant}(d) with it. \Cref{cor:fbs-constant-ordinary} gives the unrestricted long-run Bayes Selector result.

\subsection{Comparison between Replenishment and Online Fulfillment}
\label{subsec:comparison}

Combining the three signed differences gives the following bounds.
\begin{corollary}
\label{coro:replenish-fulfill-gap}
The long-run average profit difference between the online greedy algorithm
with base-stock replenishment and the general online algorithm with
unrestricted optimized constant-order replenishment, whenever the
optimized long-run profits involved exist and are finite,
satisfies the lower bound
\begin{equation}
\begin{aligned}
\pi_{\mathbf{b}}^{\text{greedy}} - \pi_{\mathbf{c}}^{\text{on}} 
&\geq a_2(T,h,r_1) - \bar{a}_1(T,h,r_M,L) \\
&\quad - \left\{
(r_M-r_1)\frac{1-\lambda_0-\lambda_1}{1-\lambda_0}
\left[
B_T^{\mathrm{off}}+\sqrt{T(1-\lambda_0)\lambda_0}
\right]
\right\} \\
&\quad -r_M(1-\lambda_0)T,
\end{aligned}
\label{ineq:replenish-fulfill-LB}
\end{equation}
and the upper bound
\begin{equation}
\begin{aligned}
\pi_{\mathbf{b}}^{\text{greedy}} - \pi_{\mathbf{c}}^{\text{on}} 
&\leq a_2(T,h,r_M) - \underline{a}_1(T,h,r_1,L) \\
&\quad - \left\{
\frac{A}{1-\lambda_0}G_T^{\mathrm{greedy}}
- M_1\sqrt{T}
\right\} \\
&\quad +r_M(1-\lambda_0)T .
\end{aligned}
\label{ineq:replenish-fulfill-UB}
\end{equation}
If the constant-order online algorithm is the cycle-restarted Bayes Selector
defined in \cref{sec:model}, then, for all sufficiently large \(T\), the
final terms \(-r_M(1-\lambda_0)T\) and \(+r_M(1-\lambda_0)T\)
in \eqref{ineq:replenish-fulfill-LB} and
\eqref{ineq:replenish-fulfill-UB}
can both be replaced by \(-K_{\mathrm{gap}}\) and
\(+K_{\mathrm{gap}}\), respectively, for optimized long-run
profits. The constant is from \cref{cor:fbs-constant-ordinary}.
\end{corollary}

\medskip

\noindent\textbf{Interpretation and scaling.}
For the replenishment part, write
\[
\Delta_{\mathrm{rep}}^{\mathrm{LB}}(T)
:= a_2(T,h,r_1)-\bar{a}_1(T,h,r_M,L),\qquad
\Delta_{\mathrm{rep}}^{\mathrm{UB}}(T)
:= a_2(T,h,r_M)-\underline{a}_1(T,h,r_1,L).
\]
For fixed \(T\), these quantities are evaluated from the displayed formulas.
As \(T\to\infty\), both replenishment contributions are \(O(\sqrt T)\) by
\cref{coro:replenishment-gap-sqrtT-scaling}, although their difference may
exhibit parameter-dependent cancellation. The optimized base-stock
offline-versus-greedy gap is \(\Theta(\sqrt T)\) by
\cref{coro:offline-greedy-base}. For the cycle-restarted Bayes Selector, the optimized
long-run constant-order gap is \(O(1)\) by
\cref{cor:fbs-constant-ordinary}.

The general lower bound in \eqref{ineq:replenish-fulfill-LB} is strictly
negative under the standing assumptions. Indeed, the candidate \(c=0\) gives
\(a_2(T,h,r_1)\leq r_1(1-\lambda_0)T\), while \(\bar a_1\) and the bracketed
offline-versus-greedy term are nonnegative. Its right-hand side is therefore
at most \(-(r_M-r_1)(1-\lambda_0)T<0\). Thus the general expression is a valid
but coarse comparison envelope and cannot certify the desired ranking.
Replacing the linear term by \(-K_{\mathrm{gap}}\) for the Bayes Selector removes this
obstruction and gives a potentially informative parameter-dependent
condition. The proof does not provide numerical values for
\(K_{\mathrm{gap}}\) or for the sufficiently-large-\(T\) threshold.

\medskip

\Cref{fig:sim-comparison} reports the simulated policy ranking and does not
evaluate the sharper Bayes Selector condition.

\section{Look-Ahead Online Fulfillment}\label{sec:look-ahead}

Because the myopic offline benchmark ignores cross-cycle inventory value, an online rule may outperform it. The following example uses the cycle-restarted Bayes Selector defined in \cref{sec:model}.

\begin{example}\label{example:online_better}
Let \(N=T=2\), \((r_1,r_2,r_3)=(1,25,26)\), \((\lambda_0,\lambda_1,\lambda_2,\lambda_3)=(0.3,0.3,0.3,0.1)\), \(h=0.5\), and \(c=1\), with zero initial carryover. Exact enumeration of the \(4^4\) demand paths gives average profits per cycle of \(16.598825\) for myopic offline fulfillment and \(16.6325\) for the Bayes Selector.
\hfill\(\diamondsuit\)
\end{example}

On the path \((j_1^1,j_2^1)=(1,0)\), myopic offline accepts type 1. The Bayes Selector rejects it because its fluid allocation of \(0.2\) is below half of the expected remaining type-1 demand, \(2\lambda_1/2=0.3\). If both next-cycle arrivals are type 3, their pathwise profits are \(27\) and \(51.5\). The online rejection preserves one unit for the next cycle.

\subsection{Algorithm}
We next introduce a receding-horizon rule using the current cycle and \(\tilde N\le\min\{L,N-n\}\) pipeline arrivals. Variables \(v_{t,j}^n,v_{t,0j}^n\) are accepted and rejected type-\(j\) demand in the rest of cycle \(n\); \(v_j^{n+i},v_{0j}^{n+i}\) are their future-cycle counterparts. Let \(I_t^n\) be current inventory and \(O^{n+i}\) the order arriving in cycle \(n+i\). The offline rule solves

\begin{equation}
\begin{aligned}
\max_{v} \quad & \sum_{j \in[M]} \left\{\left(r_j+\tilde{N} \cdot h\right) \cdot v_{t, j}^n+\sum_{i=1}^{\tilde{N}}\left[r_j+(\tilde{N}-i) \cdot h\right] \cdot v_j^{n+i}\right\} -h \tilde{N} \cdot I_t^n \\
\text{s.t.} \quad & \sum_{j \in [M]} v_{t, j}^n \leq I_t^n, \\
& v_{t, j}^n + v_{t, 0 j}^n = D_j^n[t, T], \quad \forall j \in[M], \\
& \sum_{j \in [M]}\left(v_{t, j}^n+\sum_{i=1}^k v_j^{n+i}\right) \leq I_t^n+\sum_{i=1}^k O^{n+i}, \quad \forall k \in[\tilde{N}], \\
& v_j^{n+i} + v_{0j}^{n+i} = \Lambda_j^{n+i}[1, T], \quad \forall j \in[M], i \in[\tilde{N}], \\
& v_{t, j}^n, v_{t, 0 j}^n, v_j^{n+i}, v_{0 j}^{n+i} \geq 0, \quad \forall j \in[M], i \in[\tilde{N}].
\end{aligned}
\label{opt:lookahead-opt}
\end{equation}
The holding coefficients equal the remaining number of cycle ends. The online rule replaces \(D_j^n[t,T]\) by \(\Lambda_j^n[t,T]\) and applies \cref{alg:look-ahead}.

\begin{algorithm}[htpb] 
    \caption{Look-Ahead Online Fulfillment Algorithm}
    \label{alg:look-ahead}
    \begin{algorithmic}[1]
        \For{$t=1,2,\ldots,T$}
            \State Observe the realized demand sequence $ j^n_1,j^n_2,\dots,j^n_t $.
            \If{$j_t^n=0$}
                \State Update inventory: $I_{t+1}^n\gets I_t^n$.
            \Else
            \State Solve the look-ahead optimization problem for $ v^n_{t,j_t^n} $ and $ v^n_{t,0j_t^n} $.
            \If{$I_t^n>0$ and $v^n_{t,j_t^n} \geq v^n_{t,0j_t^n}$}
                \State Accept customer $ j^n_t $ and update inventory: $ I_{t+1}^n \gets I_t^n - 1 $.
            \Else
                \State Reject customer $ j^n_t $ and maintain inventory: $ I_{t+1}^n \gets I_t^n $.
            \EndIf
            \EndIf
        \EndFor 
    \end{algorithmic}
\end{algorithm}

\begin{remark}
The formulation extends to multiple resources by indexing allocations and constraints by resource. This does not extend its guarantees; \cref{sec:extension} analyzes separate implementations.
\end{remark}

We next give an exact full-information benchmark and a separate quota policy. Unlike \cref{alg:look-ahead}, the benchmark uses realized demand over the full horizon, and the guarantee below applies only to the quota policy.

\subsection{Full-Information Benchmark and a Guaranteed Quota Policy}
\label{subsec:full-information}
Write \(D_{nj}=D_j^n[1,T]\) and let \(x_{nj}\) be the number of type \(j\)
customers accepted in cycle \(n\), with
\(a_n=\sum_jx_{nj}\).

\begin{lemma}[Exact full-information formulations]
\label{lem:full-information-lp}
Fix a finite horizon \(N\).
\begin{enumerate}[label=(\alph*)]
\item Under base stock \(S\), let
\((I_1,q^2,\ldots,q^{L+1})\) be the complete post-order state entering
cycle 1, so \(I_1+\sum_{m=2}^{L+1}q^m=S\), and define
\[
B_n=I_1+\sum_{m=2}^{\min\{n,L+1\}}q^m.
\]
The full-information profit equals the optimum of
\begin{equation}\label{eq:full-information-base-lp}
\begin{aligned}
\max_x\quad&
\sum_{n=1}^N\sum_{j=1}^Mr_jx_{nj}
-h\sum_{n=1}^Ne_n\\
\text{\rm s.t.}\quad&
0\leq x_{nj}\leq D_{nj},\\
&\sum_{k=\max\{1,n-L\}}^na_k\leq B_n,\qquad n\in[N],\\
&e_n=B_n-\sum_{k=\max\{1,n-L\}}^na_k .
\end{aligned}
\end{equation}
\item Under constant order \(c\) and zero carryover before cycle 1, the
full-information profit equals the optimum of
\begin{equation}\label{eq:full-information-constant-lp}
\begin{aligned}
\max_x\quad&
\sum_{n=1}^N\sum_{j=1}^Mr_jx_{nj}
-h\sum_{n=1}^Ne_n\\
\text{\rm s.t.}\quad&
0\leq x_{nj}\leq D_{nj},\\
&\sum_{k=1}^na_k\leq nc,\qquad n\in[N],\\
&e_n=nc-\sum_{k=1}^na_k .
\end{aligned}
\end{equation}
\end{enumerate}
Both LP relaxations have integral optima for integer data.
\end{lemma}

Let \(\pi_{\mathbf b}^{\mathrm{FI}}(T,S,N)\) and
\(\pi_{\mathbf c}^{\mathrm{FI}}(T,c,N)\) denote the expected optimal
values of these full-information problems, divided by \(N\).

\begin{proposition}[Constant-order bridge to full information]
\label{prop:constant-full-information-bridge}
For every integer \(c\geq0\) and \(N\geq1\),
\[
0\leq
\pi_{\mathbf c}^{\mathrm{FI}}(T,c,N)
-\pi_{\mathbf c}^{\mathrm{off}}(T,c,N)
\leq
2r_M\sum_{j=1}^M\sqrt{T\lambda_j(1-\lambda_j)}.
\]
All profit terms include holding costs, and the bound is uniform in
\(N\) and \(c\). It therefore remains valid after separately optimizing
the full-information and myopic-offline objectives over \(c\). Let
\(c^{\mathrm{FI}}(T,N)\) and \(c^{\mathrm{on}}(T,N)\) maximize the
full-information and online finite-horizon objectives, respectively.
Whenever the hypotheses of
\cref{thm:reg-total-constant}(d) hold,
\[
\pi_{\mathbf c}^{\mathrm{FI}}(T,c^{\mathrm{FI}}(T,N),N)
-\pi_{\mathbf c}^{\mathrm{on}}(T,c^{\mathrm{on}}(T,N),N)
=O(\sqrt T+T^\alpha).
\]
\end{proposition}

Pathwise, work conservation prevents holding cost from enlarging the gap. Lipschitz continuity of the fractional knapsack value and centered demand then give the uniform \(O(\sqrt T)\) bound.

For fixed base-stock state and target, replace \(D_{nj}\) in \eqref{eq:full-information-base-lp} by \(T\lambda_j\), and let \(x^*\) be optimal. The mean-LP quota policy accepts at most \(\lfloor x^*_{nj}\rfloor\) type-\(j\) customers in cycle \(n\).

\begin{theorem}[Full-information guarantee for the mean-LP quota policy]
\label{thm:base-quota-full-information}
For every common post-order initialization and fixed \(S\),
\[
0\leq
\pi_{\mathbf b}^{\mathrm{FI}}(T,S,N)
-\pi_{\mathbf b}^{\mathrm{quota}}(T,S,N)
\leq
\bigl(r_M+h(L+1)\bigr)
\sum_{j=1}^M
\left[1+\sqrt{T\lambda_j(1-\lambda_j)}\right].
\]
The bound is uniform in \(N,S\). If optimization over \(S\) uses the common
initialization map in \cref{sec:model}, it also holds between the
separately optimized full-information and quota values. No stability
condition across cycles is required.
\end{theorem}

Concavity and Jensen's inequality bound the expected full-information value by the mean-demand LP. Truncating its solution by realized counts preserves the rolling constraints and loses at most one plus a standard deviation per type and cycle.

\Cref{prop:full-information-sqrt-lower} shows that the fixed-target rate is sharp. The simulations below evaluate \cref{alg:look-ahead}, not the quota policy.

\subsection{Simulation}

We compare the myopic-offline and cycle-restarted Bayes Selector baselines with the corresponding offline and online look-ahead policies under base-stock replenishment. Profits are reported per period. The common initialization has inventory position \(S\):
\[
I_{1}^{1,\text{alg}}(S,\mathcal{H})=\left\lfloor \tfrac{S}{L+1}\right\rfloor,\qquad 
\bm{q}^{1,\text{alg}}(S,\mathcal{H})=\left(\left\lfloor \tfrac{S}{L+1}\right\rfloor,\ldots,\left\lfloor \tfrac{S}{L+1}\right\rfloor\right),
\]
and set \(q^{L+1,\text{alg}}(S,\mathcal{H})=S-I_{1}^{1,\text{alg}}(S,\mathcal{H})-\sum_{l=1}^{L-1}q^{1+l,\text{alg}}(S,\mathcal{H})\).
All entries are nonnegative integers, and every algorithm uses the same initial state. We vary the lead time and look-ahead horizon, the cycle length, and the cost and demand primitives.

\subsubsection{Lead time and look-ahead cycles}

We use \(\tilde N\in\{0,\ldots,5\}\le L\) and \(L\in\{8,10\}\). In \cref{fig:simulation-leadtime-lookahead}, \(T=N=50\), \(h=2\), \(r=[1,9,10]\), \(\lambda=[0.2,0.2,0.3,0.3]\), and \(K=1{,}000\).

\paragraph{Observations.}
Every tested \(\tilde N>0\) improves on its corresponding baseline. Most of the gain occurs by \(\tilde N=3\), and the gain is larger at \(L=10\) than at \(L=8\).

\begin{figure}[H]
    \centering
    \begin{subfigure}{0.48\textwidth}
        \centering
        \includegraphics[width=\textwidth]{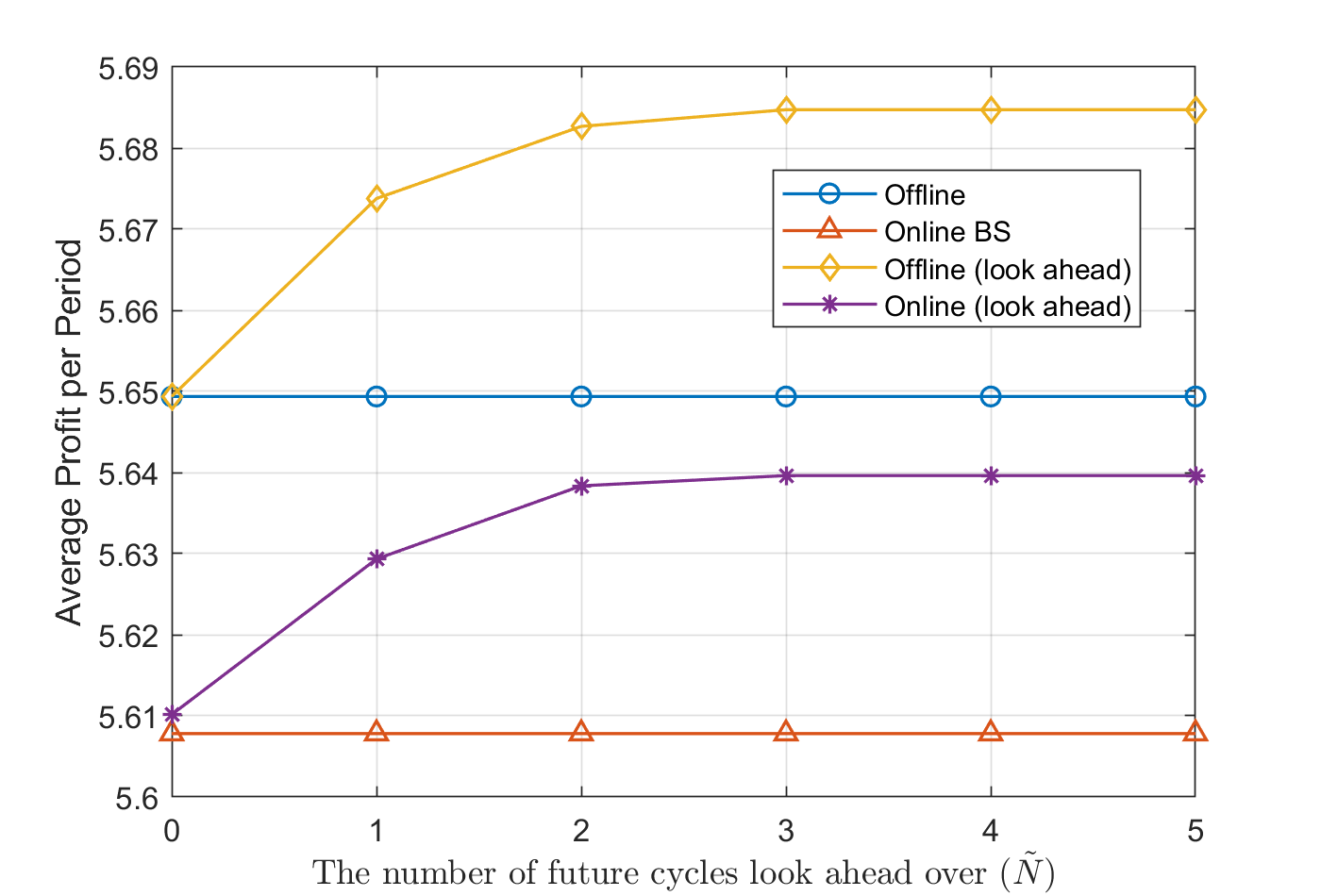}
        \caption{Lead time \( L = 8 \)}
        \label{fig:L3}
    \end{subfigure}
    \hfill
    \begin{subfigure}{0.48\textwidth}
        \centering
        \includegraphics[width=\textwidth]{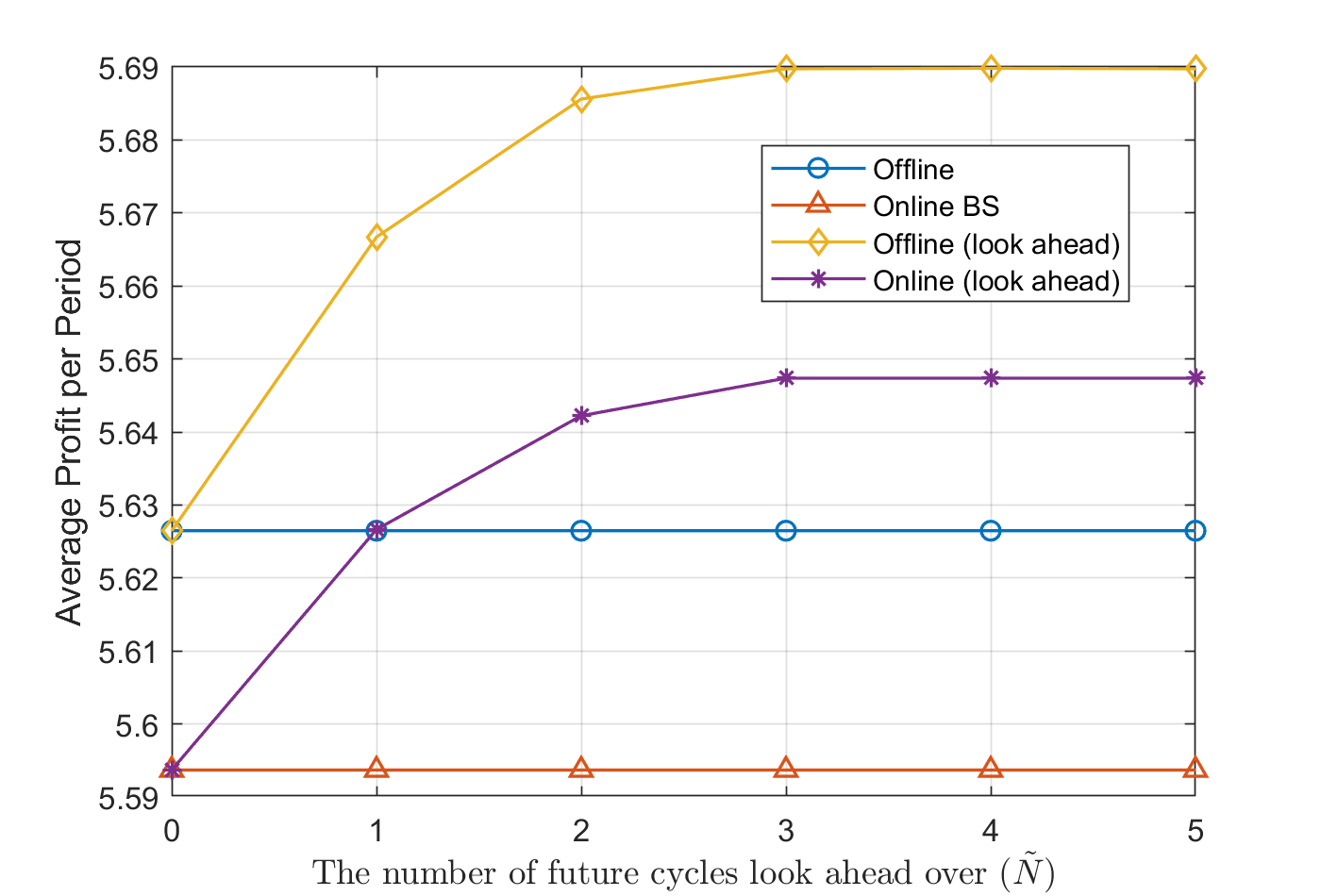}
        \caption{Lead time \( L = 10 \)}
        \label{fig:L5}
    \end{subfigure}
    \caption{Average profit per period under the myopic-offline and Bayes Selector baselines and their respective look-ahead policies for different lead times and look-ahead horizons.}
    \label{fig:simulation-leadtime-lookahead}
\end{figure}

\subsubsection{Number of periods in one cycle}

We next vary the number of periods \(T\) while fixing \(N=50\), \(L=10\) with \(\tilde{N}=5\), \(h=3\), \([r_1,r_2,r_3]=[1,9,10]\), \([\lambda_0,\lambda_1,\lambda_2,\lambda_3]=[0.2,0.2,0.3,0.3]\), and \(K=100{,}000\).

\paragraph{Observations.}
In \cref{fig:sim-comparison-T}, profit increases and approaches a plateau with \(T\). Each look-ahead policy remains above its baseline, although the gaps narrow slightly.

\begin{figure}[htbp]
  \centering
  \includegraphics[width=0.7\textwidth]{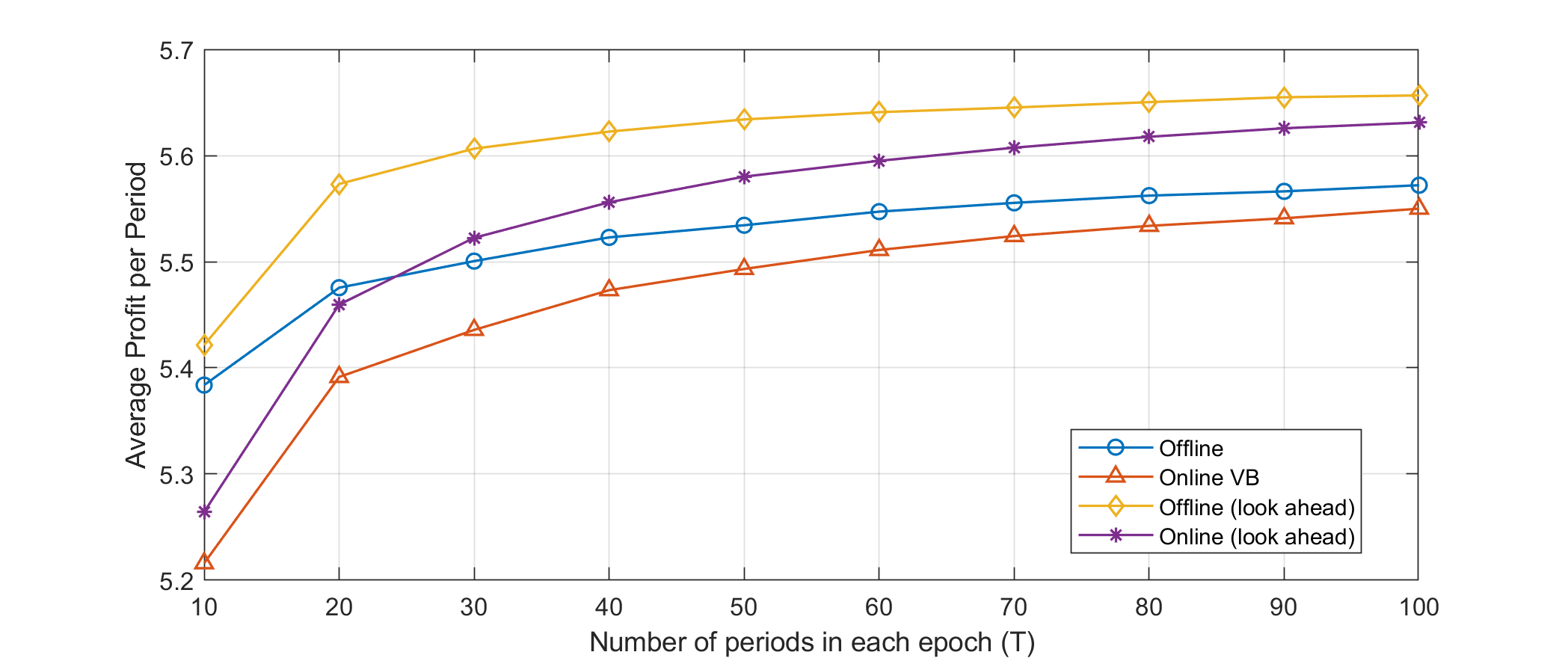}
  \caption{Average profit per period under the myopic-offline and Bayes Selector baselines and their respective look-ahead policies for different cycle lengths \(T\).
  }
  \label{fig:sim-comparison-T}
\end{figure}

\subsubsection{Holding costs, rewards and arrival rates}

Finally, we vary \((h,r,\lambda)\) jointly because their separate effects on the look-ahead gap need not be monotone. We fix \(T=N=50\), \(L=10\), and \(K=2{,}000\).

In \cref{tab:profit-results}, look-ahead improves average profit by about \(0.7\%\) to \(2.4\%\). The online look-ahead policy also often exceeds myopic offline fulfillment.

\begin{table}[htbp]
\centering
\caption{Average profits per period, defined as total profit divided by \(NT\), and percentage improvements with look-ahead. Demand entries are \((\lambda_1,\lambda_2,\lambda_3)\), with \(\lambda_0=0.2\) throughout.}
\label{tab:profit-results}
\resizebox{\textwidth}{!}{%
\begin{tabular}{ccccccccc}
\toprule
\textbf{Rewards} & \textbf{Demand $\boldsymbol{\lambda}$} & \textbf{h} &
\shortstack{\textbf{Offline}\\\textbf{(Myopic)}} &
\shortstack{\textbf{Offline}\\\textbf{(Look-Ahead)}} & \textbf{Gain (\%)} &
\shortstack{\textbf{Online}\\\textbf{Bayes Selector}} &
\shortstack{\textbf{Online}\\\textbf{Look-Ahead}} & \textbf{Gain (\%)} \\
\midrule
$[1,9,10]$ & $[0.2, 0.3, 0.3]$ & 2 & 5.6265 & 5.6898 & 1.13\% & 5.5937 & 5.6472 & 0.96\% \\
$[1,9,10]$ & $[0.2, 0.3, 0.3]$ & 3 & 5.5343 & 5.6340 & 1.80\% & 5.4932 & 5.5801 & 1.58\% \\
$[1,9,10]$ & $[0.2, 0.2, 0.4]$ & 2 & 5.7304 & 5.7902 & 1.04\% & 5.6966 & 5.7341 & 0.66\% \\
$[1,9,10]$ & $[0.2, 0.2, 0.4]$ & 3 & 5.6298 & 5.7293 & 1.77\% & 5.5897 & 5.6663 & 1.37\% \\
$[1,9,10]$ & $[0.2, 0.1, 0.5]$ & 2 & 5.8250 & 5.8876 & 1.07\% & 5.7816 & 5.8402 & 1.01\% \\
$[1,9,10]$ & $[0.2, 0.1, 0.5]$ & 3 & 5.7256 & 5.8266 & 1.76\% & 5.6829 & 5.7700 & 1.53\% \\
\midrule
$[1,5,10]$ & $[0.2, 0.3, 0.3]$ & 2 & 4.4475 & 4.4985 & 1.15\% & 4.4121 & 4.4656 & 1.21\% \\
$[1,5,10]$ & $[0.2, 0.3, 0.3]$ & 3 & 4.3738 & 4.4480 & 1.69\% & 4.3392 & 4.3976 & 1.35\% \\
$[1,5,10]$ & $[0.2, 0.2, 0.4]$ & 2 & 4.9404 & 4.9966 & 1.14\% & 4.9102 & 4.9616 & 1.05\% \\
$[1,5,10]$ & $[0.2, 0.2, 0.4]$ & 3 & 4.8643 & 4.9583 & 1.93\% & 4.8286 & 4.9009 & 1.50\% \\
$[1,5,10]$ & $[0.2, 0.1, 0.5]$ & 2 & 5.4372 & 5.5000 & 1.16\% & 5.4048 & 5.4559 & 0.95\% \\
$[1,5,10]$ & $[0.2, 0.1, 0.5]$ & 3 & 5.3463 & 5.4527 & 1.99\% & 5.3087 & 5.3978 & 1.68\% \\
\midrule
$[1,2,10]$ & $[0.2, 0.3, 0.3]$ & 2 & 3.5797 & 3.6219 & 1.18\% & 3.5560 & 3.5877 & 0.89\% \\
$[1,2,10]$ & $[0.2, 0.3, 0.3]$ & 3 & 3.5174 & 3.6030 & 2.43\% & 3.4891 & 3.5556 & 1.90\% \\
$[1,2,10]$ & $[0.2, 0.2, 0.4]$ & 2 & 4.3636 & 4.4132 & 1.14\% & 4.3362 & 4.3766 & 0.93\% \\
$[1,2,10]$ & $[0.2, 0.2, 0.4]$ & 3 & 4.2878 & 4.3907 & 2.40\% & 4.2452 & 4.3472 & 2.40\% \\
$[1,2,10]$ & $[0.2, 0.1, 0.5]$ & 2 & 5.1417 & 5.2024 & 1.18\% & 5.1149 & 5.1635 & 0.95\% \\
$[1,2,10]$ & $[0.2, 0.1, 0.5]$ & 3 & 5.0591 & 5.1748 & 2.29\% & 5.0254 & 5.1089 & 1.66\% \\
\bottomrule
\end{tabular}
}
\end{table}

\section{Extension to Multi-Resource Online Fulfillment} \label{sec:extension}

We extend the regret transfers to \(d\) resources. Aggregate inventory no
longer determines either the reward earned by an allocation or the composition
of the remaining stock, so the single-resource coupling does not transfer
automatically. The base-stock results therefore treat a shadow-capped rule
under common resource order, the original uncapped rule under additive rewards
and common lead time, a universal-best resource with arbitrary leads, and a
type-dedicated rule for general rewards. The constant-order result allows
general positive rewards; its low-\(\alpha\) unrestricted case additionally
uses common resource order.

\textbf{Multiple resources.}
Type \(j\) may use any resource \(\ell\in[d]\), earning \(r_{j\ell}\), where
\[
0<r_{\min}\triangleq\min_{j,\ell}r_{j\ell}
\leq r_{\max}\triangleq\max_{j,\ell}r_{j\ell}<\infty,
\]
and the rewards are not all identical.

Here \cref{ass:online-algorithm} is vector-uniform: the same conditional \(C_1T^\alpha\) bound holds for every feasible initial inventory vector.

The following within-cycle structure supplies the componentwise and prefix comparisons used below.

\begin{assumption}[Common resource order and nested use]
\label{ass:mr-common-order}
There is a permutation of the resources such that, after relabeling,
\[
r_{j1}\geq r_{j2}\geq\cdots\geq r_{jd}>0,
\qquad j\in[M].
\]
Whenever the online implementation accepts a request, it assigns that request
to an available unit of the smallest-index resource. Ties are resolved by a
fixed deterministic rule. For the myopic offline benchmark, among
reward-maximizing within-cycle allocations we select one that never leaves a
unit of a smaller-index resource unused while using a unit of a larger-index
resource.
\end{assumption}

Availability means physical on-hand inventory for uncapped and constant-order implementations. For the shadow-capped rule it means the virtual eligible vector \(\bar{\mathbf I}^n=\mathbf I_1^{n,\mathrm{off}}\); excess physical inventory is protected.

The offline selection is without loss: replacing a used resource \(m\) by an unused \(\ell<m\) changes reward by \(r_{j\ell}-r_{jm}\ge0\). The condition is entirely within one cycle.

If \(r_{j\ell}=a_j+b_\ell\), sorting by \(b_\ell\) gives a common order, but additivity is not necessary. For example, \(\left(\begin{smallmatrix}10&1\\4&2\end{smallmatrix}\right)\) has a common order without being additive.

Common order does not imply low regret; \cref{prop:mr-common-order-impossibility} gives a nonadditive counterexample.

For additive rewards, the scalar Bayes Selector on aggregate inventory followed by common-order assignment verifies the vector-uniform assumption with \(\alpha=0\); see \cref{prop:mr-additive-fbs-validation}. A nonadditive nested-use rule requires separate verification.

\textbf{State and inventory (multi-resource).}
Resource \(\ell\) has lead time \(L_\ell\), inventory \(I_{t,\ell}^{n,\mathrm{alg}}(\mathbf z,\mathcal H)\), and orders \(q_\ell^{n+l,\mathrm{alg}}(\mathbf z,\mathcal H)\) arriving in cycle \(n+l\). Its pre-order pipeline is
\[
\bm{q}_{\ell}^{n,\text{alg}}(\mathbf{z},\mathcal H)
  \triangleq  
\left(q_{\ell}^{  n+1,\text{alg}}(\mathbf{z},\mathcal H),\ldots,q_{\ell}^{  n+L_\ell-1,\text{alg}}(\mathbf{z},\mathcal H)\right),
\]
Collect these quantities as
\[
{\mathbf{Q}}^{  n,\text{alg}}(\mathbf{z},\mathcal H)  \triangleq  \left(\bm{q}_{\ell}^{  n,\text{alg}}(\mathbf{z},\mathcal H)\right)_{\ell\in[d]},
\qquad
\mathbf{I}_{t}^{n,\text{alg}}(\mathbf{z},\mathcal H)  \triangleq  \left(I_{t,\ell}^{n,\text{alg}}(\mathbf{z},\mathcal H)\right)_{\ell\in[d]}.
\]
The cycle-start state is \((\mathbf I_1^{n,\mathrm{alg}},\mathbf Q^{n,\mathrm{alg}})\), measured after current deliveries. The holding charge is \(\sum_\ell h_\ell I_{T+1,\ell}^{n,\mathrm{alg}}\), where \(0<h_\ell<\infty\).

\textbf{Replenishment policies (multi-resource).}
\begin{itemize}
\item \textbf{Base-stock policy} (\(\mathbf z=\mathbf S\)). For every \(n,\ell\),
\begin{equation}\label{eq:base-replenish-mr}
q_{\ell}^{  n+L_\ell,\text{alg}}(  \mathbf{S},\mathcal H)
  =  
S_\ell  -  I_{1,\ell}^{n,\text{alg}}(  \mathbf{S},\mathcal H)
  -  
\sum_{l=1}^{L_\ell-1} q_{\ell}^{  n+l,\text{alg}}(  \mathbf{S},\mathcal H),
\end{equation}
\[
I_{1,\ell}^{n+1,\text{alg}}(  \mathbf{S},\mathcal H)
  =  
I_{T+1,\ell}^{n,\text{alg}}(  \mathbf{S},\mathcal H)
  +  
q_{\ell}^{  n+1,\text{alg}}(  \mathbf{S},\mathcal H),
\]
with the usual pipeline shift
\[
\bm q_{\ell}^{n+1,\mathrm{alg}}(\mathbf S,\mathcal H)
=\left(q_{\ell}^{n+2,\mathrm{alg}},\ldots,
q_{\ell}^{n+L_\ell,\mathrm{alg}}\right),
\]
which is empty when \(L_\ell=1\). Fixed-parameter comparisons allow any common nonnegative integer initial state satisfying
\[
I_{1,\ell}^{1,\text{alg}}(  \mathbf{S},\mathcal H)  +  \sum_{l=1}^{L_\ell-1} q_{\ell}^{  1+l,\text{alg}}(  \mathbf{S},\mathcal H)\ \le\ S_\ell
\qquad\text{for all }\ell\in[d].
\]
For optimization, use the same deterministic map for both algorithms. Write \(P_\ell=L_\ell+1\) and \(S_\ell=P_\ell s_\ell+t_\ell\), where \(0\le t_\ell<P_\ell\), and set
\[
\bigl(I_{1,\ell}^{1,\mathrm{alg}},q_\ell^{2,\mathrm{alg}},
\ldots,q_\ell^{L_\ell+1,\mathrm{alg}}\bigr)
=(a_{\ell,0},a_{\ell,1},\ldots,a_{\ell,L_\ell}),
\qquad
a_{\ell,k}=
\begin{cases}
s_\ell, & 0\leq k<L_\ell,\\
s_\ell+t_\ell, & k=L_\ell.
\end{cases}
\]
This is the componentwise version of the single-resource initialization.

\item \textbf{Constant-order policy} (\(\mathbf z=\mathbf c\)). Resource \(\ell\) orders \(c_\ell\) each cycle:
\[
q_{\ell}^{n,\text{alg}}(  \mathbf{c},\mathcal H)  =  c_\ell\quad\text{for all cycles }n,
\qquad
I_{1,\ell}^{n+1,\text{alg}}(  \mathbf{c},\mathcal H)
  =  
I_{T+1,\ell}^{n,\text{alg}}(  \mathbf{c},\mathcal H)+c_\ell.
\]
Lead times then leave the recursion. We use zero initial carryover, so
\(
I_{1,\ell}^{1,\text{alg}}(\mathbf c,\mathcal H)=c_\ell
\) for every \(\ell\) and every fulfillment algorithm.
\end{itemize}

\textbf{Profit and performance metrics.} 
Let \(\pi_\mathbf b^{\mathrm{alg}}(T,\mathbf S,N)\) and \(\pi_\mathbf c^{\mathrm{alg}}(T,\mathbf c,N)\) be expected profit per cycle, summed over resources. Vector arguments distinguish this notation from the single-resource case.

Let \(C_{\mathrm m}=C_1/r_{\min}\). The augmentation argument in \cref{lem:mr-total-count-gap} bounds the expected aggregate terminal-inventory difference by \(C_{\mathrm m}T^\alpha\). Under common ordering, resource use is a nondecreasing one-Lipschitz function of total acceptances, so \cref{lem:extend-inv-diff-resource} gives the same componentwise bound. Without common ordering, only the aggregate comparison remains.

\subsection{Base-Stock Replenishment}

\paragraph{Finite-horizon convention.}
In every optimized base-stock comparison below, each profit is evaluated from
the deterministic componentwise initialization specified above at its target.
For \(a\in\{\mathrm{off},\mathrm{shadow},\mathrm{uncap}\}\), write
\[
\mathbf S^a=\mathbf S^a(T,N)
\in\arg\max_{\mathbf S\in\mathbb Z_+^d}
\pi_{\mathbf b}^a(T,\mathbf S,N)
\]
whenever the corresponding policy is considered. Thus the systems optimize
their targets separately. Here \(\mathrm{uncap}\) denotes the original pooled
uncapped implementation.

\paragraph{Shadow-capped comparison.}
Both within-cycle allocations use \(\bar{\mathbf I}^n=\mathbf I_1^{n,\mathrm{off}}\). A virtual use of resource \(\ell\) consumes the same physical resource, while excess physical stock is protected. \Cref{lem:inv-diff-base-resource} then bounds the expected start and terminal differences by \(L_\ell C_{\mathrm m}T^\alpha\) and \((L_\ell+1)C_{\mathrm m}T^\alpha\).
For a fixed target, the online and offline systems start from the same complete
feasible state.

\begin{theorem}[Shadow-capped multi-resource base-stock transfer]
\label{thm:extend-reg-total-base}
Suppose \cref{ass:online-algorithm,ass:mr-common-order} hold, and consider
the componentwise shadow-capped online implementation. For every finite
\(N\geq1\),
\begin{equation*} %\label{ineq:reg_total}
\pi_\mathbf{b}^{\text{off}}(T,\mathbf{S}^{\text{off}},N)
-\pi_\mathbf{b}^{\text{shadow}}(T,\mathbf{S}^{\text{shadow}},N)
\leq
\left[C_1+\sum_{\ell=1}^d h_\ell(L_\ell+1)C_{\mathrm m}\right]T^\alpha
=\mathcal O(T^\alpha).
\end{equation*}
\end{theorem}

\emph{Proof idea.}
The common virtual initial vector gives reward loss at most
\(C_1T^\alpha\) per cycle. The componentwise pipeline bounds in
\cref{lem:inv-diff-base-resource} control the holding-cost difference;
summing them and evaluating at the offline optimizer proves the result.

\paragraph{Original uncapped policy with a common lead time.}
For additive rewards and a common lead time, aggregate inventory follows a scalar recursion and common ordering controls its resource composition.

\begin{proposition}[Uncapped multi-resource base stock with common lead time]
\label{prop:mr-uncapped-base-additive}
Suppose \cref{ass:online-algorithm,ass:mr-common-order} hold,
\(r_{j\ell}=a_j+b_\ell\), and \(L_\ell=L\geq1\) for every resource
\(\ell\). For the original uncapped online implementation, there are constants
\(K,T_0<\infty\), independent of \(N\), such that, for all \(T\geq T_0\)
and finite \(N\geq1\),
\[
\pi_{\mathbf b}^{\mathrm{off}}
 (T,\mathbf S^{\mathrm{off}},N)
-\pi_{\mathbf b}^{\mathrm{uncap}}
 (T,\mathbf S^{\mathrm{uncap}},N)
\leq KT^\alpha.
\]
\end{proposition}

The result requires additive rewards and a common lead time. The nested-use
scalar Bayes Selector in \cref{prop:mr-additive-fbs-validation} has
\(\alpha=0\), and hence an \(O(1)\) bound.

\emph{Proof idea.}
A work-conserving candidate locates the offline aggregate target within
\(O(\sqrt T)\) below mean pipeline demand. The scalar readiness-span bound
then controls the aggregate state difference, while additivity and prefix
sales identities transfer this bound to resource bonuses and holding costs.

\paragraph{A resource that is best for every type.}
A resource \(\ell^\star\) is best for every type if
\[
r_{j\ell^\star}=\max_{\ell\in[d]}r_{j\ell},
\qquad j\in[M].
\]
Stocking only this resource gives a scalar candidate under heterogeneous lead
times.

\begin{proposition}[Original uncapped base stock with a universal-best
resource]\label{prop:mr-universal-best-base}
Suppose \(\ell^\star\) is best for every type and the original uncapped
online algorithm, when only \(\ell^\star\) has inventory, satisfies the
scalar uniform \cref{ass:online-algorithm}. The other rewards need not be
additive, and the lead times may be arbitrary fixed positive integers. There
are constants
\(A,C,T_0<\infty\), independent of \(N,T\), such that, for all
\(T\geq T_0\) and finite \(N\geq1\),
\[
\pi_{\mathbf b}^{\mathrm{off}}
 (T,\mathbf S^{\mathrm{off}},N)
-\pi_{\mathbf b}^{\mathrm{uncap}}
 (T,\mathbf S^{\mathrm{uncap}},N)
\leq A\sqrt T+CT^\alpha.
\]
\end{proposition}

The bound is \(O(T^\alpha)\) when \(\alpha\geq1/2\) and \(O(\sqrt T)\)
when \(\alpha<1/2\).

\emph{Proof idea.}
Stocking only the universal-best resource gives a scalar lower bound on the
optimized online profit. A mean-demand target loses \(O(\sqrt T)\) relative
to serving all demand, and the scalar uncapped transfer adds
\(O(T^\alpha)\). The fixed-target version is
\cref{cor:mr-best-resource-fixed-target} in the electronic companion.

\begin{corollary}[Additive rewards with arbitrary fixed lead times]
\label{cor:mr-additive-arbitrary-leads}
Suppose \(r_{j\ell}=a_j+b_\ell>0\) and the lead times \(L_\ell\geq1\) are
arbitrary fixed constants. Choose \(\ell^\star\in\arg\max_\ell b_\ell\),
and suppose that, when only resource \(\ell^\star\) has inventory, the
original uncapped online algorithm satisfies the scalar uniform
\cref{ass:online-algorithm}. Then the conclusion of
\cref{prop:mr-universal-best-base} holds. In particular, there are constants
\(A,C,T_0<\infty\), independent of \(N,T\), such that, for all \(T\geq T_0\)
and finite \(N\geq1\),
\[
\pi_{\mathbf b}^{\mathrm{off}}
 (T,\mathbf S^{\mathrm{off}},N)
-\pi_{\mathbf b}^{\mathrm{uncap}}
 (T,\mathbf S^{\mathrm{uncap}},N)
\leq A\sqrt T+CT^\alpha.
\]
\end{corollary}

The bound is \(O(T^\alpha)\) when \(\alpha\geq1/2\) and \(O(\sqrt T)\) when
\(\alpha<1/2\). A resource maximizing \(b_\ell\) is best for every type, so
the result permits heterogeneous lead times without nested resource use.

\paragraph{General rewards with a dedicated implementation.}
For each type choose \(\ell(j)\in\arg\max_\ell r_{j\ell}\), and define
\[
\mathcal J_\ell=\{j:\ell(j)=\ell\},
\qquad
p_\ell=\sum_{j\in\mathcal J_\ell}\lambda_j.
\]
The type-dedicated Bayes Selector is the uncapped policy that serves types in
\(\mathcal J_\ell\) only from resource \(\ell\), runs a separate
cycle-restarted scalar Bayes Selector defined in \cref{sec:model} on each
nonempty group, and does not use cross-resource substitution. Types with tied
rewards inside a group may first be merged, while a singleton group uses the
accept-all rule. Define the
mean-pipeline target by
\[
S_\ell^0=(L_\ell+1)\lceil p_\ell T\rceil,
\qquad \ell\in[d].
\]

\begin{proposition}[Type-dedicated Bayes Selector]
\label{prop:mr-dedicated-base}
For arbitrary positive reward matrices and arbitrary fixed lead times, the
type-dedicated Bayes Selector defined above satisfies the following bound.
There are constants \(A_{\mathrm d},C_{\mathrm d},T_{\mathrm d}<\infty\),
independent of \(N,T\), such that, for all \(T\geq T_{\mathrm d}\) and
finite \(N\geq1\),
\[
\pi_{\mathbf b}^{\mathrm{off}}
 (T,\mathbf S^{\mathrm{off}},N)
-\pi_{\mathbf b}^{\mathrm{ded\text{-}Bayes}}(T,\mathbf S^0,N)
\leq A_{\mathrm d}\sqrt T+C_{\mathrm d}.
\]
Optimizing the target vector for this policy preserves the bound.
\end{proposition}

\emph{Proof idea.}
Each nonempty type group is a scalar problem on its assigned resource. A
mean-demand target loses \(O(\sqrt T)\) relative to serving all assigned
demand, and the cycle-restarted Bayes Selector adds \(O(1)\). Summing the scalar bounds
gives the result.

\subsection{Constant-Order Replenishment}
Let \(c=\sum_\ell c_\ell\). Since all rewards are positive, myopic offline serves \(\min\{\sum_\ell I_\ell,D^n\}\) customers, giving the scalar aggregate recursion from \cref{subsec:constant-replenishment}.
All comparisons in this subsection use the zero-carryover initialization
specified above.

The unrestricted result uses the following vector-uniform, single-cycle condition.
\begin{assumption}[Multi-resource single-cycle quadratic count loss]
\label{ass:mr-single-cycle-sm2}
For the exponent \(0\leq\alpha<1\) in \cref{ass:online-algorithm}, let
\(K^{\mathrm{on}}(\mathbf I)\) be the online number fulfilled from any
initial vector \(\mathbf I\), and put
\[
\epsilon^n(\mathbf I)
=\min\{D^n,\textstyle\sum_\ell I_\ell\}
-K^{\mathrm{on}}(\mathbf I)\geq0.
\]
There is a constant \(K_{2,\mathrm m}<\infty\), independent of
\(T,n,\mathbf I\), and the pre-cycle history, such that
\[
\mathbb E\!\left[
\epsilon^n(\mathbf I)^2
\middle|\mathcal F_{\mathrm{pre}}\right]
\leq K_{2,\mathrm m}T^{2\alpha}.
\]
The arrivals in the new cycle are independent of
\(\mathcal F_{\mathrm{pre}}\).
\end{assumption}

The augmentation argument bounds count loss by realized regret divided by \(r_{\min}\), so a conditional quadratic regret bound is sufficient. \Cref{ass:mr-single-cycle-et} is a stronger alternative. The additive Bayes Selector construction satisfies both conditions with \(\alpha=0\); common order alone does not verify them for a nonadditive matrix.

Fix \(\delta_{\mathrm m}\in(0,1-\lambda_0)\) and define the restricted domain
\[
\mathcal C_{\mathrm m,\delta_{\mathrm m}}(T)
\triangleq
\left\{\mathbf c\in\mathbb Z_+^d:
\sum_{\ell=1}^dc_\ell
\leq\lfloor(1-\lambda_0-\delta_{\mathrm m})T\rfloor\right\}.
\]
For \(a\in\{\mathrm{on},\mathrm{off}\}\), write
\[
\mathbf c^{a,\delta_{\mathrm m}}
\in\arg\max_{\mathbf c\in\mathcal C_{\mathrm m,\delta_{\mathrm m}}(T)}
\pi_{\mathbf c}^a(T,\mathbf c,N),
\qquad
\mathbf c^a
\in\arg\max_{\mathbf c\in\mathbb Z_+^d}
\pi_{\mathbf c}^a(T,\mathbf c,N).
\]
The first pair are the restricted optimizers and the second pair are the
unrestricted optimizers; all four depend on \((T,N)\).

On this domain, \cref{lem:mr-constant-online-total} gives expected total terminal inventory at most \(K_{\mathrm m}T^\alpha\), uniformly in \(n,T,\mathbf c\), for large \(T\).

Define \(h_{\max}\triangleq\max_{\ell\in[d]}h_\ell\).

\begin{theorem}[Multi-resource constant-order transfer]
\label{thm:extend-reg-total-constant}
Suppose \cref{ass:online-algorithm} holds with \(0\leq\alpha<1\).
\begin{enumerate}[label=(\alph*)]
\item For every fixed \(\delta_{\mathrm m}\in(0,1-\lambda_0)\), all
sufficiently large \(T\), and every finite \(N\geq1\),
\[
\pi_\mathbf{c}^{\text{off}}(T,\mathbf c^{\text{off},\delta_{\mathrm m}},N)
-\pi_\mathbf{c}^{\text{on}}(T,\mathbf c^{\text{on},\delta_{\mathrm m}},N)
\leq\left[C_1+(r_{\max}+h_{\max})K_{\mathrm m}\right]T^\alpha.
\]
\item Suppose also that \cref{ass:mr-single-cycle-sm2} holds. If
\(0\leq\alpha<1/2\), additionally impose \cref{ass:mr-common-order}; if
\(1/2\leq\alpha<1\), no common resource order is needed. There are constants
\(K,T_0<\infty\), independent of \(N\), such that
for all \(T\geq T_0\) and every finite \(N\geq1\),
\begin{equation}\label{eq:mr-constant-finite-unrestricted-et}
\left|
\pi_\mathbf c^{\mathrm{off}}(T,\mathbf c^{\mathrm{off}},N)
-\pi_\mathbf c^{\mathrm{on}}(T,\mathbf c^{\mathrm{on}},N)
\right|\leq KT^\alpha.
\end{equation}
\item Without the extra conditions in part (b), for every finite \(N\geq1\),
the absolute difference between the two unrestricted optimized profits is at
most \(r_{\max}(1-\lambda_0)T\).
\end{enumerate}
\end{theorem}

The additive-reward, nested-use scalar Bayes Selector in
\cref{prop:mr-additive-fbs-validation} satisfies part (b) with \(\alpha=0\),
and hence gives an unrestricted, separately optimized \(O(1)\) comparison
for every finite \(N\).

\emph{Proof idea.}
For part (a), aggregate inventory obeys the same perturbed Lindley recursion
as in the single-resource model, and linear demand slack gives part (a).
For part (b), the quadratic condition controls near-critical carryover.
Common ordering yields scalar prefix recursions when \(\alpha<1/2\); when
\(\alpha\geq1/2\), a type-dedicated candidate with \(O(T^\alpha)\) slack
provides the comparison. Part (c) follows from the common profit envelope.

\section{Conclusion}
\label{sec:conclusion}

We study how fulfillment loss within a cycle propagates through repeated
replenishment. Under base-stock replenishment, order adjustments absorb past
allocation differences. Under constant orders, the carryover follows a
perturbed Lindley recursion, and its magnitude is determined by demand slack
and the moments of the single-cycle count loss. The resulting bounds separate
the replenishment and fulfillment components of the policy comparison. For
the Bayes Selector, the constant-order transfer is bounded independently of
the cycle length.

For multiple resources, the transfer also depends on the allocation
structure because aggregate inventory does not determine the remaining
resource mix. We further formulate fixed-policy full-information benchmarks
and analyze a base-stock quota policy. Analytical bounds for the offline
replenishment gap in general multi-resource systems remain open.

\bibliographystyle{informs2014}
\bibliography{secretary}

\ECRUNAUTHOR{}
\ECAUpunct{}
\ECSwitch
\SingleSpacedXI

\ECHead{Electronic Companion for Online Order Fulfillment with Replenishment}
This e-companion contains the detailed proofs and additional technical results.

\section{Proofs in \cref{sec:model}}

\proof{Proof of \cref{lem:inv-diff}.}
Let \(D(\mathcal H)\) be the total number of requests in the cycle.  Since every
reward is strictly positive, the myopic offline allocation fulfills
\(\min\{I,D(\mathcal H)\}\) requests.  Hence its total number of fulfilled
requests is at least that of any feasible online allocation and
\[
\epsilon^n(I,\mathcal H)
=f^{n,\mathrm{off}}(I)-f^{n,\mathrm{on}}(I)
=I_{T+1}^{n,\mathrm{on}}(I,\mathcal H)
 -I_{T+1}^{n,\mathrm{off}}(I,\mathcal H)\geq0 .
\]
To relate this count difference to reward regret, augment the online allocation
by any \(\epsilon^n(I,\mathcal H)\) additional feasible requests.  Each added
request is worth at least \(r_1\), while the offline allocation maximizes reward
among all feasible allocations.  Therefore, pathwise,
\[
V^{\mathrm{off}}(T,I,\mathcal H)-V^{\mathrm{on}}(T,I,\mathcal H)
\geq r_1\epsilon^n(I,\mathcal H).
\]
Taking conditional expectations and applying
\cref{ass:online-algorithm} gives
\[
\mathbb E[\epsilon^n(I,\mathcal H)\mid\mathcal F_{\mathrm{pre}}]
\leq (C_1/r_1)T^\alpha .
\]
Thus the claim holds with \(C_2=C_1/r_1\).
\hfill \Halmos
\endproof

\section{Proofs in \cref{subsec:base-replenishment}}

\subsection{Shadow-capped implementation}
\label{app:shadow-base}

Alongside its physical system, the shadow-capped policy keeps a myopic-offline
system with the same base-stock level \(S\). At the beginning of cycle \(n\),
the shadow inventory \(I_1^{n,\mathrm{off}}\) is computed from demand observed
in completed cycles. If the physical online inventory is larger, the excess
\(I_1^{n,\mathrm{on}}-I_1^{n,\mathrm{off}}\) is protected during the current
cycle. The single-cycle online algorithm is then run with virtual inventory
\(I_1^{n,\mathrm{off}}\). The realized demand updates the shadow system at the
end of the cycle. The policy is nonanticipating because it does not use future
arrivals in the current cycle.

\begin{lemma}
\label{lem:inv-diff-base}
Under \cref{ass:online-algorithm}, suppose the physical and shadow systems
start from the same complete feasible augmented state and use the same
base-stock level \(S\). For every cycle \(n\) and every demand path,
\[
I_1^{n,\mathrm{on}}(S)-I_1^{n,\mathrm{off}}(S)\geq0,
\qquad
\mathbb E\!\left[
 I_1^{n,\mathrm{on}}(S)-I_1^{n,\mathrm{off}}(S)
\right]\leq LC_2T^\alpha,
\]
and
\[
I_{T+1}^{n,\mathrm{on}}(S)-I_{T+1}^{n,\mathrm{off}}(S)\geq0,
\qquad
\mathbb E\!\left[
 I_{T+1}^{n,\mathrm{on}}(S)-I_{T+1}^{n,\mathrm{off}}(S)
\right]\leq(L+1)C_2T^\alpha,
\]
where \(C_2\) is defined in \cref{lem:inv-diff}.
\end{lemma}

\begin{proposition}\label{prop:shadow-base-transfer}
Under \cref{ass:online-algorithm}, use the shadow-capped implementation and
the common initialization in \cref{sec:model}. Let \(S^{\mathrm{shadow}}\)
maximize its finite-\(N\) base-stock profit. Then
\[
\pi_\mathbf{b}^{\mathrm{off}}(T,S^{\mathrm{off}},N)
-\pi_\mathbf{b}^{\mathrm{shadow}}(T,S^{\mathrm{shadow}},N)
\leq\bigl[C_1+h(L+1)C_2\bigr]T^\alpha
=\mathcal O(T^\alpha).
\]
\end{proposition}

\proof{Proof of \cref{lem:inv-diff-base}.} 
Write
\[
\Delta I_t^n=I_t^{n,\mathrm{on}}-I_t^{n,\mathrm{off}},
\qquad
\Delta q^n=q^{n,\mathrm{on}}-q^{n,\mathrm{off}},
\]
and let \(\epsilon^n\) be the same-start terminal-inventory gap when the
single-cycle online algorithm and the shadow offline algorithm both start from
\(I_1^{n,\mathrm{off}}\). By the shadow cap and \cref{lem:inv-diff},
\begin{equation}\label{eq:shadow-ending-gap}
\Delta I_{T+1}^n=\Delta I_1^n+\epsilon^n,\qquad
\epsilon^n\geq0,\qquad
\mathbb E[\epsilon^n\mid\mathcal F_{\mathrm{pre}}]\leq C_2T^\alpha .
\end{equation}

For \(L\geq2\), we induct on the following three claims:
\[
\begin{array}{ll}
\textnormal{(a)}&
\Delta I_1^n\geq0,\quad \mathbb E[\Delta I_1^n]\leq LC_2T^\alpha,\\
\textnormal{(b)}&
\Delta q^{n+\ell}\leq0,\quad
\mathbb E[\Delta q^{n+\ell}]\geq-C_2T^\alpha,\quad 1\leq\ell<L,\\
\textnormal{(c)}&
P_n\triangleq\Delta I_1^n+\sum_{\ell=1}^{L-1}\Delta q^{n+\ell}
\geq0,\quad \mathbb E[P_n]\leq C_2T^\alpha .
\end{array}
\]
They hold initially because the two pipelines coincide. The base-stock order
equation gives
\begin{equation}\label{eq:shadow-order-gap}
\Delta q^{n+L}=-P_n,
\end{equation}
which proves the newly shifted instance of (b). Next,
\[
\Delta I_1^{n+1}
=\Delta I_{T+1}^n+\Delta q^{n+1}
=\Delta I_1^n+\Delta q^{n+1}+\epsilon^n\geq0,
\]
where (b)--(c) imply
\(\Delta I_1^n+\Delta q^{n+1}
=P_n-\sum_{\ell=2}^{L-1}\Delta q^{n+\ell}\geq0\).
The same identities and expectation bounds give
\[
\mathbb E[\Delta I_1^{n+1}]
\leq C_2T^\alpha+(L-2)C_2T^\alpha+C_2T^\alpha
=LC_2T^\alpha .
\]
Finally, \cref{eq:shadow-ending-gap,eq:shadow-order-gap} imply
\[
P_{n+1}=\epsilon^n,
\]
proving (c). For \(L\geq2\), the orders remain feasible throughout the
induction. For \(n\geq2\), the offline order placed in cycle \(n\) equals the number served
in cycle \(n-1\), while
\(P_n=\epsilon^{n-1}\) is no larger than that number. Hence
\(q^{n+L,\mathrm{on}}=q^{n+L,\mathrm{off}}-P_n\geq0\); feasibility in the
first cycle follows from the common initial state. This completes that
induction. For \(L=1\), the same calculation
reduces to
\(\Delta q^{n+1}=-\Delta I_1^n\) and
\(\Delta I_1^{n+1}=\epsilon^n\). The same comparison with the preceding
offline sales shows that the online order is nonnegative.
The end-of-cycle claims follow immediately from
\eqref{eq:shadow-ending-gap} and (a).
\hfill \Halmos
\endproof

\proof{Proof of \cref{prop:shadow-base-transfer}.}
Fix \(S\). By construction, the shadow-capped policy earns exactly the reward
of the single-cycle online algorithm initialized with
\(I_1^{n,\mathrm{off}}\); its extra physical units are protected. Conditional
uniformity in \cref{ass:online-algorithm} therefore bounds the expected reward
gap in every cycle by \(C_1T^\alpha\). By
\cref{lem:inv-diff-base}, the online end inventory is no smaller and its
expected excess holding cost is at most \(h(L+1)C_2T^\alpha\). Averaging over
cycles gives
\[
\pi_{\mathbf b}^{\mathrm{off}}(T,S,N)
-\pi_{\mathbf b}^{\mathrm{shadow}}(T,S,N)
\leq [C_1+h(L+1)C_2]T^\alpha .
\]

For fixed \(T,N\), the offline and shadow optimizers exist. Under the
initialization in \cref{sec:model}, the cycle-1 ending inventory is at least \((a_0-T)^+\),
where \(a_0=\lfloor S/(L+1)\rfloor\), while expected reward is bounded
by \(r_MT\). Thus the finite-horizon objective tends to \(-\infty\) as
\(S\to\infty\). Finally, optimality of \(S^{\mathrm{shadow}}\) gives
\[
\pi_{\mathbf b}^{\mathrm{off}}(T,S^{\mathrm{off}},N)
-\pi_{\mathbf b}^{\mathrm{shadow}}(T,S^{\mathrm{shadow}},N)
\leq
\pi_{\mathbf b}^{\mathrm{off}}(T,S^{\mathrm{off}},N)
-\pi_{\mathbf b}^{\mathrm{shadow}}(T,S^{\mathrm{off}},N),
\]
so the fixed-\(S\) bound with \(S=S^{\mathrm{off}}\) completes the proof.
\hfill \Halmos 
\endproof

\proof{Proof of \cref{prop:uncapped-direct-bound}.}
For a post-order state \(z=(I,q_1,\ldots,q_L)\), let
\(C_k(z)=I+\sum_{i=1}^kq_i\) and define
\[
d(z,w)=\max_k\{C_k(z)-C_k(w)\}-\min_k\{C_k(z)-C_k(w)\}.
\]
Since \(C_L(z)=C_L(w)=S\), \(d\) is a metric and
\(|I(z)-I(w)|\le d(z,w)\). Under the work-conserving transition
\(\mathcal K_Dz=(I-g_z+q_1,q_2,\ldots,q_L,g_z)\), where
\(g_z=\min\{I,D\}\), the new cumulative differences have the same span as
\(\{\Delta_1,\ldots,\Delta_L,g_z-g_w\}\). Because \(g_z-g_w\) lies between
\(0\) and \(\Delta_0\),
\begin{equation}\label{eq:work-conserving-nonexpansive}
d(\mathcal K_Dz,\mathcal K_Dw)\le d(z,w).
\end{equation}
If \(\mathcal A_{\mathcal H,\xi}\) is the uncapped online transition, its
deviation from \(\mathcal K_D\) is the same-start count loss \(\epsilon\):
\begin{equation}\label{eq:online-one-step-span-error}
d(\mathcal A_{\mathcal H,\xi}z,\mathcal K_Dz)=\epsilon,
\qquad \mathbb E[\epsilon\mid\mathcal F_{\mathrm{pre}}]\le C_2T^\alpha.
\end{equation}

Set \(m=2L\) and
\(
\mathcal E_S=\bigcap_{k=1}^L\{\sum_{i=k}^{k+L}D_i\le S\}.
\)
On \(\mathcal E_S\), every work-conserving chain serves all demand in cycles
\(L+1,\ldots,2L\). Its initial pipeline has then shifted out, and all chains
end at
\(
(S-\sum_{i=L+1}^{2L}D_i,D_{L+1},\ldots,D_{2L}).
\)
Thus the \(m\)-step work-conserving map contracts \(d\) to zero on
\(\mathcal E_S\) and is nonexpansive otherwise. If
\(S\ge(L+1)\mu T-K_S\sqrt T\), the event
\[
\bigcap_{i=1}^{2L}\{D_i\le\lfloor\mu T-(K_S+1)\sqrt T/(L+1)\rfloor\}
\subseteq\mathcal E_S
\]
has probability bounded below by some \(\eta_{K_S}>0\) for all sufficiently
large \(T\), by the binomial central limit theorem.

Couple online states \(X_n\) and offline work-conserving states \(Y_n\) by
common arrivals. Applying \eqref{eq:work-conserving-nonexpansive} and
\eqref{eq:online-one-step-span-error} within each block gives
\[
d(X_{n+m},Y_{n+m})
\le \sum_{i=0}^{m-1}\epsilon_{n+i}
+\mathbb I\{\mathcal E_S^c\}d(X_n,Y_n).
\]
Future block demand is independent of the block-start history. Conditional
expectation and iteration from the common initial state therefore yield
\begin{equation}\label{eq:uncapped-uniform-span}
\sup_{n\ge1}\mathbb E d(X_n,Y_n)
\le 2LC_2(1+\eta_{K_S}^{-1})T^\alpha.
\end{equation}
The one-step estimate supplies the same bound between block boundaries.

The offline one-cycle reward is \(r_M\)-Lipschitz in starting inventory.
Insert its value at the online start state and use
\cref{ass:online-algorithm} and \eqref{eq:uncapped-uniform-span}. Moreover,
the terminal-inventory difference is at most \(d(X_n,Y_n)+\epsilon_n\).
Hence the expected cycle profit gap is bounded by
\[
\{C_1+hC_2+(r_M+h)2LC_2(1+\eta_{K_S}^{-1})\}T^\alpha.
\]
Averaging over the \(N\) cycles proves the result.
\hfill\Halmos
\endproof

\begin{proposition}[Long-run uncapped transfer]
\label{prop:uncapped-base-longrun}
Suppose the relevant fixed-\(S\) long-run averages exist and the optimized
values below are well defined. Under the conditions of
\cref{prop:uncapped-direct-bound},
\begin{equation}\label{eq:uncapped-direct-bound}
\pi_{\mathbf b}^{\mathrm{off},*}(T,\infty)
-\pi_{\mathbf b}^{\mathrm{uncap},*}(T,\infty)
\leq C_{\mathrm{uncap}}T^\alpha,
\end{equation}
where \(C_{\mathrm{uncap}}<\infty\) is independent of \(T\).
\end{proposition}

\proof{Proof of \cref{prop:uncapped-base-longrun,cor:fbs-base-ordinary}.}
Let \(\overline S=(L+1)\mu T\),
\(S_0=\lceil\overline S\rceil\), and
\(\overline R_T=T\sum_jr_j\lambda_j\). At \(S_0\), the work-conserving rolling
identity gives, with \(U_n=D^n-F_n\),
\[
U_n\le\left[\sum_{k=0}^L(D^{n-k}-\mu T)\right]^+,\qquad
I_{T+1}^n=S_0-\sum_{k=0}^LF_{n-k}.
\]
Writing \(\sigma=\sqrt{(L+1)\mu(1-\mu)}\), Cauchy--Schwarz gives
\(\mathbb EU_n\le\sigma\sqrt T\) and
\(\mathbb EI_{T+1}^n\le1+(L+1)\sigma\sqrt T\). Hence
\[
\pi_{\mathbf b}^{\mathrm{off}}(T,S_0,\infty)
\ge\overline R_T-A_{\mathrm{cand}}\sqrt T,\qquad
A_{\mathrm{cand}}=[r_M+h(L+1)]\sigma+h.
\]
For any stationary \(S\), flow balance gives
\(\mathbb EI_{T+1}=S-(L+1)\mathbb EF\). Therefore
\[
\pi_{\mathbf b}^{\mathrm{off}}(S)\le
\begin{cases}
\overline R_T-r_1(\overline S-S)/(L+1),&S<\overline S,\\
\overline R_T-h(S-\overline S),&S>\overline S.
\end{cases}
\]
Comparison with the candidate proves existence of a finite offline optimizer
and \(|S_\infty^{\mathrm{off}}-\overline S|\le K_*^{\mathbf b}\sqrt T\) for a
fixed \(K_*^{\mathbf b}\). Thus
\cref{prop:uncapped-direct-bound} applies at this optimizer. A different
initial state contributes only a geometrically decaying term under the block
contraction, whose horizon average vanishes. Passing to the limit and using
online optimality yields
\[
\pi_{\mathbf b}^{\mathrm{off},*}
-\pi_{\mathbf b}^{\mathrm{uncap},*}
\le\pi_{\mathbf b}^{\mathrm{off}}(S_\infty^{\mathrm{off}})
-\pi_{\mathbf b}^{\mathrm{uncap}}(S_\infty^{\mathrm{off}})
\le C_{\mathrm{uncap}}T^\alpha.
\]

For the cycle-restarted Bayes Selector and fixed \(S\), the post-order state space is
finite. \(L\) zero-demand cycles reach \((S,0,\ldots,0)\), which has a
positive-probability self-loop, so both fixed-\(S\) long-run averages exist.
After the initial pipeline departs, every order is at most \(T\); hence
\[
\pi_{\mathbf b}^{\mathrm{uncap}}(T,S,\infty)
\le r_MT-h[S-(L+1)T]^+\longrightarrow-\infty.
\]
The Bayes Selector optimizer is therefore finite. Finally,
\cref{lem:fbs-single-cycle-et} supplies \(\alpha=0\), proving the corollary.
\hfill\Halmos
\endproof

\proof{Proof of \cref{thm:reg-total-base}.}
Let \(P=L+1\), \(\mu=1-\lambda_0\), and
\(\overline R_T=T\sum_jr_j\lambda_j\). If \(F_n\) is cycle-\(n\) sales,
the order-up-to identity gives \(q^{n+L+1}=F_n\). A sold unit cannot be sold
again for \(P\) cycles. Thus, writing \(N=qP+r\), \(0\le r<P\), and denoting
the initial readiness buckets by \(a_k\),
\begin{equation}\label{eq:finite-cooldown-bound}
\sum_{n=1}^NF_n\le qS+\sum_{k=0}^{r-1}a_k
\le\frac{NS}{P}+L.
\end{equation}
The last inequality uses the initialization in \cref{sec:model}. Since every
lost request forfeits at least \(r_1\),
\begin{equation}\label{eq:finite-offline-low-S-upper}
\pi_{\mathbf b}^{\mathrm{off}}(T,S,N)
\le\overline R_T-r_1\left(\mu T-\frac SP-\frac LN\right).
\end{equation}

For a comparison target take \(s_0=\lceil\mu T\rceil\) and \(S_0=Ps_0\),
so every initial bucket equals \(s_0\). Encode the initialization by
\(D_m=F_m=s_0\) for \(m\le0\). With lost demand
\(U_n=D_n-F_n\) and ending inventory \(J_n\), the rolling identities give
\[
I_1^n=S_0-\sum_{i=1}^LF_{n-i},\qquad
J_n=S_0-\sum_{i=0}^LF_{n-i},\qquad
U_n\le\left[\sum_{i=0}^LD_{n-i}-S_0\right]^+.
\]
Only actual cycles contribute to the last centered sum. Independence and
Cauchy--Schwarz give \(\mathbb EU_n\le\sigma\sqrt T\), where
\(\sigma=\sqrt{P\mu(1-\mu)}\), and the second identity gives
\(\mathbb EJ_n\le P+P\sigma\sqrt T\). Hence, with
\(A=r_M\sigma+hP(1+\sigma)\),
\begin{equation}\label{eq:finite-balanced-candidate}
\pi_{\mathbf b}^{\mathrm{off}}(T,S_0,N)
\ge\overline R_T-A\sqrt T,\qquad N\ge1.
\end{equation}
Comparison with \eqref{eq:finite-offline-low-S-upper} at an offline
maximizer yields
\begin{equation}\label{eq:finite-optimizer-low-band}
S^{\mathrm{off}}(T,N)\ge P\mu T-K\sqrt T,\qquad
K=P(A/r_1+L+1),
\end{equation}
uniformly in \(N\). The maximizers exist because reward is bounded, whereas
first-cycle holding cost diverges as \(S\to\infty\).

By \eqref{eq:finite-optimizer-low-band},
\cref{prop:uncapped-direct-bound} applies at \(S^{\mathrm{off}}(T,N)\), with
a constant independent of \(T,N\). Online optimality then gives
\[
\pi_{\mathbf b}^{\mathrm{off}}(S^{\mathrm{off}})
-\pi_{\mathbf b}^{\mathrm{on}}(S^{\mathrm{on}})
\le\pi_{\mathbf b}^{\mathrm{off}}(S^{\mathrm{off}})
-\pi_{\mathbf b}^{\mathrm{on}}(S^{\mathrm{off}})
\le C_KT^\alpha.
\]
\hfill\Halmos
\endproof

\section{Proofs in \cref{subsec:constant-replenishment}}

\subsection{Additional technical statements}

The following results are used in the proof of
\cref{thm:reg-total-constant}.

\begin{lemma}\label{lem:constant-off-more}
Under \cref{ass:online-algorithm}, suppose the online and myopic-offline
systems start from the same inventory and use the same constant order \(c\).
For every cycle \(n\) and demand path \(\mathcal H\),
\[
I_{T+1}^{n,\mathrm{off}}(c,\mathcal H)
\leq I_{T+1}^{n,\mathrm{on}}(c,\mathcal H).
\]
\end{lemma}

\begin{lemma}[Single-cycle tail bound for the Bayes Selector]
\label{lem:fbs-single-cycle-et}
Suppose the single-resource Bayes Selector defined in \cref{sec:model}
is restarted at the beginning of every cycle. Arrivals within each
cycle follow the fresh multinomial model in this paper, and a fixed
tie-breaking rule is used. There are constants
\(\eta,K_{\epsilon}>0\), independent of \(I,T,n\), and the pre-cycle
history, such that
\[
\mathbb E\!\left[
  \exp\{\eta\epsilon^n(I,\mathcal H)\}
  \middle|\mathcal F_{\mathrm{pre}}
\right]\leq K_{\epsilon}.
\]
Consequently, the Bayes Selector satisfies
\cref{ass:online-algorithm,ass:single-cycle-sm2} with \(\alpha=0\), as well
as \cref{ass:single-cycle-et} below.
\end{lemma}

\begin{lemma}[Linear-slack carryover bound]\label{lem:Inv-on-bound}
Under \cref{ass:online-algorithm}, suppose \(\alpha<1\) and
\(0\leq c\leq(\mu-\delta)T\) for some \(\delta>0\). There are constants
\(K_{\mathrm c},T_1<\infty\), depending only on the primitive parameters,
\(C_2,\alpha\), and \(\delta\), such that, for every \(T\geq T_1\) and cycle
\(n\),
\[
\mathbb E\!\left[I_{T+1}^{n,\mathrm{on}}(c,\mathcal H)\right]
\leq K_{\mathrm c}T^\alpha .
\]
\end{lemma}

\begin{lemma}[Average near-critical inventory-gap transfer]
\label{lem:constant-sqrt-transfer-average}
Under \cref{ass:online-algorithm,ass:single-cycle-sm2}, suppose
\(\alpha<1/2\) and fix \(\kappa>0\). Couple the online and
myopic-offline constant-order systems through
\eqref{eq:constant-coupled-recursions}, from zero carryover. There are
constants \(\bar K_\Delta,\bar T_\kappa<\infty\), independent of \(c,T,N\),
such that, whenever \(T\geq\bar T_\kappa\) and
\(0\leq c\leq\mu T-\kappa\sqrt T\),
\begin{equation}\label{eq:constant-sqrt-inventory-gap-average}
 \frac1N\sum_{n=1}^N\mathbb E[X_n-Y_n]
 \leq\bar K_\Delta T^\alpha
 \qquad\text{for every }N\geq1.
\end{equation}
\end{lemma}

\begin{lemma}[Location of the myopic-offline optimizer]
\label{lem:constant-off-optimizer-slack}
For a stable integer constant order \(c<\mu T\), let \(H_T(c)\) and
\(Q_T(c)\) be the stationary expected ending inventory and fulfilled reward
under myopic-offline fulfillment. For \(p>0\), let \(c_p\) be the smallest
minimizer of
\[
J_p(c)=hH_T(c)+p(\mu T-c),
\qquad 0\leq c<\mu T,
\]
and let \(c_\infty^{\mathrm{off}}\) be the smallest maximizer of
\(Q_T(c)-hH_T(c)\) on the same domain. Then
\begin{equation}\label{eq:constant-optimizer-sandwich}
c_{r_1}\leq c_\infty^{\mathrm{off}}\leq c_{r_M}.
\end{equation}
Moreover, for fixed primitives with \(0<\mu<1\), there are constants
\(0<\kappa_-<\kappa_+<\infty\) such that, for all sufficiently large \(T\),
\begin{equation}\label{eq:constant-off-sqrt-slack}
\kappa_-\sqrt T
\leq\mu T-c_\infty^{\mathrm{off}}
\leq\kappa_+\sqrt T.
\end{equation}
\end{lemma}

We also use the standard finite-horizon Spitzer identity for the
myopic-offline inventory.

\begin{lemma}[Theorem~1 in \cite{kingman1962}]
\label{lem:Inv_off_expectation}
For all \(c\leq\mathbb E[D^n]\) and \(n\geq1\),
\[
\mathbb E\!\left[I_{T+1}^{n,\mathrm{off}}(c,\mathcal H)\right]
=\sum_{j=1}^n\frac1j
\mathbb E\!\left[\left(jc-\sum_{i=1}^jD^i\right)^+\right].
\]
\end{lemma}

The pointwise near-critical comparison uses the following stronger
single-cycle condition. It is not required for the average-in-horizon
transfer in \cref{thm:reg-total-constant}(b),(d).

\begin{assumption}[Uniform single-cycle exponential count loss]
\label{ass:single-cycle-et}
For the same exponent \(\alpha\) as in \cref{ass:online-algorithm}, suppose
\(\alpha<1\) and there are constants \(\eta>0\) and
\(K_{\epsilon}<\infty\), independent of \(I,T,n\), and the pre-cycle
history, such that
\[
\mathbb E\!\left[
 \exp\!\left\{\eta\epsilon^n(I,\mathcal H)/T^\alpha\right\}
 \middle|\mathcal F_{\mathrm{pre}}
\right]\leq K_{\epsilon}.
\]
Demand in the new cycle is independent of \(\mathcal F_{\mathrm{pre}}\).
\end{assumption}

This condition implies \cref{ass:single-cycle-sm2}, possibly with a different
constant.

\begin{lemma}[Pointwise near-critical inventory-gap transfer]
\label{lem:constant-sqrt-transfer}
Under \cref{ass:online-algorithm,ass:single-cycle-et}, suppose
\(\alpha<1/2\) and fix \(\kappa>0\). Couple the online and myopic-offline
constant-order systems using the same demands, start both with zero carryover,
and let
\[
X_n=I_{T+1}^{n,\mathrm{on}},\qquad
Y_n=I_{T+1}^{n,\mathrm{off}}.
\]
There are constants \(K_\Delta,T_\kappa<\infty\), independent of \(c,T,n\),
such that, whenever \(T\geq T_\kappa\) and
\(0\leq c\leq\mu T-\kappa\sqrt T\),
\begin{equation}\label{eq:constant-sqrt-inventory-gap}
0\leq\mathbb E[X_n-Y_n]\leq K_\Delta T^\alpha
\qquad\text{for every }n\geq1.
\end{equation}
\end{lemma}

The next proposition records why the first-moment single-cycle assumption
alone cannot support the unrestricted optimized transfer.

\begin{proposition}[A first moment alone does not support the unrestricted
transfer]\label{prop:first-moment-insufficient}
For every \(0<\alpha<1\), there is a single-resource, two-type online
algorithm satisfying \cref{ass:online-algorithm} with exponent \(\alpha\),
and a sequence of finite horizons \(N=N(T)\). Let
\(c^{\mathrm{off}}(T,N)\) and \(c^{\mathrm{on}}(T,N)\) maximize the two
finite-horizon objectives over all integer orders \(c\geq0\). Then
\[
\pi_{\mathbf c}^{\mathrm{off}}(T,c^{\mathrm{off}}(T,N),N)
-\pi_{\mathbf c}^{\mathrm{on}}(T,c^{\mathrm{on}}(T,N),N)
=\Omega\!\left(T^{(1+\alpha)/2}\right).
\]
For \(\alpha=0\), an analogous construction has uniformly bounded
one-cycle expected regret but an optimized finite-horizon gap
\(\Omega(\sqrt T)\). Thus the first-moment assumption alone cannot imply
the unrestricted \(O(T^\alpha)\) result. The construction is cyclewise
i.i.d., and the gap is caused by rare single-cycle losses of order \(T\).
\end{proposition}

\subsection{Proofs}

\proof{Proof of \cref{lem:constant-off-more}.}
The initial inventories coincide. Whenever
\(I_1^{n,\mathrm{off}}\leq I_1^{n,\mathrm{on}}\), the myopic offline
algorithm is work conserving because rewards are positive, whereas the online
algorithm serves at most \(\min\{I_1^{n,\mathrm{on}},D^n\}\) customers.
Therefore,
\[
I_{T+1}^{n,\mathrm{off}}
=(I_1^{n,\mathrm{off}}-D^n)^+
\leq (I_1^{n,\mathrm{on}}-D^n)^+
\leq I_{T+1}^{n,\mathrm{on}}.
\]
Adding the common order \(c\) preserves the inequality at the start of cycle
\(n+1\), and induction proves the claim. \hfill \Halmos
\endproof

\proof{Proof of \cref{lem:fbs-single-cycle-et}.}
Let \(\operatorname{Reg}_T(I)\) be the one-cycle reward regret of the Bayes Selector defined in \cref{sec:model}
from inventory \(I\). The multinomial high-probability regret bound of
\citet[Section~4.3, Proposition~3]{vera2021bayesian} yields
primitive-dependent constants
\(A,b>0\) and \(x_0<\infty\), independent of \(T\) and \(I\), such that
\[
\Pr\{\operatorname{Reg}_T(I)>x\}\leq Ae^{-bx},
\qquad x\geq x_0.
\]
The augmentation argument in the proof of \cref{lem:inv-diff} is pathwise and
gives
\[
r_1\epsilon^n(I,\mathcal H)\leq\operatorname{Reg}_T(I).
\]
Consequently, for any \(0<\eta<b\,r_1\), integration of the preceding tail
bound gives
\(\sup_{T,I}\mathbb E[e^{\eta\epsilon^n(I,\mathcal H)}]<\infty\).
When we condition on the pre-cycle history, the starting inventory is fixed
and the new multinomial sample is independent. Restarting this policy therefore
gives the same uniform bound conditionally. This exponential moment implies
\cref{ass:single-cycle-sm2} with \(\alpha=0\), while the multinomial
expected-regret bound of \citet{vera2021bayesian} gives
\cref{ass:online-algorithm} with \(\alpha=0\). The no-arrival outcome in our
model is included as a null type
whose only action leaves inventory and reward unchanged; it contributes
neither disagreement nor regret, so the cited multinomial argument is
unchanged.
\hfill\Halmos
\endproof

\proof{Proof of \cref{lem:Inv-on-bound}.}
Set \(X_n=I_{T+1}^{n,\mathrm{on}}\), \(X_0=0\), and
\(b_T=C_2T^\alpha\).  The recursion above and the conditional version of
\cref{lem:inv-diff} give
\[
X_n=(X_{n-1}+c-D^n)^++\epsilon^n,\qquad
0\leq\epsilon^n\leq T,\qquad
\mathbb E[\epsilon^n\mid\mathcal F_{n-1}]\leq b_T .
\]
Let \(g_T=\mu T-c\), \(a_T=g_T/2\), and
\(B_n=\{D^n<c+a_T\}\).  Since \(g_T\geq\delta T\), Hoeffding's inequality
implies
\[
p_T\triangleq\Pr(B_n)\leq \exp(-\delta^2T/2).
\]
For every realization,
\[
X_n\leq (X_{n-1}-a_T)^++A_n,\qquad
A_n\triangleq\epsilon^n+2T\mathbf 1_{B_n}.
\]
In particular, \(0\leq A_n\leq3T\) and
\[
\mathbb E[A_n\mid\mathcal F_{n-1}]
\leq m_T\triangleq b_T+2Tp_T .
\]

Fix \(s>0\) and let \(V(x)=e^{sx/T}-1\).  Convexity on \([0,3T]\) yields
\[
\mathbb E[V(A_n)\mid\mathcal F_{n-1}]
\leq \kappa_T
\triangleq\frac{m_T}{3T}(e^{3s}-1).
\]
Moreover,
\[
V((x-a_T)^+)\leq q_TV(x),\qquad
q_T=e^{-sa_T/T}\leq e^{-s\delta/2}<1.
\]
Using \(1+V(u+v)=(1+V(u))(1+V(v))\), we obtain the drift inequality
\[
\mathbb E[V(X_n)\mid\mathcal F_{n-1}]
\leq \rho_TV(X_{n-1})+\kappa_T,
\qquad \rho_T=(1+\kappa_T)q_T.
\]
Because \(\alpha<1\), \(\kappa_T\to0\); hence \(\rho_T<1\) for all
sufficiently large \(T\), uniformly over admissible \(c\). Iteration from
\(X_0=0\) gives
\[
\sup_{n\geq1}\mathbb E[V(X_n)]
\leq\frac{\kappa_T}{1-\rho_T}.
\]
Finally, \(V(x)\geq sx/T\), so
\[
\sup_{n\geq1}\mathbb E[X_n]
\leq\frac{T}{s}\frac{\kappa_T}{1-\rho_T}
=O(T^\alpha)+O\!\left(Te^{-\delta^2T/2}\right)
\leq K_{\mathrm c}T^\alpha .
\]
\hfill \Halmos 
\endproof

\proof{Proof of \cref{lem:constant-sqrt-transfer}.}
Put \(Z_n=c-D^n\).  The coupled recursions are
\[
X_n=(X_{n-1}+Z_n)^++\epsilon_n,
\qquad
Y_n=(Y_{n-1}+Z_n)^+,
\qquad X_0=Y_0=0.
\]
They imply \(X_n\geq Y_n\) pathwise.  We first obtain a conditional
exponential drift bound without assuming that \(\epsilon_n\) and \(D^n\) are
independent. Fix \(\theta\in(0,\kappa]\). For all sufficiently large \(T\),
let
\[
s_T=\frac{\eta T^{1/2-\alpha}}{\theta}\geq2,
\qquad
p_T=\frac{s_T}{s_T-1}\leq2.
\]
Conditional H\"older's inequality, Hoeffding's lemma for
\(D^n\sim\operatorname{Binomial}(T,\mu)\), and
\cref{ass:single-cycle-et} give
\begin{align*}
&\log\mathbb E\!\left[
 e^{\theta(Z_n+\epsilon_n)/\sqrt T}
 \middle|\mathcal F_{n-1}\right]\\
&\quad\leq
-\frac{\theta(\mu T-c)}{\sqrt T}
+\frac{p_T\theta^2}{8}
+\frac{\theta\log K_\epsilon}{\eta}T^{\alpha-1/2}.
\end{align*}
Uniformly over \(c\leq\mu T-\kappa\sqrt T\), the right-hand side is at most
\(-\kappa\theta/2\) for all sufficiently large \(T\). Thus, with
\(q=e^{-\kappa\theta/2}<1\),
\begin{equation}\label{eq:constant-busy-mgf}
\mathbb E\!\left[
 e^{\theta(Z_n+\epsilon_n)/\sqrt T}
 \middle|\mathcal F_{n-1}\right]\leq q.
\end{equation}
Also, log-convexity of moment generating functions gives a constant
\(K_0<\infty\) such that
\begin{equation}\label{eq:constant-epsilon-small-mgf}
\mathbb E\!\left[e^{\theta\epsilon_n/\sqrt T}
\middle|\mathcal F_{n-1}\right]\leq K_0.
\end{equation}

Call cycle \(t\) a reset if
\(\mathsf R_t=\{X_{t-1}+Z_t\leq0\}\), and declare \(t=0\) an artificial
reset with \(\epsilon_0=0\). On \(\mathsf R_t\), one has
\(X_t=\epsilon_t\) and \(Y_t=0\). If no reset occurs during the next \(m\)
cycles, then
\[
\epsilon_t+\sum_{i=t+1}^{t+m}(Z_i+\epsilon_i)>0.
\]
Markov's inequality followed by iterated use of
\cref{eq:constant-busy-mgf,eq:constant-epsilon-small-mgf} therefore yields
\[
\Pr\{\mathsf R_t,\mathsf R_{t+1}^c,\ldots,
\mathsf R_{t+m}^c\}\leq K_0q^m.
\]
If \(\mathsf A_n\) is the number of cycles since the last reset, partitioning by that
last reset gives
\begin{equation}\label{eq:constant-age-tail}
\Pr\{\mathsf A_n\geq m\}
\leq K_0\sum_{j=m}^\infty q^j
=\frac{K_0}{1-q}q^m.
\end{equation}

The map \(x\mapsto(x+Z_n)^+\) is one-Lipschitz. At a reset the reflected part
of the inventory gap is zero; between resets its increase is at most the new
\(\epsilon\). Consequently
\[
0\leq X_n-Y_n\leq\sum_{k=0}^{\mathsf A_n}\epsilon_{n-k}.
\]
The exponential-moment assumption implies
\(\sup_i\mathbb E[\epsilon_i^2]\leq C_\epsilon T^{2\alpha}\) for a finite
constant \(C_\epsilon\). Applying Cauchy--Schwarz term by term and using
\eqref{eq:constant-age-tail},
\begin{align*}
\mathbb E[X_n-Y_n]
&\leq\sum_{k=0}^n
\mathbb E[\epsilon_{n-k}\mathbb I\{\mathsf A_n\geq k\}]\\
&\leq
\sqrt{C_\epsilon}\,T^\alpha
\left(1+\sqrt{\frac{K_0}{1-q}}
\sum_{k=1}^\infty q^{k/2}\right)
\leq K_\Delta T^\alpha,
\end{align*}
uniformly in \(n,c,T\), as claimed.
\hfill\Halmos
\endproof

\proof{Proof of \cref{lem:constant-sqrt-transfer-average}.}
Set \(b_T=T^\alpha\), \(H=\sqrt T\), and \(d=b_T\), and split
\(\epsilon_n=s_n+\ell_n\), where
\(s_n=\epsilon_n\mathbb I\{\epsilon_n\le H\}\). From zero define
\[
S_n=(S_{n-1}+c+d-D^n)^++s_n,\quad
L_n=(L_{n-1}-d)^++\ell_n,\quad
Y'_n=(Y'_{n-1}+c+d-D^n)^+.
\]
Positive-part subadditivity gives
\begin{equation}\label{eq:sm2-gap-decomposition}
0\le X_n-Y_n\le(S_n-Y'_n)+(Y'_n-Y_n)+L_n.
\end{equation}

The quadratic assumption implies, conditionally on \(\mathcal F_{n-1}\),
\(
\mathbb Es_n\le\sqrt{K_2}b_T
\)
and \(\mathbb Es_n^2\le K_2b_T^2\). Convexity on \([0,H]\), conditional
Cauchy--Schwarz, and Hoeffding's lemma therefore give, for a fixed small
\(\theta>0\),
\[
\mathbb E[e^{\theta(c+d-D^n+s_n)/H}\mid\mathcal F_{n-1}]\le q<1
\]
uniformly for large \(T\). Here we use \(b_T/H\to0\) and
\(\mu T-c-d\ge\kappa H/2\).

Call \(t\) a reset when \(S_{t-1}+c+d-D^t\le0\), so
\(S_t=s_t\) and \(Y'_t=0\). If no reset occurs in the next \(m\) cycles,
the preceding exponential bound and Markov's inequality give probability at
most \(e^\theta q^m\). Thus the reset age \(\mathsf A_n\) has a uniform
geometric tail. Since reflection is one-Lipschitz between resets,
\[
0\le S_n-Y'_n\le\sum_{k=0}^{\mathsf A_n}s_{n-k}.
\]
Cauchy--Schwarz, \(\mathbb Es_n^2\le K_2b_T^2\), and the geometric tail yield
\begin{equation}\label{eq:sm2-small-gap}
\sup_n\mathbb E(S_n-Y'_n)\le C_sb_T.
\end{equation}

For the change in offline order, the finite Spitzer formula and
\((u+jd)^+-u^+\le jd\mathbb I\{u+jd>0\}\) imply
\begin{equation}\label{eq:sm2-offline-sensitivity}
\mathbb E(Y'_n-Y_n)
\le d\sum_{j\ge1}e^{-2j(\mu T-c-d)^2/T}\le C_od.
\end{equation}
Finally, \(\mathbb E\ell_n\le K_2b_T^2/H\) and
\(\mathbb E\ell_n^2\le K_2b_T^2\). From the recursion for \(L_n\),
\[
\mathbb E[L_n^2-L_{n-1}^2\mid\mathcal F_{n-1}]
\le-dL_{n-1}+C_Lb_T^2
\]
for large \(T\). Since \(\mathbb E L_1^2\leq K_2b_T^2\), summing this
inequality for \(n=2,\ldots,N+1\) gives
\[
d\sum_{n=1}^N\mathbb E L_n
\leq \mathbb E L_1^2+NC_Lb_T^2.
\]
Thus \(N^{-1}\sum_{n=1}^N\mathbb EL_n\le C'_Lb_T\). Averaging
\eqref{eq:sm2-gap-decomposition} and applying the last three bounds proves
\eqref{eq:constant-sqrt-inventory-gap-average}.
\hfill\Halmos
\endproof

\proof{Proof of \cref{lem:constant-off-optimizer-slack}.}
Use a common bi-infinite demand sequence to construct the stationary
myopic-offline Lindley processes. Their ending inventories satisfy
\[
Y_n(c)=\bigl(Y_{n-1}(c)+c-D^n\bigr)^+,
\qquad c<\mu T,
\]
and are nondecreasing in \(c\). Hence the number fulfilled in a cycle,
\(F_n(c)=\min\{D^n,Y_{n-1}(c)+c\}\), is also nondecreasing in \(c\).
Myopic offline fulfillment accepts the highest-reward \(F_n(c)\) arrivals,
so the accepted sets can be chosen nested. For adjacent stable orders
\(c,c+1<\mu T\), stationary flow balance gives
\(\mathbb EF_n(c)=c\), and therefore
\begin{equation}\label{eq:constant-reward-forward-bounds}
r_1\leq Q_T(c+1)-Q_T(c)\leq r_M.
\end{equation}

Spitzer's identity gives
\[
H_T(c)=\sum_{n\geq1}\frac1n
\mathbb E\!\left[(nc-\textstyle\sum_{i=1}^nD^i)^+\right].
\]
Thus \(H_T\), and hence every \(J_p\), is discrete convex. Write
\(d_p(c)=J_p(c+1)-J_p(c)=h[H_T(c+1)-H_T(c)]-p\) for adjacent
stable \(c,c+1\). For the smallest minimizer \(c_p\), discrete convexity
implies \(d_p(c)<0\) whenever \(c<c_p\), and \(d_p(c)\geq0\) whenever
\(c\geq c_p\) and \(c+1\) remains stable. Combining this fact with
\eqref{eq:constant-reward-forward-bounds} shows that the offline profit is
strictly increasing below \(c_{r_1}\) and nonincreasing along every stable
adjacent pair from \(c_{r_M}\) onward. This proves
\eqref{eq:constant-optimizer-sandwich}.

It remains to locate the two scalar-cost minimizers. By
\cref{lem:constant-cost-forward-difference},
\[
d_p(c)=h\sum_{n\geq1}\frac1n
\mathbb E\!\left[(n(c+1)-K_{n,T})^+-(nc-K_{n,T})^+\right]-p.
\]
Put \(g=\mu T-c\), and restrict the following comparisons to adjacent
stable orders \(c,c+1\). The summand divided by \(n\) is at least
\(\Pr(K_{n,T}\leq nc)\). Fix an integer \(m>2p/h\). The
Berry--Esseen bound, first for \(1\leq n\leq m\) and then for a sufficiently
small fixed \(\bar a_p>0\), shows that \(d_p(c)\geq0\) whenever
\(0<g\leq\bar a_p\sqrt T\), for all large \(T\). In the other direction, the
same summand is at most \(\Pr(K_{n,T}<n(c+1))\). If
\(g\geq\bar b_p\sqrt T\), Hoeffding's inequality gives
\[
d_p(c)+p
\leq h\sum_{n\geq1}
 \exp\{-2n(g-1)^2/T\}<p
\]
for a sufficiently large fixed \(\bar b_p\) and all large \(T\). Thus
\(d_p(c)<0\) in this region.

Choose fixed \(0<a_p<\bar a_p\) and \(b_p>\bar b_p\), and define
\[
c_-=\lfloor\mu T-b_p\sqrt T\rfloor,\qquad
c_+=\lfloor\mu T-a_p\sqrt T\rfloor.
\]
For all sufficiently large \(T\), these comparison points and their right
neighbors are nonnegative stable orders. Moreover,
\(d_p(c_-)<0\) and \(d_p(c_+)\geq0\). Since \(c_p\) is the smallest
minimizer, discrete convexity gives
\[
c_-+1\leq c_p\leq c_+.
\]
The definitions of the two comparison points now imply
\begin{equation}\label{eq:scalar-optimizer-sqrt-slack}
a_p\sqrt T\leq\mu T-c_p\leq b_p\sqrt T
\end{equation}
for all sufficiently large \(T\). Applying this at \(p=r_1,r_M\) and using
\eqref{eq:constant-optimizer-sandwich} gives
\eqref{eq:constant-off-sqrt-slack}, with
\(\kappa_-=a_{r_M}\) and \(\kappa_+=b_{r_1}\).
\hfill\Halmos
\endproof

\proof{Proof of \cref{thm:reg-total-constant}.}
Let \(R_n^{\mathrm{alg}}(I)\) denote the fulfilled reward in cycle \(n\)
starting from inventory \(I\). Fix \(c\in\mathcal C_\delta(T)\). By
\cref{lem:constant-off-more},
\(I_1^{n,\mathrm{on}}\geq I_1^{n,\mathrm{off}}\). The offline reward value is
nondecreasing in inventory, so
\[
R_n^{\mathrm{off}}(I_1^{n,\mathrm{off}})
-R_n^{\mathrm{on}}(I_1^{n,\mathrm{on}})
\leq
R_n^{\mathrm{off}}(I_1^{n,\mathrm{on}})
-R_n^{\mathrm{on}}(I_1^{n,\mathrm{on}}).
\]
Conditional uniformity in \cref{ass:online-algorithm} bounds the expected
right-hand side by \(C_1T^\alpha\). The same lemma gives
\(0\leq I_{T+1}^{n,\mathrm{on}}-I_{T+1}^{n,\mathrm{off}}
\leq I_{T+1}^{n,\mathrm{on}}\); hence \cref{lem:Inv-on-bound} bounds the
holding-cost contribution by \(hK_{\mathrm c}T^\alpha\). Therefore,
\[
\pi_\mathbf{c}^{\text{off}}(T,c,N)
-\pi_\mathbf{c}^{\text{on}}(T,c,N)
\leq(C_1+hK_{\mathrm c})T^\alpha .
\]
By optimality of \(c^{\mathrm{on},\delta}\) over the same restricted domain,
\[
\pi_{\mathbf c,\delta}^{\mathrm{off}}(T,N)
-\pi_{\mathbf c,\delta}^{\mathrm{on}}(T,N)
\leq
\pi_\mathbf c^{\mathrm{off}}(T,c^{\mathrm{off},\delta},N)
-\pi_\mathbf c^{\mathrm{on}}(T,c^{\mathrm{off},\delta},N),
\]
so the fixed-\(c\) estimate proves \eqref{ineq:reg-total-constant-case-a}.

For part (b), take \(c=c_\infty^{\mathrm{off}}\). By
\cref{lem:constant-off-optimizer-slack}, this order satisfies
\(c\leq\mu T-\kappa_-\sqrt T\) for all sufficiently large \(T\). The same
reward comparison as above, now combined with
\cref{lem:constant-sqrt-transfer-average}, gives for every \(N\), from the common
zero-carryover initialization,
\[
\pi_\mathbf c^{\mathrm{off}}(T,c,N)
-\pi_\mathbf c^{\mathrm{on}}(T,c,N)
\leq(C_1+h\bar K_\Delta)T^\alpha.
\]
The myopic-offline average at this stable order converges. Taking a lower
limit in the online term of the preceding finite-\(N\) inequality gives
\[
\pi_\mathbf c^{\mathrm{off}}(T,c,\infty)
-\liminf_{N\to\infty}\pi_\mathbf c^{\mathrm{on}}(T,c,N)
\leq(C_1+h\bar K_\Delta)T^\alpha .
\]
Since \(\underline\pi_{\mathbf c}^{\mathrm{on},*}(T)\) is at least the
lower long-run average at this particular \(c\), online optimality yields
\eqref{eq:constant-unrestricted-et} with
\(K_*=C_1+h\bar K_\Delta\). The Bayes Selector conclusion follows from
\cref{lem:fbs-single-cycle-et}.
If the online limit at \(c=c_\infty^{\mathrm{off}}\) exists and is finite,
the liminf above equals that limit. Since the optimized value is at least the
value at this order, the same fixed-\(c\) inequality gives the final
assertion.

Finally, for part (c), choosing \(c=0\) gives zero profit, whereas
no policy can earn more than the reward of all nonempty arrivals. Hence, for
\(a\in\{\mathrm{on},\mathrm{off}\}\),
\[
0\leq \pi_\mathbf c^a(T,c^a(T,N),N)
\leq r_M\mathbb E[D^n]=r_M\mu T.
\]
The difference of two numbers in this interval has absolute value at most
\(r_M\mu T\), which proves \eqref{eq:constant-unrestricted-envelope}.
For every fixed order and every horizon, profit is at most \(r_M\mu T\).
Thus every finite fixed-order long-run value has the same upper bound, and
taking the supremum over orders preserves it. The candidate \(c=0\) gives
the lower bound zero. Each optimized long-run value therefore lies in the
stated interval, which proves the final assertion in part (c).
\hfill\Halmos
\endproof

\proof{Proof of \cref{thm:reg-total-constant}(d).}
Let \(\mathcal R_T=T\sum_j\lambda_jr_j\). For a common order \(c\), couple
the ending inventories \(X_n,Y_n\) from zero carryover. Then
\[
X_n=(X_{n-1}+c-D^n)^++\epsilon_n,\qquad
Y_n=(Y_{n-1}+c-D^n)^+,\qquad X_n\ge Y_n.
\]
Writing \(H_N(c)=N^{-1}\sum_n\mathbb EY_n(c)\), inventory flow and
\(r_1\le r_j\le r_M\) give, for \(c<\mu T\),
\[
r_1(\mu T-c)+hH_N(c)\le\mathcal R_T-\pi_{\mathbf c}^a(T,c,N)
\le r_M\left(\mu T-c+\frac{\mathbb EZ_N^a}{N}\right)
+\frac hN\sum_n\mathbb EZ_n^a,
\]
where \(Z^a=X\) or \(Y\). For arbitrary \(c\), the lower bound
\(\mathcal R_T-\pi_{\mathbf c}^a\ge hH_N(c)\) remains valid. The finite-horizon
maximizers exist because first-cycle holding cost diverges with \(c\).

Suppose first that \(\alpha<1/2\), and take
\(c_a=\lfloor\mu T-a\sqrt T\rfloor\). The finite Spitzer formula and
Hoeffding's inequality give \(\sup_n\mathbb EY_n(c_a)=O(\sqrt T)\).
\Cref{lem:constant-sqrt-transfer-average} gives the same horizon-average order
for \(X_n(c_a)\). Since every cycle starts with at least \(c_a\) fresh units,
both candidate profit deficits are \(O(\sqrt T)\), uniformly in \(N\); the
online same-start regret is smaller, of order \(T^\alpha\). Optimality gives
the same deficit bound at both optimizers.

We next exclude near-critical orders for long horizons. Let \(C\sqrt T\) be a
common upper bound on the two candidate deficits in the preceding paragraph.
For \(\sigma^2=\mu(1-\mu)\), Spitzer's formula and the central limit theorem
give, for fixed \(n\),
\[
\frac{H_n(\lfloor\mu T\rfloor)}{\sqrt T}\longrightarrow
\frac{\sigma}{n\sqrt{2\pi}}\sum_{j=1}^n\frac{n-j+1}{\sqrt j},
\]
whose limit grows without bound as \(n\to\infty\). Choose \(N_0\) so that
this limit at \(n=N_0\), multiplied by \(h\), is strictly larger than \(C\).
For
\(c_{\kappa,T}=\lceil\mu T-\kappa\sqrt T\rceil\), the joint central limit
theorem and uniform integrability give a limit for
\(H_{N_0}(c_{\kappa,T})/\sqrt T\) that is continuous at \(\kappa=0\), where
its value at zero is the preceding limit.
Choosing \(\kappa>0\) sufficiently small preserves the strict inequality.
The Lindley maximum representation shows that \(\mathbb EY_n(c)\), and hence
\(H_N(c)\), is nondecreasing in \(c\) and \(N\). Therefore, for all
sufficiently large \(T\) and every
\(N\ge N_0\), the holding-cost lower bound excludes every
\(c\ge\mu T-\kappa\sqrt T\) from optimality.

At any fixed \(c\), the offline reward is nondecreasing and
\(r_M\)-Lipschitz in starting inventory. With \(\Delta_n=X_n-Y_n\),
\[
\pi_{\mathbf c}^{\mathrm{off}}-\pi_{\mathbf c}^{\mathrm{on}}
\le C_1T^\alpha+\frac hN\sum_n\mathbb E\Delta_n,\qquad
\pi_{\mathbf c}^{\mathrm{on}}-\pi_{\mathbf c}^{\mathrm{off}}
\le\frac{r_M}{N}\sum_n\mathbb E\Delta_{n-1}.
\]
For \(N\ge N_0\), apply \cref{lem:constant-sqrt-transfer-average} at each
optimizer and use optimality in the two directions. For \(N<N_0\),
\(\Delta_n\le\Delta_{n-1}+\epsilon_n\) gives
\(\mathbb E\Delta_n\le nC_2T^\alpha\). Since \(N_0\) is fixed, both cases
give a uniform \(O(T^\alpha)\) difference.

Now let \(1/2\le\alpha<1\), put \(b_T=T^\alpha\), and take
\(c_b=\lfloor\mu T-ab_T\rfloor\). With \(Z_n=c_b-D^n\) and
\(e_n=\epsilon_n\),
\begin{equation}\label{eq:sm2-high-alpha-drift}
\bigl((x+Z_n)^++e_n\bigr)^2-x^2
\le2x(Z_n+e_n)+Z_n^2+2e_nZ_n^++e_n^2.
\end{equation}
Because \(b_T^2\ge T\), independence, Cauchy--Schwarz, and
\cref{ass:single-cycle-sm2} show that all terms not multiplied by \(x\) are
at most \(Cb_T^2\). Choosing fixed \(a>\sqrt{K_2}\) large enough gives
\[
\mathbb E[X_n^2-X_{n-1}^2\mid\mathcal F_{n-1}]
\le-\gamma b_TX_{n-1}+Cb_T^2.
\]
After summation, \(N^{-1}\sum_n\mathbb EX_n=O(b_T)\), and the same holds for
\(Y_n\le X_n\). The candidate profits are therefore at least
\(\mathcal R_T-O(b_T)\). Since \(\mathcal R_T\) is an upper bound, the two
optimized profits differ by \(O(b_T)\), as required.
\hfill\Halmos
\endproof

\proof{Proof of \cref{prop:first-moment-insufficient}.}
Take two types with aggregate rate \(\mu<1\), rewards \(0<r_1<r_2\), and
holding cost \(h>0\). Independently in every cycle, draw
\(B_n\sim\operatorname{Bernoulli}(p_T)\). If \(B_n=1\), reject the whole
cycle; otherwise run the cycle-restarted Bayes Selector. This sequential policy has
conditional same-start regret at most \(C_{\mathrm{Bayes}}+r_2p_TT\). Set
\(p_T=T^{\alpha-1}\) for \(0<\alpha<1\), and
\(p_T=\min\{K/T,1\}\) for \(\alpha=0\). Thus
\cref{ass:online-algorithm} holds with exponent \(\alpha\).

Let \(D_n\sim\operatorname{Binomial}(T,\mu)\), put
\(W_n=(1-B_n)D_n\), and define
\[
\widetilde X_n(c)=(\widetilde X_{n-1}(c)+c-W_n)^+,\qquad
\widetilde X_0(c)=0.
\]
The online policy serves at most \(W_n\), so its ending inventory satisfies
\(X_n(c)\ge\widetilde X_n(c)\). Write
\[
q_T=p_TT,\quad m_T=(1-p_T)\mu T,\quad s_T=\sqrt{q_TT},\quad
c_0=\lfloor m_T-ks_T\rfloor,\quad
J_T=\left\lfloor\frac{\mu T}{4(\mu q_T+ks_T+1)}\right\rfloor.
\]
For \(j\le J_T\), condition on exactly one catastrophe in the block. Jensen's
inequality and the mean of the remaining demands give, for large \(T\),
\[
\mathbb E\left[\left(jc_0-\sum_{i=1}^jW_i\right)^+
\middle|\text{one catastrophe}\right]\ge\frac{\mu T}{2}.
\]
The finite Spitzer formula then yields, for \(n\ge J_T\),
\[
\mathbb E\widetilde X_n(c_0)
\ge\sum_{j=1}^{J_T}p_T(1-p_T)^{j-1}\frac{\mu T}{2}
\ge Cq_TJ_T\ge C's_T.
\]
Here \(p_TJ_T=o(1)\) and \(\mu q_T+ks_T+1=O(s_T)\). By monotonicity, the
same lower bound holds for \(X_n(c)\) whenever \(c\ge c_0\).

Choose \(N(T)\ge2J_T\). If \(c<c_0\), flow balance bounds average sales by
\(c\), so the reward deficit from
\(\mathcal R_T=T(\lambda_1r_1+\lambda_2r_2)\) is at least
\(r_1ks_T-O(1)\). If \(c\ge c_0\), at least half the cycles incur expected
holding inventory at least \(C's_T\). Hence the optimized online deficit is
at least \(c_*s_T\). The offline candidate
\(\lfloor\mu T-a\sqrt T\rfloor\) has deficit \(O(\sqrt T)\), uniformly in
\(N\). For \(\alpha>0\), \(s_T=T^{(1+\alpha)/2}\); for \(\alpha=0\),
\(s_T=\sqrt{KT}\), and fixed \(K\) can make the lower-bound coefficient
exceed the offline one. Both claims follow.
\hfill\Halmos
\endproof

\proof{Proof of \cref{cor:fbs-constant-ordinary}.}
Fix \(c<\mu T\) and let \(X_n,Y_n\) be the Bayes Selector and offline ending
inventories. If \(X_{n-1}=x\ge T\), the Bayes Selector accepts every nonnull arrival,
so \(X_n=x+c-D^n\). For sufficiently small \(s>0\),
\[
\mathbb E[e^{sX_n}\mid X_{n-1}=x]
=e^{sx+sc}(\lambda_0+\mu e^{-s})^T<e^{sx},
\]
because the logarithmic derivative at zero is \(c-\mu T<0\). Foster drift
gives positive recurrence and a finite exponential stationary moment. An
all-highest-type cycle maps the boundary inventory \(x\) to
\((x+c-T)^+\). Since \(c<T\), repetition reaches zero from every finite
state, and zero has a positive-probability self-loop. Thus the recurrent
class is unique and aperiodic, and expected average profit converges. The
offline Lindley chain is
similar. If \(c\ge\mu T\), both inventories dominate
\([\sum_{i=1}^n(c-D^i)]^+\), whose expectation diverges linearly for
\(c>\mu T\) and as \(\sqrt n\) at equality. Holding cost then diverges while
reward is bounded, so both optimizers lie among finitely many stable orders.

Let \(\sigma^2=\mu(1-\mu)\), \(g=\mu T-c\), and
\(\overline R_T=T\sum_jr_j\lambda_j\). The stationary offline workload is
\begin{equation}\label{eq:fluid-bs-workload-spitzer}
H_T(c)=\sum_{n\ge1}\frac1n
\mathbb E\left[\left(nc-\sum_{i=1}^nD^i\right)^+\right].
\end{equation}
At \(c_a=\lfloor\mu T-a\sqrt T\rfloor\), Cauchy--Schwarz and Hoeffding give
\(
H_T(c_a)\le\sigma\sqrt T\sum_{n\ge1}n^{-1/2}e^{-a^2n}=O(\sqrt T).
\)
Flow balance gives offline lost demand \(g_a\), while
\cref{lem:constant-sqrt-transfer} transfers the same order to the Bayes Selector.
Hence both optimized deficits from \(\overline R_T\) are at most \(C\sqrt T\).

At every stable \(c\), both policies serve \(c\) requests on average and the
online inventory dominates the offline workload. Therefore, for
\(b\in\{\mathrm{off},\mathrm{Bayes}\}\),
\begin{equation}\label{eq:fluid-bs-deficit-lower}
\overline R_T-\pi_{\mathbf c}^b(T,c,\infty)\ge r_1g+hH_T(c).
\end{equation}
This implies \(g=O(\sqrt T)\) at either optimizer. Conversely, if
\(0<g\le\kappa\sqrt T\), a fixed-block central limit bound in
\eqref{eq:fluid-bs-workload-spitzer} gives
\[
H_T(c)\ge\sigma p_0\sqrt T
\sum_{n=1}^{\lfloor\sigma^2/\kappa^2\rfloor}n^{-1/2}.
\]
Its coefficient diverges as \(\kappa\downarrow0\), so
\eqref{eq:fluid-bs-deficit-lower} excludes sufficiently small slack. This
proves the common \(\Theta(\sqrt T)\) optimizer band.

On that band, \cref{lem:constant-sqrt-transfer} with \(\alpha=0\) yields
\(\sup_n\mathbb E(X_n-Y_n)\le K_\Delta\). At a common order, the forward
same-start comparison and the reverse \(r_M\)-Lipschitz comparison give
\[
\pi^{\mathrm{off}}-\pi^{\mathrm{Bayes}}\le C_1+hK_\Delta,\qquad
\pi^{\mathrm{Bayes}}-\pi^{\mathrm{off}}\le r_MK_\Delta.
\]
Apply these inequalities at the two respective optimizers. The result follows
with \(K_{\mathrm{gap}}=\max\{C_1+hK_\Delta,r_MK_\Delta\}\).
\hfill\Halmos
\endproof

\section{Proofs in \cref{subsec:base-constant-off}}
\label{app:proof_base_constant_off} 

Before proceeding to the proofs in \cref{subsec:base-constant-off}, we recall
the lost-sales--backorder cost bridge used below.

\begin{lemma}[Lost-sales--backorder cost bridge]
\label{lem:lost-back-bridges}
For every fixed base-stock level \(S\), the cost sandwich in
\cite[Lemma~5]{huh2009asymptotic} gives
\[
C_\mathbf{b}^{\mathscr{B}, S} \left(h,\tfrac{p}{L+1}\right)  \le  C_\mathbf{b}^{\mathscr{L}, S}(h,p)   \le   C_\mathbf{b}^{\mathscr{B}, S}  \left(h,p+Lh\right).
\]
\end{lemma}

This bridge lets us bound the lost-sales cost through backorder systems with
adjusted penalty parameters, so it suffices to compute the corresponding
backorder costs at the appropriate newsvendor levels.

\proof{Proof of \cref{lem:bounds-base-off}.}
By optimality and \cref{lem:lost-back-bridges},
\[
C_\mathbf{b}^{\mathscr{L},    S^{\mathscr{L}*}(h,p)}  \left(h,p\right)  
\le C_\mathbf{b}^{\mathscr{L}, \bar{S}^\mathscr{B}(h,p)}(h, p) 
\le C_\mathbf{b}^{\mathscr{B}, \bar{S}^\mathscr{B}(h,p)}(h, p+Lh)
=\bar a_1(T,h,p,L).
\]
For every \(S\), the other side of the bridge and the optimality of
\(\underline S^{\mathscr B}(h,p)\) give
\[
C_\mathbf{b}^{\mathscr{L},    S}  \left(h,p\right)
  \ge  
C_\mathbf{b}^{\mathscr{B},    S}  \left(h,\tfrac{p}{L+1}\right)
  \ge  
C_\mathbf{b}^{\mathscr{B},    \underline{S}^{\mathscr{B}}(h,p)}
 \left(h,\tfrac{p}{L+1}\right)
=\underline a_1(T,h,p,L).
\]
Taking \(S=S^{\mathscr L*}(h,p)\) proves the lower bound. The equalities
above follow by expanding the two backorder costs under the binomial law
\(P_k\), and their minimizers satisfy the stated newsvendor fractiles.
\hfill \Halmos 
\endproof

\proof{Proof of \cref{lem:bound-constant-off}.}
Under a stable constant order \(c<T(1-\lambda_0)\), the stationary ending
inventory obeys the Lindley recursion.  Loynes' representation gives
\[
I_{T+1}^{\infty, \text{off}}(c,\mathcal{H}) = \sup_{j \geq 0} \left( j c - \sum_{i=1}^{j} D^i \right),
\]
and Spitzer's identity gives the exact expected inventory cost
\[
h \mathbb{E}_\mathcal{H}[I_{T+1}^{\infty, \text{off}}(c,\mathcal{H})] = h \sum_{n=1}^{\infty} \frac{1}{n} \mathbb{E}_\mathcal{H} \left[ \left(n c - \sum_{i=1}^{n} D^i \right)^{+} \right].
\]
Combining the expected lost-sale penalty and expanding the expectation via \(P_{n,k}\), we can get the total cost function as follows:
\[
C^c_\mathbf{c}(h, p) = h \sum_{n=1}^{\infty} \sum_{k=0}^{n c} \left(c - \frac{k}{n} \right) P_{n,k} + p T(1-\lambda_0) - p c.
\]
Substitution of \(c^*(h,p)\) into this expression gives
\(a_2(T,h,p)\).
\hfill \Halmos
\endproof

\begin{lemma}[Adjacent differences for the constant-order cost]
\label{lem:constant-cost-forward-difference}
Whenever both \(c\) and \(c+1\) belong to the stable domain, the exact
forward difference is
\[
d_T(c)\triangleq C_\mathbf c^{c+1}(h,p)-C_\mathbf c^c(h,p)
=h\sum_{n=1}^{\infty}\frac1n
\mathbb E\!\left[
(n(c+1)-K_{n,T})^+-(nc-K_{n,T})^+
\right]-p.
\]
The optimality characterization is
\[
d_T(c^*(h,p)-1)\leq0\quad\text{if }c^*(h,p)>0,
\]
and, if \(c^*(h,p)+1<T(1-\lambda_0)\), also
\[
d_T(c^*(h,p))\geq0.
\]
\end{lemma}

\proof{Proof of \cref{lem:constant-cost-forward-difference}.}
For integral \(c\) such that both \(c\) and \(c+1\) are stable, subtracting
the cost at \(c\) from that at \(c+1\) gives
\[
d_T(c)
=h\sum_{n=1}^{\infty}\frac1n
\mathbb E\!\left[
(n(c+1)-K_{n,T})^+-(nc-K_{n,T})^+
\right]-p.
\]
Each summand is convex in \(c\), so \(d_T(c)\) is nondecreasing wherever it
is finite. An integer \(c^*\) is therefore optimal exactly when
\(d_T(c^*-1)\leq0\), with this inequality omitted at \(c^*=0\), and
\(d_T(c^*)\geq0\) whenever \(c^*+1\) is stable; if \(c^*\) is the largest
stable integer, the latter condition is automatic under the \(+\infty\)
extension.
\hfill \Halmos 
\endproof

\proof{Proof of \cref{thm:replenishment-gap}.}
Let \(V_T\) be the expected value of all nonnull arrivals. For either
replenishment policy, replacing every lost reward by \(p\) times the number
of lost sales gives the corresponding cost objective subtracted from
\(V_T\). Since each lost reward lies in \([r_1,r_M]\), optimality within
each policy class implies
\[
V_T-C_{\mathbf b}^{\mathscr L,S^{\mathscr L*}(h,r_M)}(h,r_M)
\le \pi_{\mathbf b}^{\mathrm{off}}
\le V_T-C_{\mathbf b}^{\mathscr L,S^{\mathscr L*}(h,r_1)}(h,r_1),
\]
and
\[
V_T-C_{\mathbf c}^{c^*(h,r_M)}(h,r_M)
\le \pi_{\mathbf c}^{\mathrm{off}}
\le V_T-C_{\mathbf c}^{c^*(h,r_1)}(h,r_1).
\]
Subtracting the appropriate endpoints and applying
\cref{lem:bounds-base-off,lem:bound-constant-off} proves both bounds.
\hfill \Halmos 
\endproof

\proof{Proof of \cref{coro:replenishment-gap-sqrtT-scaling}.}
Let
\[
K\sim\operatorname{Bin}((L+1)T,\mu),\qquad
m_T=(L+1)T\mu,\qquad
\sigma_T=\sqrt{(L+1)T\mu(1-\mu)}.
\]
Both base-stock envelopes have the form
\[
A_T(u,v)=u\mathbb E(S_T-K)^++v\mathbb E(K-S_T)^+,
\]
where \(u,v>0\) are fixed and \(S_T\) is the smallest integer satisfying
\(\Pr(K\leq S_T)\geq q=v/(u+v)\). The choices
\((u,v)=(h,p+Lh)\) and \((h,p/(L+1))\) give
\(\bar a_1\) and \(\underline a_1\), respectively.

Put \(z_T=(S_T-m_T)/\sigma_T\). The Kolmogorov Berry--Esseen bound
places \(z_T\) in a fixed compact interval because \(q\in(0,1)\). The
Wasserstein version, applied to the two one-Lipschitz stop-loss functions,
gives, uniformly over this interval,
\[
\frac{\mathbb E(K-S_T)^+}{\sigma_T}
=\mathbb E(Z-z_T)^++O(T^{-1/2}),\qquad
\frac{\mathbb E(S_T-K)^+}{\sigma_T}
=\mathbb E(z_T-Z)^++O(T^{-1/2}),
\]
where \(Z\) is standard normal. Both normal expectations are continuous and
strictly positive on the compact interval. Hence each stop-loss expectation
is \(\Theta(\sigma_T)\), proving
\(\bar a_1,\underline a_1=\Theta(\sqrt T)\).

For the constant-order cost, choose the smallest stable-domain minimizer
\(c^*(h,p)\). By
\eqref{eq:scalar-optimizer-sqrt-slack}, there are constants
\(0<\kappa_-<\kappa_+<\infty\) such that, for all sufficiently large \(T\),
\[
\kappa_-\sqrt T\leq \mu T-c^*(h,p)\leq\kappa_+\sqrt T.
\]
The representation in \cref{lem:bound-constant-off} is
\[
a_2(T,h,p)=h\sum_{n\geq1}\frac1n
\mathbb E(nc^*(h,p)-K_{n,T})^+
+p(\mu T-c^*(h,p)),
\]
where \(K_{n,T}\sim\operatorname{Bin}(nT,\mu)\). Put
\(g_T=\mu T-c^*(h,p)\). A binomial Chernoff bound gives constants
\(C,c_0>0\) such that
\[
\frac1n\mathbb E(nc^*(h,p)-K_{n,T})^+
\leq C\sqrt T\,\frac{e^{-c_0n}}{n},
\qquad n\geq1.
\]
Indeed, write the expectation as a tail integral and use
\(ng_T\geq n\kappa_-\sqrt T\). Summing the displayed bound shows that the
inventory term is \(O(\sqrt T)\). The optimizer-location bound also gives
\[
p\kappa_-\sqrt T
\leq p(\mu T-c^*(h,p))
\leq p\kappa_+\sqrt T.
\]
All terms in \(a_2\) are nonnegative, so the last two displays give both the
upper and lower bounds. Hence \(a_2(T,h,p)=\Theta(\sqrt T)\).

\hfill \Halmos 
\endproof

\section{Proofs in \cref{subsec:off-greedy-base}}
\label{app:proof_off_greedy_base}

We first verify the existence statements used in
\cref{subsec:off-greedy-base}. For a fixed integer \(S\), the augmented
post-order state \((I,q_1,\ldots,q_L)\) ranges over the finite set of
nonnegative integer vectors with sum \(S\). Under either work-conserving
fulfillment rule, \(L\) consecutive zero-demand cycles send every state to
\((S,0,\ldots,0)\), and that state has a positive-probability self-loop. The
chain therefore has a unique stationary distribution and converges to it from
every feasible initial state. The base-stock identity also makes every new
order equal to the preceding cycle's fulfilled quantity, which is at most
\(T\). After the initial pipeline has left the system, every outstanding order
is at most \(T\), and every ending inventory satisfies
\[
I_{T+1}\geq [S-(L+1)T]^+.
\]
Consequently, for \(a\in\{\mathrm{off},\mathrm{greedy}\}\),
\[
\pi_{\mathbf b}^{a}(T,S,\infty)
\leq r_MT-h[S-(L+1)T]^+\longrightarrow-\infty
\qquad(S\to\infty).
\]
Since \(S\) is integer, each long-run profit attains its maximum.

\begin{lemma}[Stationary flow balance]
\label{lem:stationary-flow-balance}
Suppose the base-stock pipeline with level \(S\) has a stationary distribution,
and let \(F^{\infty,S}\) be the stationary number of units fulfilled in a
cycle. Then its order marginals are identical and
\begin{equation}\label{eq:bs-flow-balance}
\mathbb E[I_1^{\infty,S}]
=S-L\mathbb E[F^{\infty,S}].
\end{equation}
\end{lemma}

\proof{Proof of \cref{lem:stationary-flow-balance}.}
Stationarity and the pipeline shift imply
\(q^{n+1}\stackrel d=\cdots\stackrel d=q^{n+L}\). Let \(F^n\) be the
number of units fulfilled in cycle \(n\). Comparing the post-order inventory
position in two consecutive cycles gives
\[
q^{n+L+1}=F^n.
\]
Indeed, using \(I_1^{n+1}=I_1^n-F^n+q^{n+1}\) in the base-stock identity,
\[
\begin{aligned}
S
&=I_1^{n+1}+\sum_{\ell=1}^{L}q^{n+1+\ell}\\
&=I_1^n-F^n+\sum_{\ell=1}^{L}q^{n+\ell}+q^{n+L+1}
=S-F^n+q^{n+L+1}.
\end{aligned}
\]
Thus, in stationarity, \(\mathbb E[q]=\mathbb E[F^{\infty,S}]\).
Taking expectations in
\(I_1^n+\sum_{\ell=1}^{L}q^{n+\ell}=S\) now yields
\[
S=\mathbb E[I_1^{\infty,S}]
  +L\mathbb E[F^{\infty,S}],
\]
which is \eqref{eq:bs-flow-balance}. Notice that stationarity equates the order
marginals with one another, but not with the on-hand-inventory marginal.
\hfill \Halmos 
\endproof

\proof{Proof of \cref{lem:optimized-effective-shortfall-sqrt}.}
Fix \(a\in\{\mathrm{off},\mathrm{greedy}\}\). Both fulfillment rules are
work conserving, so under a common base-stock level they fulfill
\(F_n=\min\{I_1^n,D^n\}\) requests. Write
\(U_n=D^n-F_n=(D^n-I_1^n)^+\). Take
\[
S_0=\lceil(L+1)\mu T\rceil,
\qquad
\delta=S_0-(L+1)\mu T\in[0,1).
\]
The pathwise base-stock identity \(q^{n+L+1}=F_n\) gives
\(I_1^n=S_0-\sum_{k=1}^LF_{n-k}\), and hence
\[
D^n-I_1^n
=\sum_{k=0}^L(D^{n-k}-\mu T)
-\delta-\sum_{k=1}^LU_{n-k}.
\]
Consequently,
\begin{equation}\label{eq:work-conserving-unfilled-sqrt}
\mathbb EU_n
\leq
\mathbb E\!\left[
 \left(\sum_{k=0}^L(D^{n-k}-\mu T)\right)^+
\right]
\leq\sqrt{(L+1)T\mu(1-\mu)}.
\end{equation}
Flow balance also gives
\begin{equation}\label{eq:work-conserving-ending-sqrt}
\mathbb E I_{T+1}^n=\delta+(L+1)\mathbb EU_n.
\end{equation}

Let \(\overline R_T=T\sum_jr_j\lambda_j\) and define the nonnegative
lost-reward-plus-holding cost
\[
C_a(S)=\overline R_T-\pi_{\mathbf b}^a(T,S,\infty).
\]
At \(S_0\), \eqref{eq:work-conserving-unfilled-sqrt}--
\eqref{eq:work-conserving-ending-sqrt} imply
\[
C_a(S_0)
\leq
\bigl[r_M+h(L+1)\bigr]
\sqrt{(L+1)T\mu(1-\mu)}+h.
\]
Optimality gives \(C_a(S^a)\leq C_a(S_0)\). Since every unfulfilled request
has reward at least \(r_1\), it follows that, at \(S^a\),
\begin{equation}\label{eq:optimized-unfilled-sqrt}
\mathbb E[(D^n-I_1^n)^+]=\mathcal O(\sqrt T).
\end{equation}
Finally, pathwise,
\[
(\mu T-I_1^n)^+
\leq(D^n-I_1^n)^++(\mu T-D^n)^+,
\]
and
\(\mathbb E[(\mu T-D^n)^+]\leq\sqrt{T\mu(1-\mu)}\).
Together with \eqref{eq:optimized-unfilled-sqrt}, this proves
\(B_T^a\leq K_B\sqrt T\). The definition of \(G_T^a\) gives
\(0\leq G_T^a\leq B_T^a\).
\hfill\Halmos
\endproof

\proof{Proof of \cref{thm:offline-greedy-UB}.}
By the optimality of \(S^{\mathrm{greedy}}\),
\begin{equation}\label{eq:UB-reduce-to-commonS}
\pi_\mathbf b^{\mathrm{off}}-\pi_\mathbf b^{\mathrm{greedy}}
\le
\pi_\mathbf b^{\mathrm{off}}(T,S^{\mathrm{off}},\infty)
-
\pi_\mathbf b^{\mathrm{greedy}}(T,S^{\mathrm{off}},\infty).
\end{equation}
At a common base-stock level the two rules consume
\(\min\{I_1^n,D^n\}\) units, and hence have identical cycle-boundary states,
orders, and holding costs. The right side of \eqref{eq:UB-reduce-to-commonS}
is therefore the stationary expected one-cycle reward gap.

Now fix a cycle and condition on \(I_1^n=I\) and \(D^n=I+k\) with \(k\ge 0\).
Offline can improve only by replacing earlier greedy acceptances with some of
the last \(k\) nonempty arrivals. If \(H_k\) counts those last arrivals whose
types lie in \(\{2,\ldots,M\}\), then
\[
R^{n,\mathrm{off}}(I,\mathcal H^n)-R^{n,\mathrm{greedy}}(I,\mathcal H^n)
\le (r_M-r_1)\,H_k.
\]
By exchangeability conditional on \(D^n=I+k\),
\(\mathbb E H_k=k(1-\lambda_0-\lambda_1)/(1-\lambda_0)\). Thus, for every
fixed \(I\),
\[
\mathbb E\!\left[
R^{n,\mathrm{off}}(I,\mathcal H^n)-R^{n,\mathrm{greedy}}(I,\mathcal H^n)
\right]
\le
(r_M-r_1)\frac{1-\lambda_0-\lambda_1}{1-\lambda_0}\,
\mathbb E\!\left[(D^n-I)^+\right].
\]

Under the common base-stock level, let
\(I\stackrel d=I_1^{\infty,S^{\mathrm{off}}}\). It is independent of the new
demand. Subadditivity of the positive part and Cauchy--Schwarz give
\[
\mathbb E[(D^n-I)^+]
\le
B_T^{\mathrm{off}}
+\sqrt{T(1-\lambda_0)\lambda_0}.
\]
Substitution into the reward-gap bound and \eqref{eq:UB-reduce-to-commonS}
proves the result.
\hfill \Halmos 
\endproof

\proof{Proof of \cref{thm:offline-greedy-LB}.}
By offline optimality, compare both policies at \(S=S^{\mathrm{greedy}}\).
Their boundary states and holding costs agree, so, with
\(I\stackrel d=I_1^{\infty,S}\),
\begin{equation}\label{eq:LB-commonS-short}
\pi_{\mathbf b}^{\mathrm{off}}-\pi_{\mathbf b}^{\mathrm{greedy}}
\ge\mathbb E_I\mathbb E_{\mathcal H}
[R^{n,\mathrm{off}}(I,\mathcal H^n)-R^{n,\mathrm{greedy}}(I,\mathcal H^n)].
\end{equation}
For \(\mathbf d\in\mathbb R_+^M\), define the fractional offline value
\[
\mathcal R^{\mathrm{off}}(I,\mathbf d)
=\max\left\{\sum_jr_jx_j:0\le x_j\le d_j,\ \sum_jx_j\le I\right\},
\quad \overline{\mathbf d}_T=(T\lambda_1,\ldots,T\lambda_M).
\]

\begin{lemma}[Deterministic-relaxation bound]
\label{lem:deter_relax_gap}
For every integer \(I\ge0\),
\[
0\le\mathcal R^{\mathrm{off}}(I,\overline{\mathbf d}_T)
-\mathbb E R^{n,\mathrm{off}}(I,\mathcal H^n)
\le M_1\sqrt T,\qquad
M_1=\sum_jr_j\sqrt{\lambda_j(1-\lambda_j)}.
\]
\end{lemma}
Indeed, the realized integral knapsack value equals the offline reward.
Concavity gives the first inequality, while coordinatewise Lipschitz
continuity bounds the difference by
\(\sum_jr_j(T\lambda_j-D_j^n[1,T])^+\); Cauchy--Schwarz gives the second.

Conditional on total demand, greedy serves types in proportions
\(\lambda_j/(1-\lambda_0)\). Concavity of \(x\mapsto\min\{I,x\}\) therefore
gives
\[
\mathbb ER^{n,\mathrm{greedy}}(I,\mathcal H^n)
\le\mathcal R^{\mathrm{greedy}}(I,\overline{\mathbf d}_T)
:=\frac{\sum_jr_j\lambda_j}{1-\lambda_0}
\min\{I,T(1-\lambda_0)\}.
\]
The proportional allocation
\(x_j=\lambda_j\min\{I,T(1-\lambda_0)\}/(1-\lambda_0)\) is feasible for the
mean-demand offline LP. Hence
\(\Delta_{\mathrm{det}}(I):=\mathcal R^{\mathrm{off}}-
\mathcal R^{\mathrm{greedy}}\ge0\), and the lemma gives
\begin{equation}\label{eq:LB-basic-short}
\pi_{\mathbf b}^{\mathrm{off}}-\pi_{\mathbf b}^{\mathrm{greedy}}
\ge\mathbb E_I\Delta_{\mathrm{det}}(I)-M_1\sqrt T.
\end{equation}

If
\(T\sum_{j=k}^M\lambda_j\le I<T\sum_{j=k-1}^M\lambda_j\), direct evaluation
of the two deterministic values gives
\[
\Delta_{\mathrm{det}}(I)\ge
\left(T-\frac I{1-\lambda_0}\right)A_k,\qquad
A_k=\sum_{j=k}^Mr_j\lambda_j-r_{k-1}\sum_{j=k}^M\lambda_j.
\]
Thus, on \(T\lambda_M<I\le T(1-\lambda_0)\), this is at least
\(A[T(1-\lambda_0)-I]/(1-\lambda_0)\), where \(A=\min_{k\ge2}A_k\).
Outside that range use \(\Delta_{\mathrm{det}}\ge0\). Taking expectation in
\eqref{eq:LB-basic-short} gives
\[
\pi_{\mathbf b}^{\mathrm{off}}-\pi_{\mathbf b}^{\mathrm{greedy}}
\ge\frac{A}{1-\lambda_0}G_T^{\mathrm{greedy}}-M_1\sqrt T.
\]
\hfill\Halmos
\endproof

\proof{Proof of \cref{coro:offline-greedy-base}.}
At every common base-stock level, the two rules consume the same number of
units, while myopic offline maximizes within-cycle reward. Thus the optimized
gap is nonnegative. The upper bound follows from
\cref{lem:optimized-effective-shortfall-sqrt,thm:offline-greedy-UB}.

For the lower bound, let \(S_g\) be a stationary greedy optimizer and let
\(I\) be its start inventory. For independent \(D\sim\operatorname{Bin}(T,\mu)\),
put \(U=(D-I)^+\), \(J=(I-D)^+\), and
\(\overline R_T=T\sum_jr_j\lambda_j\). The balanced candidate from
\eqref{eq:finite-balanced-candidate} is also work-conserving under greedy, so
optimality and positive rewards give, for a fixed \(K_{\mathrm{bal}}\),
\[
r_1\mathbb EU+h\mathbb EJ
\le\overline R_T-\pi_{\mathbf b}^{\mathrm{greedy}}
\le K_{\mathrm{bal}}\sqrt T.
\]
Since \(|I-D|=U+J\) and
\(\mathbb E|D-\mu T|\le\sqrt{\mu(1-\mu)T}\),
\(\mathbb E|I-\mu T|\le C_0\sqrt T\). Choose \(C>2C_0\). Then
\[
\Pr\{|I-\mu T|\le C\sqrt T\}\ge\frac12,
\]
while the binomial central limit theorem gives a fixed \(p_C>0\) such that
\(\Pr\{\mu T+(C+1)\sqrt T\le D\le\mu T+(C+2)\sqrt T\}\ge p_C\).
On the intersection, greedy rejects the last \(k=D-I\) nonempty arrivals,
where \(\sqrt T\le k\le(2C+2)\sqrt T\), and accepts at least \(\mu T/2\)
earlier arrivals.

Conditional on \(I,D\), arrival types are i.i.d. among nonempty requests. Let
\(p_M=\lambda_M/\mu\), let \(X\sim\operatorname{Bin}(k,p_M)\) count top types
among the last \(k\), and let
\(Y\sim\operatorname{Bin}(I,1-p_M)\) count non-top types among the first
\(I\). They are independent, and a Chernoff bound gives
\(\Pr(Y<k\mid I,D)\le e^{-cT}\). Offline can exchange
\(\min(X,Y)\) early non-top acceptances for rejected top types, gaining at
least \(r_M-r_{M-1}\) each. Hence, on the intersection,
\[
\mathbb E[\min(X,Y)\mid I,D]\ge p_Mk-ke^{-cT}
\ge\frac{p_M}{2}\sqrt T
\]
for large \(T\). The intersection has probability at least \(p_C/2\), so the
same-\(S_g\) reward gap is \(\Omega(\sqrt T)\). At a common target, the two
rules have identical boundary states and holding costs; offline optimization
can only increase the gap. Combining both bounds proves
\eqref{eq:offline-greedy-sqrt-upper}.
\hfill\Halmos
\endproof

\section{Proofs in \cref{subsec:online-offline-constant,subsec:comparison}}

\proof{Proof of \cref{thm:on_vs_off_constant}.}
Let \(R_n^{\mathrm{alg}}(I)\) be the fulfilled reward in cycle \(n\) from
inventory \(I\). Fix \(c=c^{\mathrm{on},\delta}\) and couple the two systems.
By \cref{lem:constant-off-more}, their start- and end-inventory gaps are
nonnegative. Myopic offline fulfillment dominates every online allocation at
the same inventory, and its reward value is \(r_M\)-Lipschitz. Hence
\[
R_n^{\mathrm{on}}(I_1^{n,\mathrm{on}})
-R_n^{\mathrm{off}}(I_1^{n,\mathrm{off}})
\leq r_M\bigl(I_1^{n,\mathrm{on}}-I_1^{n,\mathrm{off}}\bigr).
\]
The online holding cost is no smaller, so its contribution to the profit
difference is nonpositive. For \(n\geq2\), the start-inventory gap equals the
preceding end-inventory gap and is at most the preceding online end inventory.
The first-cycle gap is zero. Therefore \cref{lem:Inv-on-bound} and offline
optimality give
\[
\pi_{\mathbf c,\delta}^{\mathrm{on}}(T,N)
-\pi_{\mathbf c,\delta}^{\mathrm{off}}(T,N)
\leq
\pi_{\mathbf c}^{\mathrm{on}}(T,c^{\mathrm{on},\delta},N)
-\pi_{\mathbf c}^{\mathrm{off}}(T,c^{\mathrm{on},\delta},N)
\leq r_MK_{\mathrm c}T^\alpha .
\]
The unrestricted finite-horizon bound follows from
\eqref{eq:constant-unrestricted-envelope}. The Bayes Selector statement follows
from \cref{cor:fbs-constant-ordinary}.
\hfill\Halmos
\endproof

\proof{Proof of \cref{coro:replenish-fulfill-gap}.}
Decompose the profit difference as
\[
\pi_{\mathbf{b}}^{\text{greedy}} - \pi_{\mathbf{c}}^{\text{on}}
=
\bigl(\pi_{\mathbf{b}}^{\text{greedy}} - \pi_{\mathbf{b}}^{\text{off}}\bigr)
+ \bigl(\pi_{\mathbf{b}}^{\text{off}} - \pi_{\mathbf{c}}^{\text{off}}\bigr)
+ \bigl(\pi_{\mathbf{c}}^{\text{off}} - \pi_{\mathbf{c}}^{\text{on}}\bigr).
\]
Write the decomposition as \(\eta_0-\eta_1+\eta_2\), where
\(\eta_0=\pi_{\mathbf b}^{\mathrm{off}}-\pi_{\mathbf c}^{\mathrm{off}}\),
\(\eta_1=\pi_{\mathbf b}^{\mathrm{off}}-\pi_{\mathbf b}^{\mathrm{greedy}}\),
and
\(\eta_2=\pi_{\mathbf c}^{\mathrm{off}}-\pi_{\mathbf c}^{\mathrm{on}}\).
For the lower bound, use the lower bound on \(\eta_0\), the upper bound on
\(\eta_1\), and
\(\eta_2\geq-r_M(1-\lambda_0)T\) from the long-run clause of
\cref{thm:reg-total-constant}(c). For the upper bound, use the upper
bound on \(\eta_0\), the exact finite-\(T\) lower bound on \(\eta_1\), and
\(\eta_2\leq r_M(1-\lambda_0)T\) from the same clause. Direct substitution yields
\cref{ineq:replenish-fulfill-LB,ineq:replenish-fulfill-UB}.
For the cycle-restarted Bayes Selector defined in \cref{sec:model}, the optimized long-run averages exist and
\(|\eta_2|\le K_{\mathrm{gap}}\) by
\cref{cor:fbs-constant-ordinary}; substituting this bound sharpens
both sides as stated.
\hfill \Halmos
\endproof

\section{Proofs in \cref{sec:look-ahead}}

\proof{Proof of \cref{lem:full-information-lp}.}
Under base stock, every unit sold in cycle \(n\) is restored by the order
arriving in cycle \(n+L+1\); equivalently,
\[
q^{n+L+1}=a_n.
\]
Inventory conservation from the complete post-order state in the statement
then gives, for every cycle \(n\),
\[
e_n=B_n-\sum_{k=\max\{1,n-L\}}^na_k.
\]
Thus nonnegative ending inventory is equivalent to the rolling constraints
in \eqref{eq:full-information-base-lp}. Conversely, any acceptance vector
satisfying those constraints can be implemented cycle by cycle, because its
cycle-\(n\) total does not exceed the on-hand inventory implied by the same
identity. This proves exactness of the base-stock formulation.

Under constant order, cumulative supply through cycle \(n\) is \(nc\).
Hence feasibility is exactly
\(\sum_{k\leq n}a_k\leq nc\), and the unused cumulative supply is
\(e_n=nc-\sum_{k\leq n}a_k\), proving
\eqref{eq:full-information-constant-lp}.

After eliminating the \(e_n\)'s, the rolling-constraint matrix is an interval
matrix and hence totally unimodular. Appending the unit rows that impose
\(0\leq x_{nj}\leq D_{nj}\) preserves total unimodularity. The cumulative
constant-order matrix is the corresponding prefix-interval special case.
Integer right-hand sides therefore admit integer optimal solutions.
\hfill\Halmos
\endproof

\proof{Proof of \cref{prop:constant-full-information-bridge}.}
For a demand vector \(d=(d_j)_{j\in[M]}\) and capacity \(b\), define
\[
K(d,b)=\max\left\{\sum_jr_jy_j:
0\leq y_j\leq d_j,\ \sum_jy_j\leq b\right\}.
\]
This fractional knapsack value is homogeneous and demand-Lipschitz:
\[
K(ad,ab)=aK(d,b),\qquad
|K(d,b)-K(d',b)|\leq r_M\|d-d'\|_1.
\]

On a fixed demand path, the ending inventory of the work-conserving myopic
policy is no larger than that of any feasible full-information policy.
Indeed, if their start inventories satisfy \(I_n^{\mathrm{FI}}\geq
I_n^{\mathrm{off}}\), then
\[
e_n^{\mathrm{FI}}
\geq(I_n^{\mathrm{FI}}-D_n)^+
\geq(I_n^{\mathrm{off}}-D_n)^+
=e_n^{\mathrm{off}},
\]
and both receive the same order \(c\) next cycle. The common initialization
starts the induction. Thus the holding-cost contribution to the difference
between full-information and myopic profits is nonpositive.

The full-information reward is at most
\(K(\sum_nD_n,Nc)\), while myopic fulfillment starts every cycle with at
least \(c\) units and earns at least \(\sum_nK(D_n,c)\). Put
\(\bar d=T(\lambda_1,\ldots,\lambda_M)\). Homogeneity and Lipschitzness give
pathwise
\[
\begin{split}
K\!\left(\sum_nD_n,Nc\right)-\sum_nK(D_n,c)
&\leq r_M\left\|\sum_n(D_n-\bar d)\right\|_1
 +r_M\sum_n\|D_n-\bar d\|_1\\
&\leq2r_M\sum_n\|D_n-\bar d\|_1.
\end{split}
\]
Dividing by \(N\), taking expectations, and using
\[
\mathbb E|D_{nj}-T\lambda_j|
\leq\sqrt{T\lambda_j(1-\lambda_j)}
\]
gives the stated bound uniformly in \(c\) and \(N\). Taking suprema over
\(c\) preserves the inequality, and combining it with
\cref{thm:reg-total-constant}(d) gives the final assertion.
\hfill\Halmos
\endproof

\proof{Proof of \cref{thm:base-quota-full-information}.}
For real demand upper bounds \(D\), let \(\Phi_{\mathbf b}(D;S)\) be the
LP-relaxation value in \eqref{eq:full-information-base-lp}. By
\cref{lem:full-information-lp}, it agrees with the integer full-information
value at realized integer demands. The LP dual feasible set does not depend
on \(D\), and its objective is affine in \(D\). Hence
\(\Phi_{\mathbf b}(D;S)\), as the pointwise infimum of those affine
functions, is concave in \(D\). Jensen's inequality gives
\[
\mathbb E\Phi_{\mathbf b}(D;S)
\leq\Phi_{\mathbf b}(\mathbb ED;S).
\]

Let \(x^*\) solve the deterministic mean-demand LP and define
\[
\widehat x_{nj}
=\min\{D_{nj},\lfloor x^*_{nj}\rfloor\}.
\]
Coordinatewise \(\widehat x\leq x^*\), so every rolling constraint remains
feasible. In particular, the total quotas in a cycle do not exceed its
actual start inventory, and accepting eligible customers in arrival order
is implementable.

Eliminating ending inventory in the objective assigns acceptance
\((n,j)\) the coefficient
\[
r_j+h\,w_n,\qquad
w_n=\min\{L+1,N-n+1\},
\]
apart from the constant \(-h\sum_nB_n\). Since
\(x^*_{nj}\leq T\lambda_j\),
\[
\begin{split}
x^*_{nj}-\mathbb E\widehat x_{nj}
&\leq1+\mathbb E(T\lambda_j-D_{nj})^+\\
&\leq1+\sqrt{T\lambda_j(1-\lambda_j)}.
\end{split}
\]
Multiply by \(r_j+hw_n\leq r_M+h(L+1)\), sum over \(n,j\), and divide by
\(N\). Combining with Jensen's upper bound proves the theorem. The constant
is independent of \(S\), so taking suprema over \(S\) is valid when the
initial-state map in \cref{sec:model} is used for every target.
\hfill\Halmos
\endproof

\begin{proposition}[The fixed-target square-root order is unavoidable]
\label{prop:full-information-sqrt-lower}
There is a two-cycle, \(L=1\), base-stock instance and a fixed
target \(S(T)\) for which every nonanticipating fulfillment policy has
expected profit \(\Omega(\sqrt T)\) below the full-information optimum.
Thus the fixed-\((N,S)\) rate in
\cref{thm:base-quota-full-information} is sharp; the proposition concerns a
common fixed target rather than separately optimized targets.
\end{proposition}

\proof{Proof of \cref{prop:full-information-sqrt-lower}.}
Take \(N=2\), \(L=1\),
\[
(\lambda_0,\lambda_1,\lambda_2)=(1/4,1/2,1/4),\quad
(r_1,r_2,h)=(1,3,1/2),
\]
and, for large \(T\), let \(c=\lfloor T/4\rfloor\), \(S=2c\), with initial
state \((c,c)\). Put
\[
H_n=D_{n,2},\qquad U=(c-H_1)^+,\qquad V=(H_2-c)^+.
\]
With probability \(1-e^{-\Omega(T)}\), both cycles contain at least
\(c+3b\sqrt T\) low-value requests. On this event and when
\(U\le2b\sqrt T\), retaining \(K\le U\) low-value units after cycle 1 changes
profit by at most
\[
g(K,V)=2\min\{K,V\}-\tfrac12K,
\]
with equality when cycle-2 high-value demand has priority. Full information
therefore chooses \(K^{\mathrm{FI}}=\min\{U,V\}\).

Enlarge the nonanticipating class by revealing all cycle-1 demand before it
chooses \(K\). We may also require it to accept all available high-value
cycle-1 demand, since this can displace only a customer of no greater reward
and weakly lowers holding cost. The resulting \(K\), including its
randomization, is independent of \(V\).

Choose fixed small \(b>0\). By the central limit theorem,
\(\Pr\{b\sqrt T\le U\le2b\sqrt T\}\) is bounded below, as are independently
\(\Pr\{V=0\}\) and \(\Pr\{V\ge b\sqrt T\}\). Conditional on such a value of
\(U\), if \(K\ge b\sqrt T/2\), then \(V=0\) causes regret at least
\(hb\sqrt T/2\). If \(K<b\sqrt T/2\), then \(V\ge b\sqrt T\) causes regret at
least \((r_2-r_1-h)b\sqrt T/2\). Each branch has probability bounded away
from zero. Intersecting with the low-value-demand event changes this only by
\(e^{-\Omega(T)}\). Thus even the enlarged policy class has expected regret
\(\Omega(\sqrt T)\), and so does every nonanticipating sequential policy.
\hfill\Halmos
\endproof

\section{Proofs in \cref{sec:extension}}

\begin{proposition}[Bayes Selector under additive rewards]
\label{prop:mr-additive-fbs-validation}
Suppose \(r_{j\ell}=a_j+b_\ell>0\), and relabel resources so that
\(b_1\geq\cdots\geq b_d\). Put
\(
\widetilde a_j=a_j+b_d
\)
and \(\widetilde b_\ell=b_\ell-b_d\). Merge types tied in
\(\widetilde a_j\), using the accept-all rule if one reward class remains.
Otherwise, run the scalar Bayes Selector defined in \cref{sec:model} with aggregate inventory
\(\sum_\ell I_\ell\) and rewards \(\widetilde a_j\), and assign accepted
requests in the common resource order. This implementation satisfies
\cref{ass:mr-common-order}. Its full and scalar conditional one-cycle regrets
satisfy
\begin{equation}\label{eq:separable-fbs-transfer}
\operatorname{Reg}_{r}
\leq
\left(1+\frac{\widetilde b_1}{\min_j\widetilde a_j}\right)
\operatorname{Reg}_{\widetilde a}.
\end{equation}
Consequently, the scalar Bayes Selector regret bound transfers to this
multi-resource implementation. In particular, it verifies the vector-uniform
\cref{ass:online-algorithm} with \(\alpha=0\).
\end{proposition}

\proof{Proof of \cref{prop:mr-additive-fbs-validation}.}
The ordering of \(b_\ell\) gives the common resource order, and the stated
assignment rule gives nested use. Fix an initial vector \(\mathbf I\), an
arrival path, and a realization of the online algorithm's internal
randomization. For \(0\leq K\leq\sum_\ell I_\ell\), define
\[
g_\ell(K;\mathbf I)=
\min\left\{I_\ell,
\left(K-\sum_{m<\ell}I_m\right)^+\right\},
\qquad
B_{\mathbf I}(K)=\sum_{\ell=1}^d\widetilde b_\ell
g_\ell(K;\mathbf I).
\]
Under nested use, \(B_{\mathbf I}(K)\) is the resource bonus obtained when
\(K\) requests are accepted. All rewards are positive, so the full offline
allocation and the scalar offline problem accept the same number of requests.
For this number, both select the requests with the largest scalar rewards
\(\widetilde a_j\). Thus their accepted type sets agree.

Let \(\Delta_r\) and \(\Delta_{\widetilde a}\) be the pathwise full and
scalar reward losses. Then
\[
\Delta_r=\Delta_{\widetilde a}
+B_{\mathbf I}(K^{\mathrm{off}})-B_{\mathbf I}(K^{\mathrm{on}}).
\]
The function \(B_{\mathbf I}\) is nondecreasing and
\(\widetilde b_1\)-Lipschitz. The pathwise augmentation argument also gives
\[
0\leq B_{\mathbf I}(K^{\mathrm{off}})-B_{\mathbf I}(K^{\mathrm{on}})
\leq\widetilde b_1(K^{\mathrm{off}}-K^{\mathrm{on}}),
\qquad
K^{\mathrm{off}}-K^{\mathrm{on}}
\leq\frac{\Delta_{\widetilde a}}{\min_j\widetilde a_j}.
\]
Consequently,
\[
\Delta_r\leq
\left(1+\frac{\widetilde b_1}{\min_j\widetilde a_j}\right)
\Delta_{\widetilde a}.
\]
Taking the conditional expectation gives
\eqref{eq:separable-fbs-transfer}. The scalar Bayes Selector bound then yields the
last assertion.
\hfill\Halmos
\endproof

\begin{lemma}\label{lem:mr-total-count-gap}
Under \cref{ass:online-algorithm}, if the two algorithms start a cycle from
the same inventory vector, then
\[
\sum_{\ell=1}^d
\left(I_{T+1,\ell}^{n,\mathrm{on}}
-I_{T+1,\ell}^{n,\mathrm{off}}\right)\geq0,
\]
and its conditional expectation given the pre-cycle history is at most
\(C_{\mathrm m}T^\alpha\).
\end{lemma}

\proof{Proof of \cref{lem:mr-total-count-gap}.}
Fix the common initial vector \(\mathbf I\), a realized demand path, and any
realization of the online algorithm's internal randomization. Let
\(K^{\mathrm{on}}\) and \(K^{\mathrm{off}}\) be the numbers of accepted
requests, and let \(\widehat V^{\mathrm{alg}}(T,\mathbf I,\mathcal H)\)
denote total fulfilled reward, excluding holding cost. Complete resource
compatibility and positive rewards imply
\[
K^{\mathrm{off}}=\min\!\left\{\sum_{\ell=1}^d I_\ell,D^n\right\}
\geq K^{\mathrm{on}}.
\]
Starting from the online allocation, assign
\(K^{\mathrm{off}}-K^{\mathrm{on}}\) rejected requests to unused resource
units. Such an assignment exists because every request can use every resource.
Each added assignment is worth at least \(r_{\min}\), while the myopic offline
allocation maximizes reward. Therefore, pathwise,
\[
\widehat V^{\mathrm{off}}(T,\mathbf I,\mathcal H)
-\widehat V^{\mathrm{on}}(T,\mathbf I,\mathcal H)
\geq r_{\min}(K^{\mathrm{off}}-K^{\mathrm{on}}).
\]
The difference in total terminal inventory is exactly
\(K^{\mathrm{off}}-K^{\mathrm{on}}\).  Conditional expectation and
\cref{ass:online-algorithm} now give the result with
\(C_{\mathrm m}=C_1/r_{\min}\).
\hfill\Halmos
\endproof

\begin{lemma} \label{lem:extend-inv-diff-resource}
Under \cref{ass:online-algorithm,ass:mr-common-order}, for the same
eligible initial inventory vector and every resource \(\ell\),
\begin{equation} \label{ineq:extend-inv-diff-resource}
I_{T+1,\ell}^{n,\mathrm{on}}-I_{T+1,\ell}^{n,\mathrm{off}}\geq0,\qquad
\mathbb E\!\left[
I_{T+1,\ell}^{n,\mathrm{on}}-I_{T+1,\ell}^{n,\mathrm{off}}
\,\middle|\,\mathcal F_{\mathrm{pre}}\right]
\leq C_{\mathrm m}T^\alpha .
\end{equation}
\end{lemma}

\proof{Proof of \cref{lem:extend-inv-diff-resource}.}
Fix the common eligible initial inventory vector \(\mathbf I\). For
\(0\leq k\leq\sum_\ell I_\ell\), let
\[
g_\ell(k;\mathbf I)
\triangleq
\min\!\left\{I_\ell,
\left(k-\sum_{m<\ell}I_m\right)^+\right\}
\]
be the number of resource-\(\ell\) units used when \(k\) units are consumed in
the common resource order. The function \(g_\ell(\cdot;\mathbf I)\) is
nondecreasing and one-Lipschitz. By the nested-use implementation,
\(f_\ell^{\mathrm{on}}=g_\ell(K^{\mathrm{on}};\mathbf I)\).

For the offline benchmark, if an optimal allocation uses a unit of resource
\(m\) while a unit of resource \(\ell<m\) is unused, reassign the same accepted
request from \(m\) to \(\ell\). Feasibility is preserved and reward does not decrease because
\(r_{j\ell}\geq r_{jm}\). Repeating this exchange and applying the tie-breaking
rule in \cref{ass:mr-common-order} gives
\(f_\ell^{\mathrm{off}}=g_\ell(K^{\mathrm{off}};\mathbf I)\). Hence
\[
f_\ell^{\mathrm{off}}=g_\ell(K^{\mathrm{off}};\mathbf I)
\geq g_\ell(K^{\mathrm{on}};\mathbf I)
=f_\ell^{\mathrm{on}}.
\]
Define
\(
\epsilon_\ell=f_\ell^{\mathrm{off}}-f_\ell^{\mathrm{on}}
\).
Since both allocations start from \(\mathbf I\), their terminal inventory
difference satisfies
\[
0\leq I_{T+1,\ell}^{n,\mathrm{on}}
-I_{T+1,\ell}^{n,\mathrm{off}}
=\epsilon_\ell
\leq K^{\mathrm{off}}-K^{\mathrm{on}}.
\]
Apply \cref{lem:mr-total-count-gap} conditionally on the pre-cycle history to
obtain \eqref{ineq:extend-inv-diff-resource}.
\hfill\Halmos
\endproof

\begin{lemma}
\label{lem:inv-diff-base-resource}
Under \cref{ass:online-algorithm,ass:mr-common-order}, suppose the
online and shadow systems start from the same complete feasible state
(on-hand inventories and all outstanding orders) and use
\(\mathbf S=(S_\ell)_{\ell\in[d]}\). For every cycle \(n\) and resource
\(\ell\),
\[
I_{1,\ell}^{n,\mathrm{on}}-I_{1,\ell}^{n,\mathrm{off}}\geq0,\qquad
\mathbb E[I_{1,\ell}^{n,\mathrm{on}}-I_{1,\ell}^{n,\mathrm{off}}]
\leq L_\ell C_{\mathrm m}T^\alpha,
\]
and
\[
I_{T+1,\ell}^{n,\mathrm{on}}-I_{T+1,\ell}^{n,\mathrm{off}}\geq0,\qquad
\mathbb E[I_{T+1,\ell}^{n,\mathrm{on}}-I_{T+1,\ell}^{n,\mathrm{off}}]
\leq(L_\ell+1)C_{\mathrm m}T^\alpha .
\]
\end{lemma}

\proof{Proof of \cref{lem:inv-diff-base-resource}.}
Fix a resource \(\ell\), suppress \(\mathbf S,\mathcal H\), and set
\[
\Delta I_t^n=I_{t,\ell}^{n,\mathrm{on}}-I_{t,\ell}^{n,\mathrm{off}},
\quad
\Delta q^n=q_\ell^{n,\mathrm{on}}-q_\ell^{n,\mathrm{off}},
\quad
P_n=\Delta I_1^n+\sum_{k=1}^{L_\ell-1}\Delta q^{n+k}.
\]
In cycle \(n\), put
\(
\bar{\mathbf I}^{n}=\mathbf I_1^{n,\mathrm{off}}
\)
and let \(f_\ell^{n,a}(\bar{\mathbf I}^{n})\) denote the
resource-\(\ell\) consumption of allocation
\(a\in\{\mathrm{on},\mathrm{off}\}\) from this vector. Define
\[
\epsilon_\ell^n
=f_\ell^{n,\mathrm{off}}(\bar{\mathbf I}^{n})
-f_\ell^{n,\mathrm{on}}(\bar{\mathbf I}^{n}).
\]
The two virtual allocations have the same eligible initial vector. Hence
\cref{lem:extend-inv-diff-resource} gives
\(
\epsilon_\ell^n\geq0
\)
and
\(
\mathbb E[\epsilon_\ell^n\mid\mathcal F_{\mathrm{pre}}]
\leq C_{\mathrm m}T^\alpha
\).
At the cycle-\(n\) step of the induction below,
\(\Delta I_1^n\geq0\) is the induction hypothesis, with equality in the
initial state. Thus every virtual use of resource \(\ell\) can be implemented
with a physical unit of that resource.
The physical online implementation consumes the same number of
resource-\(\ell\) units as its virtual allocation, while the excess physical
inventory is protected. Therefore,
\[
\Delta I_{T+1}^n=\Delta I_1^n+\epsilon_\ell^n.
\]
The resource-\(\ell\) base-stock equation gives
\(\Delta q^{n+L_\ell}=-P_n\). Starting from identical pipelines, the induction
in the proof of \cref{lem:inv-diff-base}, with
\((L,C_2,\epsilon^n)\) replaced by
\((L_\ell,C_{\mathrm m},\epsilon_\ell^n)\), establishes
\[
\Delta I_1^n\geq0,\quad
\mathbb E[\Delta I_1^n]\leq L_\ell C_{\mathrm m}T^\alpha,\quad
P_n\geq0,\quad
\mathbb E[P_n]\leq C_{\mathrm m}T^\alpha .
\]
This recursion also preserves feasibility. For \(n\geq2\),
\(P_n=\epsilon_\ell^{n-1}\) is no larger than the amount of resource
\(\ell\) used by the offline allocation in cycle \(n-1\), which equals
\(q_\ell^{n+L_\ell,\mathrm{off}}\). Therefore
\(q_\ell^{n+L_\ell,\mathrm{on}}
=q_\ell^{n+L_\ell,\mathrm{off}}-P_n\geq0\); the first-cycle order is feasible
by the common initial state.
The displayed terminal-inventory identity then yields the end-of-cycle bounds.
Since the argument holds for every \(\ell\), the proof is complete.
\hfill\Halmos
\endproof

\proof{Proof of \cref{thm:extend-reg-total-base}.}
Fix \(\mathbf S\). By construction, the componentwise shadow-capped policy
earns the reward of the single-cycle online algorithm initialized with the
shadow offline inventory vector. Thus \cref{ass:online-algorithm} bounds the
expected reward gap in every cycle by \(C_1T^\alpha\).
\cref{lem:inv-diff-base-resource} bounds the expected excess holding cost by
\[
\sum_{\ell=1}^dh_\ell(L_\ell+1)C_{\mathrm m}T^\alpha .
\]
Averaging over cycles gives the claimed fixed-\(\mathbf S\) bound. Under the
componentwise initialization above, the two finite-horizon optimizers exist.
Indeed, if \(S_\Sigma=\sum_\ell S_\ell\to\infty\), the total cycle-1 ending
inventory is at least
\(\bigl(\sum_\ell\lfloor S_\ell/(L_\ell+1)\rfloor-T\bigr)^+\). Moreover, with
\(P_{\max}=\max_\ell(L_\ell+1)\),
\[
\sum_{\ell=1}^d\left\lfloor\frac{S_\ell}{L_\ell+1}\right\rfloor
\geq \frac{S_\Sigma}{P_{\max}}-d.
\]
Since
\(\min_\ell h_\ell>0\), the first-cycle holding charge diverges while reward
is bounded by \(r_{\max}T\).
Finally, optimality of \(\mathbf S^{\mathrm{shadow}}\) implies
\[
\pi_{\mathbf b}^{\mathrm{off}}(T,\mathbf S^{\mathrm{off}},N)
-\pi_{\mathbf b}^{\mathrm{shadow}}(T,\mathbf S^{\mathrm{shadow}},N)
\leq
\pi_{\mathbf b}^{\mathrm{off}}(T,\mathbf S^{\mathrm{off}},N)
-\pi_{\mathbf b}^{\mathrm{shadow}}(T,\mathbf S^{\mathrm{off}},N),
\]
which completes the proof.
\hfill\Halmos
\endproof

\proof{Proof of \cref{prop:mr-uncapped-base-additive}.}
Let \(P=L+1\), \(S_\Sigma=\sum_\ell S_\ell\), and
\(U_T=T\sum_j\lambda_jr_{j1}\). The candidate
\(S_1=P\lceil\mu T\rceil\), \(S_\ell=0\) for \(\ell>1\), is a scalar
work-conserving base-stock system and has profit at least
\(U_T-A\sqrt T\), uniformly in \(N\), by
\eqref{eq:finite-balanced-candidate}. For arbitrary \(\mathbf S\), the
componentwise cooldown bound gives
\(
\sum_{n,\ell}F_{\ell,n}\le NS_\Sigma/P+Ld.
\)
Since every missed request loses at least \(r_{\min}\), comparison with the
candidate shows that every offline optimizer satisfies
\begin{equation}\label{eq:mr-base-off-optimizer-total-band}
S_\Sigma^{\mathrm{off}}(T,N)\ge P\mu T-K_S\sqrt T,
\end{equation}
with \(K_S\) independent of \(N,T\). Existence follows from bounded reward
and first-cycle holding-cost coercivity.

Fix a target in this band and couple the two systems from the same complete
state. Their aggregate states are scalar base-stock systems with target
\(S_\Sigma\). At the current online resource vector, construct the same-start
offline comparator and set
\(
\epsilon_{d,n}=\min\{D^n,\sum_\ell I_{1,\ell}^{n,\mathrm{on}}\}
-K_n^{\mathrm{on}}.
\)
Then \cref{lem:mr-total-count-gap} gives
\(
\mathbb E[\epsilon_{d,n}\mid\mathcal F_{n-1}]
\le C_{\mathrm m}T^\alpha
\), where \(C_{\mathrm m}=C_1/r_{\min}\).
Write \(w_j=a_j+b_d\) and \(\widetilde b_\ell=b_\ell-b_d\). The scalar
offline allocation selects the same highest-\(w_j\) customers as the full
offline allocation. For any common-start resource vector \(\mathbf I\), let
\(B_{\mathbf I}(K)\) be the resource bonus obtained by assigning \(K\)
accepted requests in common resource order. If \(\Delta_r\) and \(\Delta_w\)
are the pathwise full and scalar reward regrets, respectively, then
\[
\Delta_w
=\Delta_r-
\left[B_{\mathbf I}(K^{\mathrm{off}})
-B_{\mathbf I}(K^{\mathrm{on}})\right]
\leq\Delta_r.
\]
Indeed, \cref{lem:mr-total-count-gap} gives
\(K^{\mathrm{off}}\geq K^{\mathrm{on}}\), and
\(B_{\mathbf I}\) is nondecreasing. Hence the conditional scalar regret is
at most \(C_1T^\alpha\), uniformly over resource compositions. Together with
the conditional count-loss bound above, these are precisely the one-cycle
inputs used by the readiness-span argument in
\cref{prop:uncapped-direct-bound}. That argument remains valid when the
aggregate online transition depends on the current resource composition,
which is measurable with respect to the pre-cycle history. Combining the span
bound, the same-start conditional scalar-regret bound, and Lipschitz
continuity of the scalar offline value bounds the cumulative \(w_j\)-reward
difference by \(O(NT^\alpha)\).

It remains to control resource bonuses and holding costs. For prefix
\(k\), let \(F^a_{k,n}=\sum_{\ell\le k}F^a_{\ell,n}\), and define
\(
G_{k,n}=\sum_{i=1}^n(F^{\mathrm{off}}_{k,i}-F^{\mathrm{on}}_{k,i})
\), with \(G_{k,n}=0\) for \(n\le0\). Nested use and the common lead give
\[
F^{\mathrm{off}}_{k,n}=\min(D^n,I^{\mathrm{off}}_{k,n}),\quad
F^{\mathrm{on}}_{k,n}=\min(K_n^{\mathrm{on}},I^{\mathrm{on}}_{k,n}),\quad
I^{\mathrm{off}}_{k,n}-I^{\mathrm{on}}_{k,n}
=-G_{k,n-1}+G_{k,n-P}.
\]
Put \(\epsilon_{k,n}=\min(D^n,I^{\mathrm{on}}_{k,n})-
F^{\mathrm{on}}_{k,n}\). Then
\(0\le\epsilon_{k,n}\le\epsilon_{d,n}\). Monotonicity and the
one-Lipschitz property of \(x\mapsto\min(D^n,x)\) yield
\[
\min\{G_{k,n-1},G_{k,n-P}\}+\epsilon_{k,n}
\le G_{k,n}\le
\max\{G_{k,n-1},G_{k,n-P}\}+\epsilon_{k,n}.
\]
Induction gives \(0\le G_{k,n}\le E_{d,n}:=\sum_{i=1}^n\epsilon_{d,i}\).
The corresponding end-inventory identity telescopes to a cumulative prefix
holding difference in \([0,PE_{d,N}]\). Hence the cumulative offline-minus-
online resource bonus is at most \(\widetilde b_1E_{d,N}\), and the reverse
holding-cost difference is at most
\[
P\left[h_d+\sum_{k=1}^{d-1}(h_k-h_{k+1})_+\right]E_{d,N}.
\]
Since \(\mathbb EE_{d,N}\le NC_{\mathrm m}T^\alpha\), the fixed-target
profit gap is \(O(T^\alpha)\), uniformly in \(N\), throughout the band.
Apply this at \(\mathbf S^{\mathrm{off}}(T,N)\) and use online optimality.
\hfill\Halmos
\endproof

\begin{proposition}[Common order need not admit low-regret nested use]
\label{prop:mr-common-order-impossibility}
There is a strictly common-order, nonadditive two-type, two-resource instance
for which every sequential policy that is required to use the resources in
the common order has one-cycle expected regret \(\Omega(T)\) at a feasible
inventory vector. Consequently, common order plus mandatory nested use does
not by itself imply the vector-uniform
\cref{ass:online-algorithm} with any \(\alpha<1\).
\end{proposition}

\proof{Proof of \cref{prop:mr-common-order-impossibility}.}
Let \(H,L\), and no-arrival occur in a period with probabilities
\(p,q\), and \(1-p-q\), where \(p,q>0\) and
\(\mu=p+q<1\). Take two resources, initial inventory
\[
I_1=\lfloor pT\rfloor,\qquad I_2=\lfloor qT\rfloor,
\]
and rewards
\[
r_H=(1+\varepsilon+\Delta,1),\qquad
r_L=(1+\varepsilon,1),
\]
where \(\varepsilon,\Delta>0\). Both types strictly prefer resource 1,
but the resource bonuses differ, so the matrix is not additively
separable.

For an arbitrary sequential nested-use policy, let \(\tau\) be the time at
which resource 1 is exhausted, with \(\tau=T\) if it is not exhausted.
If \(K^{\mathrm{on}}\) is the number accepted and
\(A_H^{\mathrm{on}}\) is the number of type \(H\) customers assigned to
resource 1, then pathwise
\[
K^{\mathrm{on}}\leq I_1+D[\tau+1,T],
\qquad
A_H^{\mathrm{on}}\leq D_H[1,\tau].
\]
The bounded stopping time \(\tau\) may include the policy's internal
randomization. Optional sampling nevertheless gives
\[
\mathbb ED[\tau+1,T]=\mu(T-\mathbb E\tau),
\qquad
\mathbb ED_H[1,\tau]=p\,\mathbb E\tau.
\]
Consequently,
\[
\mathbb EV^{\mathrm{on}}
\leq I_1+\mu(T-\mathbb E\tau)
       +\varepsilon I_1+\Delta p\,\mathbb E\tau.
\]
The offline allocation serves \(\min\{D,I_1+I_2\}\) customers, uses
resource 1 \(\min\{D,I_1\}\) times, and assigns
\(\min\{D_H,I_1\}\) high types to it. Binomial mean-deviation bounds give
\[
\mathbb EV^{\mathrm{off}}
\geq\mu T+\varepsilon pT+\Delta pT-O(\sqrt T).
\]
Let \(\sigma_T\) be the time of the \(I_1\)-st nonempty arrival, truncated
at \(T\). Since \(\tau\geq\sigma_T\), the law of large numbers and bounded
convergence give \(\mathbb E\tau/T\geq p/\mu-o(1)\). Writing
\(x_T=\mathbb E\tau/T\), we obtain
\[
\operatorname{Reg}_T
\geq T\{\mu x_T-p+\Delta p(1-x_T)\}-O(\sqrt T).
\]
The expression in braces is a continuous linear function. Since
\(x_T\in[p/\mu-o(1),1]\), its value is at least
\[
\min\{q,\Delta pq/\mu\}-o(1)>0
\]
for all sufficiently large \(T\). Thus the regret is \(\Omega(T)\).
\hfill\Halmos
\endproof

\proof{Proof of \cref{prop:mr-universal-best-base}.}
Let \(U_T=T\sum_j\lambda_jr_{j\ell^\star}\). This is an upper bound
on every multi-resource profit. For each fixed \(T,N\), the finite-horizon
maximizers defined above exist. Under the componentwise initialization
above, the total first-cycle ending inventory is at least
\(\bigl(\sum_\ell\lfloor S_\ell/P_\ell\rfloor-T\bigr)^+\). Also,
\[
\sum_{\ell=1}^d\left\lfloor\frac{S_\ell}{P_\ell}\right\rfloor
\geq \frac{\sum_\ell S_\ell}{\max_\ell P_\ell}-d.
\]
Together with \(\min_\ell h_\ell>0\), this shows that first-cycle holding cost
diverges as \(\sum_\ell S_\ell\to\infty\), whereas the reward per cycle is
bounded by \(U_T\). Choose
\[
S_{\ell^\star}^0=(L_{\ell^\star}+1)\lceil\mu T\rceil,
\qquad S_\ell^0=0\quad(\ell\ne\ell^\star).
\]
The system reduces exactly to a scalar base-stock system. The
rolling-demand estimate in \eqref{eq:finite-balanced-candidate}, applied to
the rewards \(\{r_{j\ell^\star}\}_{j\in[M]}\), gives
\[
\pi_{\mathbf b}^{\mathrm{off}}(T,\mathbf S^0,N)
\geq U_T-A\sqrt T
\]
uniformly in \(N\). At the same target, the original multi-resource
online algorithm reduces to its scalar restriction, so
\cref{prop:uncapped-direct-bound} gives
\[
\pi_{\mathbf b}^{\mathrm{uncap}}(T,\mathbf S^0,N)
\geq U_T-A\sqrt T-CT^\alpha.
\]
Therefore,
\[
\pi_{\mathbf b}^{\mathrm{off}}
 (T,\mathbf S^{\mathrm{off}}(T,N),N)
-\pi_{\mathbf b}^{\mathrm{uncap}}
 (T,\mathbf S^{\mathrm{uncap}}(T,N),N)
\leq U_T-\pi_{\mathbf b}^{\mathrm{uncap}}(T,\mathbf S^0,N)
\leq A\sqrt T+CT^\alpha.
\]
\hfill\Halmos
\endproof

\begin{corollary}[Best-resource-only fixed-target transfer]
\label{cor:mr-best-resource-fixed-target}
Under the universal-best-resource and scalar-restriction hypotheses of
\cref{prop:mr-universal-best-base}, fix \(K_S<\infty\) and a base-stock
vector \(\mathbf S\) satisfying
\[
S_\ell=0\quad(\ell\ne\ell^\star),
\qquad
S_{\ell^\star}\geq
(L_{\ell^\star}+1)\mu T-K_S\sqrt T.
\]
If the myopic-offline and original uncapped online systems start from the
same complete feasible augmented post-order state for this target vector
(in particular, from the same componentwise state specified above), then
there are
constants \(C_{K_S},T_{K_S}<\infty\),
independent of \(\mathbf S,T,N\), such that, for all
\(T\geq T_{K_S}\) and every finite \(N\geq1\),
\[
\pi_{\mathbf b}^{\mathrm{off}}(T,\mathbf S,N)
-\pi_{\mathbf b}^{\mathrm{uncap}}(T,\mathbf S,N)
\leq C_{K_S}T^\alpha.
\]
Here both profit functions are evaluated at the common target \(\mathbf S\).
\end{corollary}

\proof{Proof of \cref{cor:mr-best-resource-fixed-target}.}
For every \(\ell\ne\ell^\star\), the zero target and the common feasible
initial state imply, inductively from the base-stock recursion, that the
on-hand inventory and every pipeline order of resource \(\ell\) remain zero
in both systems. Hence both systems use only resource \(\ell^\star\).
On this invariant state space, the myopic-offline system is exactly the scalar
myopic-offline system with rewards \(\{r_{j\ell^\star}\}_{j\in[M]}\), holding
cost \(h_{\ell^\star}\), lead time \(L_{\ell^\star}\), and target
\(S_{\ell^\star}\). By assumption, the original uncapped online algorithm
reduces to its scalar restriction on the same state space. The two scalar
systems have the same augmented post-order initial state, so
\cref{prop:uncapped-direct-bound} gives the claimed bound, uniformly over
finite \(N\).
\hfill\Halmos
\endproof

\proof{Proof of \cref{cor:mr-additive-arbitrary-leads}.}
For the resource \(\ell^\star\) chosen in the statement,
\[
r_{j\ell^\star}-r_{j\ell}=b_{\ell^\star}-b_\ell\geq0
\qquad(j\in[M],\ \ell\in[d]),
\]
so \(\ell^\star\) is a universal-best resource. The result follows directly
from \cref{prop:mr-universal-best-base}.
\hfill\Halmos
\endproof

\proof{Proof of \cref{prop:mr-dedicated-base}.}
The offline maximizer defined above exists by the first-cycle holding-cost
argument in the proof of \cref{prop:mr-universal-best-base}. The same argument
gives existence of a maximizing target for the dedicated policy.
For each resource with \(p_\ell>0\), arrivals of types in
\(\mathcal J_\ell\) form a scalar iid online-allocation instance; all other
periods are no-arrivals for that subsystem. Linearity of expectation allows
the groupwise bounds to be summed without cross-resource independence. Merge
tied-reward types within a group, and use
accept-all when only one reward class remains. Put
\[
U_T^\star=T\sum_j\lambda_j\max_\ell r_{j\ell}.
\]
Summing the scalar rolling-demand estimates over the finitely many resources
shows that the dedicated offline profit is at least
\(U_T^\star-A_{\mathrm d}\sqrt T\), uniformly in \(N\). The scalar
fixed-target uncapped Bayes Selector transfer in
\cref{prop:uncapped-direct-bound} gives an \(O(1)\) gap between offline and online profits for
each active group.
Take \(T_{\mathrm d}\) to be the largest of the finitely many scalar
thresholds in that proposition.
Thus
\[
\pi_{\mathbf b}^{\mathrm{ded\text{-}Bayes}}(T,\mathbf S^0,N)
\geq U_T^\star-A_{\mathrm d}\sqrt T-C_{\mathrm d}.
\]
No policy can earn more than \(U_T^\star\), so the global offline profit at
\(\mathbf S^{\mathrm{off}}(T,N)\) is at most \(U_T^\star\). Combining these
inequalities gives the displayed bound, and optimizing the dedicated target
vector can only increase its profit.
\hfill\Halmos
\endproof

For reference, the exponential sufficient condition used to verify
\cref{ass:mr-single-cycle-sm2} for the Bayes Selector can be stated as follows.

\begin{assumption}[Multi-resource single-cycle exponential count loss]
\label{ass:mr-single-cycle-et}
For the exponent \(0\leq\alpha<1\) in \cref{ass:online-algorithm}, let
\(K^{\mathrm{on}}(\mathbf I)\) be the online number fulfilled from any
initial vector \(\mathbf I\), and put
\[
\epsilon^n(\mathbf I)
=\min\{D^n,\textstyle\sum_\ell I_\ell\}
-K^{\mathrm{on}}(\mathbf I)\geq0.
\]
There are constants \(\eta>0\) and \(K_{\epsilon}<\infty\), independent
of \(T,n,\mathbf I\), and the pre-cycle history, such that
\[
\mathbb E\!\left[
\exp\!\left\{\eta\epsilon^n(\mathbf I)/T^\alpha\right\}
\middle|\mathcal F_{\mathrm{pre}}\right]\leq K_{\epsilon}.
\]
The arrivals in the new cycle are independent of
\(\mathcal F_{\mathrm{pre}}\).
\end{assumption}

\begin{lemma}\label{lem:mr-constant-online-total}
Under \cref{ass:online-algorithm}, suppose \(0\leq\alpha<1\). If
\(c\leq(1-\lambda_0-\delta_{\mathrm m})T\), then there is a constant
\(K_{\mathrm m}<\infty\), independent of \(n,T,\mathbf c\), such that,
for all sufficiently large \(T\),
\[
\mathbb E\!\left[
\sum_{\ell=1}^d I_{T+1,\ell}^{n,\mathrm{on}}(\mathbf c,\mathcal H)
\right]\leq K_{\mathrm m}T^\alpha .
\]
\end{lemma}

\proof{Proof of \cref{lem:mr-constant-online-total}.}
Let
\(X_n=\sum_{\ell=1}^dI_{T+1,\ell}^{n,\mathrm{on}}\) and \(X_0=0\).
At the online start vector, a same-start myopic offline allocation fulfills
\(\min\{X_{n-1}+c,D^n\}\) customers.  Let \(\epsilon^n\) be the
difference between this number and the number fulfilled online.  By
\cref{lem:mr-total-count-gap},
\[
X_n=(X_{n-1}+c-D^n)^++\epsilon^n,\qquad
0\leq\epsilon^n\leq T,\qquad
\mathbb E[\epsilon^n\mid\mathcal F_{n-1}]
\leq C_{\mathrm m}T^\alpha .
\]
The Lyapunov proof of \cref{lem:Inv-on-bound} applies verbatim after replacing
\(C_2\) by \(C_{\mathrm m}\) and \(\delta\) by
\(\delta_{\mathrm m}\). It yields constants independent of
\(n,T,\mathbf c\), with
\(\sup_n\mathbb E[X_n]\leq K_{\mathrm m}T^\alpha\).
\hfill\Halmos
\endproof

\proof{Proof of \cref{thm:extend-reg-total-constant}(a),(c).}
Let \(\widehat V^{\mathrm{alg}}(T,\mathbf I,\mathcal H)\) denote fulfilled
reward (excluding holding cost) from initial vector \(\mathbf I\), and write
\(\widehat V_n^{\mathrm{alg}}\) for its cycle-\(n\) realization.
Fix \(\mathbf c=\mathbf c^{\mathrm{off},\delta_{\mathrm m}}\). Applied to aggregate inventory,
the two-case induction in \cref{lem:constant-off-more} shows
\[
\sum_{\ell=1}^dI_{1,\ell}^{n,\mathrm{on}}
\geq\sum_{\ell=1}^dI_{1,\ell}^{n,\mathrm{off}}
\quad\text{and}\quad
\sum_{\ell=1}^dI_{T+1,\ell}^{n,\mathrm{on}}
\geq\sum_{\ell=1}^dI_{T+1,\ell}^{n,\mathrm{off}} .
\]
For any two inventory vectors \(\mathbf I,\mathbf J\), adding the positive
parts of \(\mathbf J-\mathbf I\) to \(\mathbf I\) gives the offline-value
Lipschitz bound
\[
\widehat V^{\mathrm{off}}(T,\mathbf J,\mathcal H)
-\widehat V^{\mathrm{off}}(T,\mathbf I,\mathcal H)
\leq r_{\max}\sum_{\ell=1}^d(J_\ell-I_\ell)^+ .
\]
Use this with \(\mathbf J=\mathbf I_1^{n,\mathrm{off}}\) and
\(\mathbf I=\mathbf I_1^{n,\mathrm{on}}\), and insert the offline value at
\(\mathbf I_1^{n,\mathrm{on}}\).  Conditional uniform regret then gives
\[
\mathbb E[\widehat V_n^{\mathrm{off}}-\widehat V_n^{\mathrm{on}}]
\leq C_1T^\alpha
+r_{\max}\mathbb E\!\left[\sum_{\ell=1}^d
\left(I_{1,\ell}^{n,\mathrm{off}}
-I_{1,\ell}^{n,\mathrm{on}}\right)^+\right].
\]
For \(n\geq2\), start inventory is the preceding end inventory plus the
common vector \(\mathbf c\).  That vector cancels componentwise, so
\[
\sum_{\ell=1}^d
\left(I_{1,\ell}^{n,\mathrm{off}}
-I_{1,\ell}^{n,\mathrm{on}}\right)^+
=
\sum_{\ell=1}^d
\left(I_{T+1,\ell}^{n-1,\mathrm{off}}
-I_{T+1,\ell}^{n-1,\mathrm{on}}\right)^+
\leq \sum_{\ell=1}^d I_{T+1,\ell}^{n-1,\mathrm{off}}
\leq X_{n-1}.
\]
The term is zero in the common initial state.  For holding costs,
\[
\sum_{\ell=1}^dh_\ell
\left(I_{T+1,\ell}^{n,\mathrm{on}}
      -I_{T+1,\ell}^{n,\mathrm{off}}\right)
\leq
h_{\max}\sum_{\ell=1}^d
\left(I_{T+1,\ell}^{n,\mathrm{on}}
      -I_{T+1,\ell}^{n,\mathrm{off}}\right)^+
\leq h_{\max}\sum_{\ell=1}^dI_{T+1,\ell}^{n,\mathrm{on}} .
\]
\cref{lem:mr-constant-online-total} bounds the expectation of the last
term, as well as the preceding reward-comparison term, by a constant times
\(K_{\mathrm m}T^\alpha\). Thus the fixed-\(\mathbf c\) gap is at most
\([C_1+(r_{\max}+h_{\max})K_{\mathrm m}]T^\alpha\). Averaging and then using
optimality of \(\mathbf c^{\mathrm{on},\delta_{\mathrm m}}\) over the same
restricted domain proves the first claim.  For unrestricted optimizers, the
zero vector earns zero profit and no policy earns more than the reward of all
nonempty arrivals.  Both optimized profits therefore lie in
\([0,r_{\max}(1-\lambda_0)T]\), which proves the final absolute bound.
\hfill\Halmos
\endproof

\proof{Proof of \cref{thm:extend-reg-total-constant}(b).}
Let \(\mu=1-\lambda_0\), \(b_T=T^\alpha\), and
\(U_T=T\sum_j\lambda_j\max_\ell r_{j\ell}\). At a common vector order with
total \(c\), the aggregate ending inventories satisfy
\[
X_n=(X_{n-1}+c-D^n)^++\epsilon_n,\qquad
Y_n=(Y_{n-1}+c-D^n)^+,\qquad X_0=Y_0=0,
\]
so \(X_n\ge Y_n\). The finite-horizon maximizers exist because reward is at
most \(U_T\) and first-cycle holding cost is coercive in \(c\).

First let \(\alpha<1/2\) and impose \cref{ass:mr-common-order}. For prefix
orders \(C_k=\sum_{\ell\le k}c_\ell\) and terminal prefix inventories
\(P^a_{k,n}=\sum_{\ell\le k}I_{T+1,\ell}^{n,a}\), nested use gives
\[
P^{\mathrm{off}}_{k,n}=(P^{\mathrm{off}}_{k,n-1}+C_k-D^n)^+,\quad
P^{\mathrm{on}}_{k,n}=(P^{\mathrm{on}}_{k,n-1}+C_k-D^n)^++\epsilon_{k,n},
\]
where \(0\le\epsilon_{k,n}\le\epsilon_n\). Replace \(\epsilon_{k,n}\) by
\(\epsilon_n\) to obtain a scalar upper comparison. Since
\cref{ass:mr-single-cycle-sm2} applies conditionally to \(\epsilon_n\),
\cref{lem:constant-sqrt-transfer-average} yields, whenever
\(c\le\mu T-\kappa\sqrt T\),
\begin{equation}\label{eq:mr-prefix-nearcritical-gap}
\frac1N\sum_{n=1}^N\mathbb E
(P^{\mathrm{on}}_{k,n}-P^{\mathrm{off}}_{k,n})
\le K_{\Delta,\kappa}b_T,\qquad k\in[d].
\end{equation}
If \(\Delta P_{k,n}=P^{\mathrm{on}}_{k,n}-P^{\mathrm{off}}_{k,n}\) and
\(\Delta P_{0,n}=0\), then
\[
\|\mathbf I_{T+1}^{n,\mathrm{on}}-
\mathbf I_{T+1}^{n,\mathrm{off}}\|_1
\le2\sum_{k=1}^d\Delta P_{k,n},
\]
and the same holds for the next start state. The offline reward value is
\(r_{\max}\)-Lipschitz in this distance. Thus the same-start regret and
\eqref{eq:mr-prefix-nearcritical-gap} give a fixed-order two-sided
\(O(b_T)\) profit difference, uniformly in \(N\).

Common order makes resource 1 best for all types. At
\(\mathbf c_a=(\lfloor\mu T-a\sqrt T\rfloor,0,\ldots,0)\), the scalar
candidate argument from \cref{thm:reg-total-constant}(d) gives both profit
deficits at most \(C_0\sqrt T\). For any vector order, total inventory
dominates the scalar offline Lindley chain \(Y_n(c)\); hence its profit deficit
is at least \((\min_{\ell\in[d]}h_\ell)N^{-1}\sum_n\mathbb EY_n(c)\). The Spitzer and central
limit argument used there yields fixed \(N_0,\kappa>0\) such that, for
\(N\ge N_0\), optimal orders satisfy \(c<\mu T-\kappa\sqrt T\). Applying the
fixed-order comparison at the two optimizers proves the claim for such \(N\).
For \(N<N_0\), the prefix recursions give
\(0\le\Delta P_{k,n}\le\sum_{i=1}^n\epsilon_i\), while
\(\mathbb E\epsilon_i\le\sqrt{K_{2,\mathrm m}}b_T\). The same reward and
holding comparison is therefore \(O(N_0b_T)\), uniformly over orders.

Now let \(1/2\le\alpha<1\). Assign type \(j\) to
\(\ell(j)\in\arg\max_\ell r_{j\ell}\), set
\(p_\ell=\sum_{j:\ell(j)=\ell}\lambda_j\), and for each nonempty group take
\(c_\ell=\lfloor p_\ell T-ab_T\rfloor\). Then
\[
\mathbb E(D_\ell^n-c_\ell)^+
\le\sqrt{Tp_\ell(1-p_\ell)}+ab_T+1=O(b_T).
\]
The offline same-start reward is \(U_T-O(b_T)\), and the online reward loses
at most another \(C_1b_T\). If there are \(m\) nonempty groups, total order
is at most \(\mu T-mab_T\). Applying
\eqref{eq:sm2-high-alpha-drift} to total inventory gives
\[
\mathbb E[X_n^2-X_{n-1}^2\mid\mathcal F_{n-1}]
\le2X_{n-1}\{-mab_T+\sqrt{K_{2,\mathrm m}}b_T+O(1)\}+Cb_T^2.
\]
Choose \(a\) so that the coefficient is at most \(-\gamma b_T\). Summation
gives \(N^{-1}\sum_n\mathbb EX_n=O(b_T)\), and likewise for \(Y_n\le X_n\).
Both candidate profits are at least \(U_T-O(b_T)\); since \(U_T\) is an upper
bound, the optimized profits differ by \(O(b_T)\).
\hfill\Halmos
\endproof

\end{document}